\PassOptionsToPackage{dvipsnames,cmyk}{xcolor}
\documentclass[final,hidelinks]{siamart251216}

\newcommand{\fullTitle}{Visibility in Polygonal Environments with Holes: \\ Finding Best Spots for Hiding and Surveillance}
\newcommand{\shortTitle}{Visibility in Polygonal Environments with Holes}

\newcommand{\shortAuthors}{N. Banzal, J. Cort\'es, S. Mart\'inez}

\usepackage[T1]{fontenc}                    
\usepackage[utf8]{inputenc}                 
\usepackage[english]{babel}                 
\usepackage{xifthen}                        

\usepackage{amssymb}                        
\usepackage{amsfonts}                       
\usepackage{mathrsfs}
\usepackage{thmtools}                       
\usepackage{bbm}                            

\usepackage{graphicx,subfigure}             
\usepackage{float}                          
\usepackage{tikz}                           
\usepackage{stmaryrd}                       

\usepackage{enumitem}                       

\usepackage{xfrac}                          

\usepackage[font=small]{caption}

\usepackage{algpseudocode}

\usepackage{setspace}
\usepackage{etoolbox}

\usepackage{mparhack}

\usetikzlibrary{calc}
\usetikzlibrary{arrows, arrows.meta}

\makeatletter
\renewcommand*{\thesection}{\arabic{section}}
\renewcommand*{\thesubsection}{\thesection.\Alph{subsection}}
\renewcommand*{\p@subsection}{}
\renewcommand*{\thesubsubsection}{\thesection.\thesubsection.\arabic{subsubsection}}
\renewcommand*{\p@subsubsection}{}
\makeatother

\def\tallqed{\smash{\scalebox{.5}[.75]{$\blacksquare$}}}
\renewcommand{\proofbox}{\ensuremath{\tallqed}}

\newsiamthm{prop}{Proposition}
\crefname{prop}{proposition}{propositions}
\Crefname{prop}{Proposition}{Propositions}

\newsiamremark{remark}{Remark}
\crefname{remark}{remark}{remarks}
\Crefname{remark}{Remark}{Remarks}

\newsiamremark{note}{Note}
\crefname{note}{note}{notes}
\Crefname{note}{Note}{Notes}

\makeatletter
  \newcommand*\rel@kern[1]{\kern#1\dimexpr\macc@kerna}
  \newcommand*\widebar[1]{%
    \begingroup
    \def\mathaccent##1##2{%
      \rel@kern{0.8}%
      \overline{\rel@kern{-0.8}\macc@nucleus\rel@kern{0.2}}%
      \rel@kern{-0.2}%
    }%
    \macc@depth\@ne
    \let\math@bgroup\@empty \let\math@egroup\macc@set@skewchar
    \mathsurround\z@ \frozen@everymath{\mathgroup\macc@group\relax}%
    \macc@set@skewchar\relax
    \let\mathaccentV\macc@nested@a
    \macc@nested@a\relax111{#1}%
    \endgroup
  }
\makeatother

\hypersetup{pdftitle={\shortTitle},
            pdfauthor={\shortAuthors}
            pdfsubject = {Control Systems},
            pdfkeywords = {visibility, analysis, surveillance, hiding, optimization},
            plainpages = false,
            colorlinks = true,
            linkcolor = NavyBlue,
            citecolor = WildStrawberry,
            } 

\definecolor{VibrantPurple}{RGB}{128, 64, 192}

\newcommand{\state}{z}
\newcommand{\heading}{\theta}

\newcommand{\fov}{\varphi}
\newcommand{\fovCone}[3]{\mathfrak{C}^{#1}_{#2}(#3)}
\newcommand{\fovConeInfinite}[2]{\mathfrak{C}^\infty_{#1}(#2)}
\newcommand{\radius}{R}

\protected\def\verythinspace{%
  \ifmmode
    \mskip0.5\thinmuskip
  \else
    \ifhmode
      \kern0.08334em
    \fi
  \fi
}

\newcommand{\symDiff}{\verythinspace \Delta \verythinspace}

\newcommand{\radiusRay}[2]{r(#1, #2)}
\newcommand{\radiusRotatedRay}[3]{r_{#3}(#1, #2)}

\newcommand{\setRay}[2]{\gamma(#1, #2)}
\newcommand{\setRotatedRay}[3]{\gamma_{#3}(#1, #2)}
\newcommand{\setRotatedRaySet}[4]{\gamma_{[#3, #4]}(#1, #2)}

\newcommand{\setCriticalPoints}{\mathcal{C}}
\newcommand{\setStationaryPoints}{\mathcal{S}}
\newcommand{\setNonsmoothPoints}{\mathcal{S}'}
\newcommand{\setInflectionPoints}{\mathcal{I}}

\newcommand{\myDomain}{D_1}
\newcommand{\advDomain}{D_2}

\newcommand{\conicalHull}[1]{\mathrm{cone}(#1)}
\newcommand{\convexHull}[1]{\mathrm{co}(#1)}

\newcommand{\lineSegment}[2]{\overline{#1 #2}}

\newcommand{\tangentCone}[2]{\mathcal{T}_{#1}(#2)}
\newcommand{\edgeUVA}[1]{\hat{a}_{#1}}
\newcommand{\edgeUVB}[1]{\hat{b}_{#1}}

\newcommand{\interior}[1]{\mathrm{int}(#1)}
\newcommand{\exterior}[1]{\mathrm{ext}(#1)}
\newcommand{\boundary}[1]{\mathrm{bdy}(#1)}
\newcommand{\closure}[1]{\mathrm{cl}(#1)}
\newcommand{\polygon}{P}

\newcommand{\vertices}[1]{\mathrm{Ve}(#1)}
\newcommand{\edges}[1]{\mathrm{Ed}(#1)}
\newcommand{\anchors}[1]{\mathrm{Ve}_a(#1)}
\newcommand{\reflexVertices}[1]{\mathrm{Ve}_r(#1)}
\newcommand{\vertex}{v}
\newcommand{\anchor}{\vertex_a}
\newcommand{\reflexVertex}{\vertex_r}

\newcommand{\visibilityPolygon}[1]{S(#1)}
\newcommand{\visibilityPolygonSymbol}{S}
\newcommand{\visibilityRegion}[1]{S(#1)}

\newcommand{\visibilityMetric}{V}
\newcommand{\visibilityArea}{\visibilityMetric_{\textrm{area}}}
\newcommand{\augmentedVisibilityMetric}[1]{\mathcal{V}_{#1}}
\newcommand{\augmentedVisibilityMetricMax}[1]{\mathcal{W}_{#1}}
\newcommand{\obstacles}{\mathcal{O}}
\newcommand{\obstacle}{O}
\newcommand{\freeSpace}{\mathcal{F}}
\newcommand{\reducedFreeSpace}{\freeSpace_r}
\newcommand{\inflectionSegment}{I}
\newcommand{\partition}{\mathcal{P}}
\newcommand{\setInflectionSegments}{\mathfrak{I}}

\newcommand{\componentTerrain}{\mathcal{U}}

\newcommand{\genericPoint}{x}
\newcommand{\genericPointQ}{y}
\newcommand{\genericPointR}{z}

\newcommand{\component}{\widebar{\componentTerrain}}

\newcommand{\genericPointSet}{\MakeUppercase{\genericPoint}}
\newcommand{\anySet}{\Omega}
\newcommand{\region}{M}

\newcommand{\distanceCenters}{d_\genericPointQ}

\newcommand{\proj}[2]{\pi_{#1}(#2)}

\newcommand{\environment}{Q}

\newcommand{\rightDirectionalDerivative}[1]{D #1}
\newcommand{\rightMuDirectionalDerivative}[1]{D_\mu #1}
\newcommand{\direction}{\nu}

\newcommand*\xbar[1]{%
    \hbox{%
        \vbox{%
            \hrule height 0.5pt 
            \kern0.5ex
            \hbox{%
                \kern-0.1em
                \ensuremath{#1}%
                \kern-0.1em
            }%
        }%
    }%
} 

\newcommand{\foralln}{\forall \verythinspace}
\newcommand{\existsn}{\exists \verythinspace}

\newcommand{\Norcent}{%
  \hyperlink{alg:Norcent}{Norcent algorithm}%
}
\newcommand{\NorcentTO}{%
  Norcent algorithm%
}

\makeatletter
\newcommand*{\etc}{%
    \@ifnextchar{.}%
        {etc}%
        {etc.\@\xspace}%
}
\makeatother

\newcommand\oprocendsymbol{\hbox{$\bullet$}}
\newcommand\oprocend{\relax\ifmmode\else\unskip\hfill\fi\oprocendsymbol}

\newcommand{\normalization}{\hat{\normalfont\textsc{n}}}

\DeclareMathOperator{\arctanTwo}{arctan2}

\newcommand{\genericSet}{A}

\newcommand{\integernonnegative}{\ensuremath{\mathbb{Z}}_{\geqslant 0}}
\newcommand{\real}{\ensuremath{\mathbb{R}}}
\newcommand{\realpositive}{\ensuremath{\mathbb{R}}_{> 0}}
\newcommand{\realnonnegative}{\ensuremath{\mathbb{R}}_{\geqslant 0}}
\newcommand{\extendedreal}{\ensuremath{\widebar{\mathbb{R}}}}
\newcommand{\extendedrealpositive}{\ensuremath{\widebar{\mathbb{R}}}_{> 0}}
\newcommand{\extendedrealnonnegative}{\ensuremath{\widebar{\mathbb{R}}}_{\geqslant 0}}

\newcommand{\union}{\ensuremath{\operatorname{\cup}\,}}
\newcommand{\intersection}{\ensuremath{\operatorname{\cap}}}
\newcommand{\bigunion}{\bigcup}

\newcommand{\map}[3]{#1: #2 \rightarrow #3}

\newcommand{\oBall}[2]{\mathbb{B}_{#1}(#2)}

\newcommand{\argmin}{\ensuremath{\operatorname{argmin}}}

\newcommand{\gen}{\ensuremath{\operatorname{gen}}}

\newcommand{\solutions}{\ensuremath{\Gamma}}
\newcommand{\stepSize}{a}
\newcommand{\stepSizeB}{b}

\newcommand{\longthmtitle}[1]{\mbox{}\textup{{(#1).}}}

\allowdisplaybreaks

\headers{\shortTitle}{\shortAuthors}

\title{\fullTitle\thanks{Submitted to the editors on \today.
    \funding{This work was supported by Award ARL-W911NF-25-2-0042.
    }}}

\author{Neilabh Banzal\thanks{Department of Mechanical and Aerospace
Engineering and the Contextual Robotics Institute, University of
California San Diego, CA (\{\href{mailto:neilabh@ucsd.edu}{neilabh}, \href{mailto:cortes@ucsd.edu}{cortes}, \href{mailto:soniamd@ucsd.edu}{soniamd}\}@ucsd.edu)} \and Jorge Cort\'es\footnotemark[2] \and Sonia Mart\'inez\footnotemark[2]}

\ifpdf
\hypersetup{
  pdftitle={\shortTitle},
  pdfauthor={\shortAuthors}
}
\fi

\allowdisplaybreaks

\begin{document}

\maketitle


\begin{abstract}
  Visibility plays an important role for decision making in cluttered,
  uncertain environments.  This paper considers the problem of
  identifying optimal hiding spots for an agent against line-of-sight
  detection by an adversary whose location is unknown.  We consider
  environments modeled as polygons with holes. We develop a set of
  mathematical tools for reasoning about visibility as a function of
  position and rely on non-smooth analysis to formally characterize
  the regularity properties of various visibility-based metrics. These
  metrics are non-smooth and non-convex, so off-the-shelf optimization
  algorithms can only guarantee convergence to Clarke critical points.
  To address this, the proposed Normalized Descent algorithm leverages
  the structure of non-smooth points in visibility problems and
  introduces randomness to escape saddle points. Our technical
  analysis allows for the non-monotonic decrease in the visibility
  metric and strengthens the algorithm guarantees, ensuring
  convergence to local minima with high probability. Simulations on
  two hide-and-seek scenarios showcase the effectiveness of the
  proposed approach.
\end{abstract}



\section{Introduction}
\label{sec:introduction}


This paper is motivated by the deployment of autonomous agents in
cluttered, uncertain environments. In such scenarios, line-of-sight
visibility is critical to accomplish their goals.  Under partial and
incomplete knowledge, an agent cannot fully rely on prior knowledge of
the environment and must instead acquire fresh information from
sensors. In turn, the information available to sensors can be
obstructed by an obstacle, and therefore, autonomous agents have to
leverage the readings within their line-of-sight for navigation and
decision making.
In adversarial settings, agents could seek to avoid detection while
trying to achieve their objective or, alternatively, a surveillance
team might seek to effectively monitor an area, both of which rely on
line-of-sight information in the presence of obstacles. In multi-agent
settings, formations that ensure line-of-sight visibility between
neighbors ensure better communication between the agents.
Visibility also plays an important role in the information mismatch
between the opposing teams in multi-player games such as Among Us.
Motivated by scenarios in which adversaries are present but their
locations are unknown, this paper seeks to develop mathematical tools
for reasoning about visibility in obstacle-rich environments as a
function of the agent's position and apply them to identify optimal
locations for hiding and seeking in environments modeled as polygons
with holes.


\subsubsection*{Literature Review} 

Our work leverages notions and methods from the extensive body of work
on visibility~\cite{JOR:17}, a cornerstone of computational
geometry~\cite{DL-FP:84-tc,JOR:00}.  The \emph{art gallery problem} is
a classical problem, where the objective is to find the minimum number
of guards to cover a non-convex polygon~\cite{VC:75-jct,SF:78-jct}.
The key tool used for the analysis of the art gallery problem and its
variants is coloring from graph
theory~\cite{JOR:87,TS:92-ieee,JU:00-hcg}.
The focus of this paper is on the variant with a single guard proposed
in~\cite{SN-MT:94-is}. Even though sampling-based algorithms have been
developed~\cite{OC-AE-SH:04-dcg} to find the best vertex-, perimeter-
and point-guards numerically, the theoretical understanding of how
visibility changes as the single guard moves within the environment is
limited. The solution methodologies for other variants like the
floodlight problem~\cite{VE-JO-JU-DX:95-ipl}, a single guard with
limited field of view~\cite{CT:02-cg,CT:00-cg}, and optimal sensor
placement~\cite{YO-MS-RY:08-WiMob}, also rely on discretization of the
space to generate computational graphs, followed by application of
graph theory to solve the problems. This results in good numerical
solutions, though lacking interpretability.  This problem lays the
foundation for solving visibility-based persistent monitoring
problems, for which there has been renewed interest in recent
years~\cite{AG-JC-FB:06-tr,NT-JB:10-icra,ES-NM-VK-VI:11-icra,AM-JC:20-jirs,JC-AB-ZZ-PT:21-iros}.

We seek to find the best spots for hiding and surveillance in
polygonal environments. The problem is non-trivial because the
visibility polygon may change drastically as the observer moves, and
as such, metrics based on visibility are non-smooth and non-convex.
This makes the optimization problem to find best spots challenging.
The works~\cite{AG-JC-FB:06-sicon,HS-MZ:15-ajor} tackle these issues
using a combination of non-smooth analysis, computational geometry,
and regularity properties of continuous flows. 
However, \cite{HS-MZ:15-ajor} takes a qualitative approach, and does
not provide guarantees for convergence to a local minima.  On the
other hand, \cite{AG-JC-FB:06-sicon} analyzes only the area of the
visibility polygon, in simple polygons without holes, and assuming
omnidirectional, infinite range visibility. The maximization algorithm
is designed in continuous time, and results in jittering when
discretizations are implemented.

We employ gradient descent algorithms for optimizing visibility-based
metrics. Standard gradient methods slow down progress towards critical
points in regions where the magnitude of the gradient becomes small.
Instead, for optimizing visibility-based metrics, we turn to
normalized gradient algorithms, that use only the gradient direction
information at each step.  For non-smooth objectives, as is the case
for visibility-based metrics, the convex case has been
well-developed~\cite{RC-CL:93-mp,BP:77,AR:06}.
For non-convex scenarios, in continuous time,~\cite{JC:06-auto}
establishes global convergence to Clarke stationary points under
normalized generalized gradient flows, while~\cite{RM-BS-SK:19-tac}
demonstrates the non-convergence of the said flows to saddle points
under a strong thrice differentiable assumption in a neighborhood of
the saddle points. In discrete
time,~\cite{LB:72-kibernetika,NS-KK-AR:85} show local convergence to
isolated local minima, the gradient sampling
method~\cite{JB-AL-MO:05-siamjo,KK:07-siamjo} achieves almost sure
global convergence to Clarke stationary points,
while~\cite{VM-AG-VN:87} globally converges to said stationary points
under mild conditions. We notice that the above algorithms only
converge to stationary points, but not to local minimum points.
Moreover, for the scenarios considered in this paper, due to the
geometry of the problem, one could have a continuum of local minima,
and the metrics could be non-differentiable at the saddle points,
making the construction of a gradient-based algorithm in this setting
non-trivial.


\subsubsection*{Statement of Contributions}

We consider an agent deployed over an obstacle-rich environment,
modeled as a polygon with holes. The agent's objective is to minimize
the risk of line-of-sight detection from an adversary whose location
is unknown. Our first set of contributions pertain to the
characterization of mathematical notions for visibility. Specifically,
we study the \emph{critical points} of the environment, which
correspond to the set of points for the agent where visibility
properties change discontinuously, and are naturally associated to
anchors, which obstruct a part of the environment from the observer.
We establish their connection with the notion of \emph{inflection
segment}, leading to a partition of the free space with each of its
connected components having a constant set of anchors.  Building on
these results, our second set of contributions concern the study of
various non-smooth visibility-based metrics, including under limited
range/field of view, that encode the chances of the agent being
detected in the environment.  To do so, we extend the classical
notions of local Lipschitzness and directional derivatives to their
set-valued counterparts, which we term $\mu$-local Lipschitzness and
$\mu$-directional derivatives.  These notions provide us a way to
formally characterize of the regularity properties of the visibility
metrics, including their directional derivatives and generalized
gradients.  Finally, we introduce the \textbf{Nor}malized
Des\textbf{cent} (Norcent) algorithm to optimize the visibility
metrics.  Our design strategy first transforms the constrained
optimization problem into an unconstrained one and then synthesizes a
normalized generalized gradient-based descent.  Given the non-smooth
and non-convex nature of the objective functions, this algorithm also
incorporates randomization, allowing for a non-monotonic decrease of
the cost by exploiting the problem structure.  We establish the almost
surely convergence to local minimum points of the metrics while
avoiding undesired stationary points, including non-differentiable
and/or non-isolated ones.  Simulations on two scenarios (one where
expert mountaineers try to hide from a drone with a camera and another
where a robot looks for other robots in a multi-room setting)
illustrate our results.


\section{Preliminaries}\label{sec:preliminaries}

In this section we introduce our notational conventions and review key
concepts from non-smooth analysis, set-valued maps, and computational
geometry.


\subsection{Notation}\label{sub:notation}

We let $\real^n$ (resp. $\extendedreal^n$) denote the $n$-dimensional
real (resp. extended real) vector space. The positive and non-negative
orthants of these spaces are denoted as \smash{$\realpositive^n$
($\extendedrealpositive^n$) and $\realnonnegative^n$
($\extendedrealnonnegative^n$)}, respectively. 
We represent the Euclidean norm of a column vector $v \in \real^n$ as
$\|v\|$, its transpose as $v^\top$\!, its $k$\textsuperscript{th}
component as $v_k$, and its normalized version as $\hat{v} := v /
\|v\|$.  We also employ the alternate form $\normalization(v)$ when
the use of the hat notation becomes cumbersome. 
The open $r$-ball in $\real^n$ with center at $\genericPoint$ is
denoted as $\oBall{r}{\genericPoint} = \left\{ \genericPointQ \in
\real^n \mid \|\genericPointQ - \genericPoint\| < r \right\}$.
The cardinality of a set $\genericSet$ is represented by
$|\genericSet|$. For sets $\genericSet_1$ and $\genericSet_2$, their
intersection and union are expressed as $\genericSet_1 \cap
\genericSet_2$, and $\genericSet_1 \cup \genericSet_2$, respectively,
and the difference of $\genericSet_2$ from $\genericSet_1$ as
$\genericSet_1 \setminus \genericSet_2$. The symmetric difference of
$\genericSet_1$ and $\genericSet_2$ is denoted as $\genericSet_1
\symDiff \genericSet_2 := (\genericSet_1 \setminus \genericSet_2) \cup
(\genericSet_2 \setminus \genericSet_1)$.
Given a set $ \genericSet \subset \real^n$, let
$\interior{\genericSet}$, $\boundary{\genericSet}$,
$\conicalHull{\genericSet}$, and $\convexHull{\genericSet}$ denote its
interior, boundary, conical hull, and convex hull, respectively.  If
$\real^n \supseteq \Omega \supset \genericSet$, the exterior of
$\genericSet$ with respect to $\Omega$ is defined as
$\exterior{\genericSet} := \Omega \setminus \genericSet$.
Given a measure $\mu$ on $\Omega$ and a measurable set $\genericSet
\subset \Omega$, we define $\| \genericSet \|_\mu := \mu(\genericSet)$
(with a slight abuse of notation, as this is only a semi-norm, not a
norm). A set-valued map $\mathcal{F} : \Omega \rightrightarrows
\real^m$ assigns a subset of $\real^m$ to each point in~$\Omega$.
Given a sequence $\{\genericPoint_k\}$, we use $\genericPoint_k \to
\bar{\genericPoint}$ and $\genericPoint_k \downarrow
\bar{\genericPoint}$ to denote convergence and monotonic convergence
to $\bar{\genericPoint}$.


\subsection{Non-Smooth Analysis}\label{sub:non_smooth}

We consider functions that may not be differentiable everywhere
following~\cite{FC:13-gtm,FHC:83,RTR-RJBW:98}. For a set
$\genericPointSet \subset \real^n$, a function $f : \genericPointSet
\rightarrow \real^m$ is \emph{locally Lipschitz} at $\genericPoint \in
\genericPointSet$ if there exist $L_\genericPoint, \epsilon > 0$ such
that
\begin{align*}
  \left\| f(\genericPointQ) - f(\genericPointQ') \right\| \leqslant
  L_\genericPoint \left\| \genericPointQ - \genericPointQ' \right\|, \quad
  \foralln \genericPointQ, \genericPointQ' \in
  \oBall{\epsilon}{\genericPoint} \intersection \genericPointSet. 
\end{align*}
We extend this notion to set-valued maps. Let $\mu$ be the Lebesgue
measure on $\real^m$. For a measure space $(\genericPointSet,
\mathcal{A}, \mu)$, a set-valued map $F: \genericPointSet
\rightrightarrows \real^m$ is \emph{$\mu$-locally Lipschitz} at
$\genericPoint \in \genericPointSet$ if there exist $L_\genericPoint,
\epsilon > 0$ such that
\begin{align*}
  \| F(\genericPointQ) \symDiff
  F(\genericPointQ') \|_\mu \leqslant
  L_\genericPoint \left\| \genericPointQ - \genericPointQ' \right\|, \quad
  \foralln \genericPointQ, \genericPointQ' \in
  \oBall{\epsilon}{\genericPoint} \intersection \genericPointSet,
\end{align*}
The notion of distance on the left-hand side, $d(A_1, A_2) := \| A_1
\symDiff A_2 \|_\mu, A_1, A_2 \in \mathcal{A}$, is known as the
\emph{symmetric difference metric}.
Let $\sim$ be a relation on $\mathcal{A}$ such that $A_1 \sim A_2$ if
$d(A_1, A_2) = 0$. This metric induces a \textit{bona-fide} metric
over the quotient space $\mathcal{A} / \mu$, where an equivalence
class is defined by sets of the same measure, cf.~\cite{VB:07,
MD-ED:16}.
$\mu$-local Lipschitzness is useful for establishing regularity of
functions expressed as compositions of a set-valued map with a
measure.

\begin{lemma}\longthmtitle{Relation between
  \texorpdfstring{$\mu$-}{μ-}local Lipschitzness and local
  Lipschitzness}
  \label{lem:mu_LL}
  Given a measure space $(\anySet, \mathcal{A}, \mu)$ and a set-valued
  map $F : \genericPointSet \rightrightarrows \anySet$, if $F$ is
  $\mu$-locally Lipschitz at $\genericPoint \in \genericPointSet$,
  then $\mu \circ F$ is locally Lipschitz at $\genericPoint$.
\end{lemma}
\begin{proof}
  For a given $\genericPoint \in \genericPointSet$, for all
  $\genericPointQ, \genericPointQ' \in \oBall{\epsilon}{\genericPoint}
  \intersection \genericPointSet$, we have 
  \begin{multline*}
    | \|F(\genericPointQ)\|_\mu - \|F(\genericPointQ')\|_\mu | =
    \left| \| F(\genericPointQ) \setminus 
    F(\genericPointQ') \|_\mu - \| F(\genericPointQ') \setminus
    F(\genericPointQ) \|_\mu \right| 
    \\
    \leqslant \| F(\genericPointQ) \setminus F(\genericPointQ') \|_\mu
    + \| F(\genericPointQ') \setminus F(\genericPointQ) \|_\mu 
    = \| S(\genericPointQ) \symDiff S(\genericPointQ') \|_\mu
    \leqslant L_{\genericPoint} \left\| \genericPointQ -
      \genericPointQ' \right\|, 
  \end{multline*}
  where $L_{\genericPoint}, \epsilon > 0$ are constants from
  $\mu$-locally Lipschitzness of $F$ at $x \in \genericPointSet$.
\end{proof}

The intersection of two set-valued maps that are $\mu$-locally
Lipschitz is also $\mu$-locally Lipschitz.

\begin{lemma}[Intersection of \texorpdfstring{$\mu$}{μ}-locally
  Lipschitz set-valued maps]
  \label{lem:intersection}
  Given a measure space $(\anySet, \mathcal{A}, \mu)$, let $S_1, S_2 :
  \genericPointSet \rightrightarrows \anySet$ be $\mu$-locally
  Lipschitz at $\genericPoint \in \genericPointSet$, with Lipschitz
  constants $L_1$ and $L_2$, resp. Then, $S_1 \intersection S_2$ is
  $\mu$-locally Lipschitz at $\genericPoint \in \genericPointSet$,
  with Lipschitz constant $L_1 + L_2$.
\end{lemma}
\begin{proof}
  The result follows directly from the fact that, for arbitrary sets
  $A_1, A_2, C_1, \allowbreak C_2$, one has $(A_1 \intersection C_1)
  \symDiff (A_2 \intersection C_2) \subseteq (A_1 \symDiff A_2) \union
  (C_1 \symDiff C_2)$.
\end{proof}

The \emph{directional derivative} of a function $f : \genericPointSet
\rightarrow \real^m$ at $\genericPoint \in \genericPointSet$ in the
direction $\nu \in \real^n$ is defined as 
$\rightDirectionalDerivative{f}(\genericPoint ; \nu) = \lim_{h
\rightarrow 0^+} (f(\genericPoint + h \nu) - f(\genericPoint))/{h}$
whenever the limit exists. The \emph{generalized gradient} is given by
\begin{align*}
  \partial f(\genericPoint) = \mathrm{co} \left\{ \lim_{i \rightarrow
  \infty} \nabla f(\genericPoint_i) \mid \genericPoint_i \rightarrow
  \genericPoint, \genericPoint_i \not \in E \cup E_f \right\}. 
\end{align*}
Here, $E$ is any zero measure set in the Lebesgue sense, $E_f$ is the
set of points where $f$ fails to be differentiable, and $\nabla f$ is
the gradient of~$f$. For $f$ locally Lipschitz, Rademacher's Theorem
implies that $E_f$ has Lebesgue measure zero. Furthermore, $\partial
f(x)$ is a non-empty, convex, and compact subset of $\real^{n \times
n}$.

We extend the notion of directional derivatives to compositions of
measures with set-valued maps. Let $(\anySet, \mathcal{A}, \mu)$ be a
measure space.  
The \emph{$\mu$-directional derivative} of a set-valued map $F :
\genericPointSet \rightrightarrows \anySet$ at $\genericPoint \in
\genericPointSet$ in the direction $\direction$ is given by
\begin{align}
  \label{eqn:mu_directional_derivative}
  \rightMuDirectionalDerivative{F}(\genericPoint; \direction) =
  \lim_{h \to 0^+} \frac{\|F(\genericPoint
  + h \direction) \Delta
  F(\genericPoint)\|_\mu}{h}, 
\end{align}
whenever the limit exists. Note that $|
\rightDirectionalDerivative{(\mu \circ F)}(\genericPoint; \direction)
| \leqslant \rightMuDirectionalDerivative{F}(\genericPoint;
\direction)$.

Finally, we state a useful result  characterizing local minima (resp.
maxima) via the directional derivative.

\begin{lemma}[Local extrema]
  \label{lem:local_extrema}
  Given $f: \real^n \rightarrow \real$ locally Lipschitz,
  $\genericPoint \in \real^n$ is a local minimum (resp. maximum) over
  $\real^n$ iff $\rightDirectionalDerivative{f}(\genericPoint;
  \hat{\nu}) \geqslant 0$ (resp.
  $\rightDirectionalDerivative{f}(\genericPoint; \hat{\nu}) \leqslant
  0$), for all $ \nu \in \real^n$ wherever defined, and is an isolated
  local minimum (resp. maximum) iff
  $\rightDirectionalDerivative{f}(\genericPoint; \hat{\nu}) > 0$
  (resp. $\rightDirectionalDerivative{f}(\genericPoint; \hat{\nu}) <
  0$), for all $ \nu \in \real^n$ wherever defined. \oprocend
\end{lemma}


\subsection{Visibility in Polygonal Environments}
\label{sub:computational_geometry}

Here, we formalize the notion of visibility in polygonal environments
with obstacles and introduce the concept of projected rays.

\subsubsection*{Geometry of an environment with obstacles}

Following~\cite{HC-JE-CX:15-siam}, we define a \emph{rigidly weakly
simple polygon} as a polygon $\polygon$ which can be made simple by an
arbitrarily small perturbation of its vertices.
Given such polygon, let $\vertices{\polygon} := \{ p^1, \dots
p^{n_\polygon} \}$ denote its $n_\polygon$ vertices. We consider
\emph{positively-oriented} and \emph{negatively-oriented} polygons,
depending on whether the vertices are listed in counter-clockwise or
clockwise order.  Its $n_\polygon$ edges are $\edges{\polygon} := \{
\lineSegment{p^1}{p^2}, \lineSegment{p^2}{p^3}, \dots,
\lineSegment{p^{n_\polygon}}{p^1} \}$ and we let
$\interior{\polygon}$, $\boundary{\polygon}$, and
$\exterior{\polygon}$ denote the interior, boundary, and exterior of
$\polygon$, respectively, where the exterior is taken wrt~$\real^2$.

A polygonal \emph{environment} for an \emph{agent} is a
negatively-oriented rigidly weakly simple polygon $\environment
\subset \real^2$ such that the agent is not allowed in
$\exterior{\environment}$.
Similarly, a polygonal \emph{obstacle} is a positively-oriented
rigidly weakly simple polygon $\obstacle \subset \environment$ such
that the agent is not allowed in $\interior{\obstacle}$. Let
$\obstacles = \{ \obstacle^1, \dots \obstacle^{n_o} \}$ denote the set
of $n_o$ \emph{obstacles} in the environment such that
$\interior{\obstacle^i} \cap \interior{\obstacle^j} = \varnothing$,
for all $\obstacle^i, \obstacle^j \in \obstacles, i \neq j$. We define
the \emph{free space} $\freeSpace := \environment \setminus
\interior{\obstacles}$, where $\interior{\obstacles} :=
\cup_{i=1}^{n_o} \interior{\obstacle_i}$. We denote the vertices and
edges of the free space by $\vertices{\freeSpace} =
\vertices{\environment} \union \bigunion_i \vertices{\obstacle^i}$ and
$\edges{\freeSpace} = \edges{\environment} \union \bigunion_i
\edges{\obstacle^i}$, respectively. With this formulation, the
\emph{free space} characterizes the set of all the positions the agent
is allowed in.

A point $\genericPointQ \in \freeSpace$ is \emph{visible} from a point
$\genericPoint \in \freeSpace$ if the line segment
$\lineSegment{\genericPoint}{\genericPointQ}$ is contained in
$\freeSpace$. The \emph{visibility region} is the set-valued mapping
$S: \freeSpace \rightrightarrows \freeSpace$ which maps a point
$\genericPoint \in \freeSpace$ to the set of all points in
$\freeSpace$ \emph{visible} from $\genericPoint$, that is,
\begin{align*}
  \visibilityRegion{\genericPoint} = \{ \genericPointQ \in \freeSpace
  \mid \lineSegment{\genericPoint}{\genericPointQ} \subset \freeSpace
  \}. 
\end{align*}
When the environment and obstacles are polygonal, the visibility
region $\visibilityRegion{\genericPoint}$ is a polygon for every
$\genericPoint \in \freeSpace$, which we term \emph{visibility
polygon} of $\genericPoint$. 

For a vertex $\vertex_i$, the counter-clockwise angle from
$\lineSegment{\vertex_i}{\vertex_{i-1}}$ to
$\lineSegment{\vertex_i}{\vertex_{i+1}}$ is the angle made by
$\vertex_i$ over the free space. This is consistent with the ordering
on the environment and the obstacle polygons. If this angle is greater
than $\pi$, then $\vertex_i \in \vertices{\freeSpace}$ is a
\emph{reflex vertex}. For a polygon $\polygon$,
$\reflexVertices{\polygon}$ denotes its reflex vertices. Then, the set
of all reflex vertices of $\freeSpace$ is denoted by
$\reflexVertices{\freeSpace} := \reflexVertices{\environment} \union
\bigunion_i \reflexVertices{\obstacle^i}$.
\Cref{fig:visibility_notions} illustrates these notions.

\begin{figure}[htb]
  \centering
  \subfigure[]{
    \label{fig:visibility_notions}
    \resizebox{0.46\linewidth}{!}{
      \centering
      \pgfdeclarelayer{background}
      \pgfdeclarelayer{foreground}
      \pgfsetlayers{background,main,foreground}
      \begin{tikzpicture}[scale=0.55]
        \tikzset{
          style 0/.style={inner sep=0pt, minimum size=3pt},
          style A/.style={shape=circle, fill=PineGreen, style 0},
          style B/.style={shape=circle, fill=WildStrawberry, style 0},
          style C/.style={shape=circle, fill=Blue!50!RoyalBlue, style 0},
          font={\fontsize{8pt}{9.6}\selectfont}
        }

        \begin{pgfonlayer}{foreground}
          \node[draw, style A] (q1) at (-1, -1) {};
          \node[draw, style A] (q2) at (11, -1) {};
          \node[draw, style A] (q3) at (11,  1) {};
          \node[draw, style A] (q4) at ( 9,  1) {};
          \node[draw, style A] (q5) at ( 9,  5) {};
          \node[draw, style A] (q6) at ( 1,  5) {};
          \node[draw, style A] (q7) at ( 1,  1) {};
          \node[draw, style A] (q8) at (-1,  1) {};

          \node[text=PineGreen] at ([shift={(225:0.7)}]q1) {$q^1$};
          \node[text=PineGreen] at ([shift={(315:0.5)}]q2) {$q^8$};
          \node[text=PineGreen] at ([shift={( 45:0.7)}]q3) {$q^7$};
          \node[text=PineGreen] at ([shift={( 45:0.7)}]q4) {$q^6$};
          \node[text=PineGreen] at ([shift={( 45:0.7)}]q5) {$q^5$};
          \node[text=PineGreen] at ([shift={(135:0.5)}]q6) {$q^4$};
          \node[text=PineGreen] at ([shift={(135:0.7)}]q7) {$q^3$};
          \node[text=PineGreen] at ([shift={(135:0.5)}]q8) {$q^2$};
        \end{pgfonlayer}

        \begin{pgfonlayer}{main}
          \tikzstyle{every path}=[draw, line width=1.0pt, color=black]
          \draw[fill = gray!20] (q1.center) -- (q2.center) -- (q3.center) -- (q4.center) -- (q5.center) -- (q6.center) -- (q7.center) -- (q8.center) -- cycle;
          \node at (8, 2) {$\freeSpace$};
        \end{pgfonlayer}

        \begin{pgfonlayer}{foreground}
          \node[draw, style B] (o1) at (3, 1) {};
          \node[draw, style B] (o2) at (7, 1) {};
          \node[draw, style B] (o3) at (7, 3) {};
          \node[draw, style B] (o4) at (3, 3) {};

          \node[text=WildStrawberry] at ([shift={( 45:0.6)}]o1) {$o^1$};
          \node[text=WildStrawberry] at ([shift={(135:0.6)}]o2) {$o^2$};
          \node[text=WildStrawberry] at ([shift={(225:0.6)}]o3) {$o^3$};
          \node[text=WildStrawberry] at ([shift={(315:0.6)}]o4) {$o^4$};
        \end{pgfonlayer}

        \begin{pgfonlayer}{main}
          \tikzstyle{every path}=[draw, line width=1.0pt, color=black]
          \draw[fill = white] (o1.center) -- (o2.center) -- (o3.center) -- (o4.center) -- cycle;
          \node at (5, 2) {$\obstacle$};
        \end{pgfonlayer}

        \begin{pgfonlayer}{foreground}
          \node[draw, style C] (p) at (2, 2) {};
          \node[text = Blue!50!RoyalBlue] at ([shift={(135:0.5)}]p) {$\genericPoint$};
        \end{pgfonlayer}

        \begin{pgfonlayer}{main}
          \tikzstyle{every path}=[draw, line width=0.5pt, color=RoyalBlue]
          \draw[fill = RoyalBlue!20] (1, 1) -- (-1,-1) -- (5, -1) -- (3, 1) -- (3, 3) -- (5, 5) -- (1, 5) -- cycle;
          \node[text = Blue!50!RoyalBlue] at (2, 0) {$\visibilityPolygon{\genericPoint}$};
        \end{pgfonlayer}

        \begin{pgfonlayer}{main}
          \draw[<-, Violet, line width=0.75pt] (7, 3)+(-0.5, 0) arc[start angle=180, end angle=-90, radius=0.5];
          \draw[<-, Violet, line width=0.75pt] (7, 1)+(0, 0.5) arc[start angle=90, end angle=-180, radius=0.5];
          \draw[<-, Violet, line width=0.75pt] (9, 1)+(0.5, 0) arc[start angle=360, end angle=90, radius=0.5];
          \draw[<-, Violet, line width=0.75pt] (3, 1)+(0.5, 0) arc[start angle=0, end angle=-270, radius=0.5];
          \draw[<-, Violet, line width=0.75pt] (3, 3)+(0, -0.5) arc[start angle=270, end angle=0, radius=0.5];
          \draw[<-, Violet, line width=0.75pt] (1, 1)+(0, 0.5) arc[start angle=90, end angle=-180, radius=0.5];

          \draw[<-, RedOrange, line width=0.75pt] (9, 5)+(0, -0.5) arc[start angle=270, end angle=180, radius=0.5];
          \draw[<-, RedOrange, line width=0.75pt] (11, 1)+(0, -0.5) arc[start angle=270, end angle=180, radius=0.5];
          \draw[<-, RedOrange, line width=0.75pt] (11, -1)+(-0.5, 0) arc[start angle=180, end angle=90, radius=0.5];
          \draw[<-, RedOrange, line width=0.75pt] (1, 5)+(0.5, 0) arc[start angle=0, end angle=-90, radius=0.5];
          \draw[<-, RedOrange, line width=0.75pt] (-1, 1)+(0.5, 0) arc[start angle=0, end angle=-90, radius=0.5];
          \draw[<-, RedOrange, line width=0.75pt] (-1, -1)+(0, 0.5) arc[start angle=90, end angle=0, radius=0.5];
        \end{pgfonlayer}

        \begin{pgfonlayer}{main}
          \tikzstyle{every path}=[draw, line width=1.0pt, color=black]
          \draw (q1.center) -- (q2.center) -- (q3.center) -- (q4.center) -- (q5.center) -- (q6.center) -- (q7.center) -- (q8.center) -- cycle;
          \draw (o1.center) -- (o2.center) -- (o3.center) -- (o4.center) -- cycle;
        \end{pgfonlayer}

      \end{tikzpicture}
    } 
  }
  %
  \subfigure[]{
    \label{fig:projected_rays}
    \resizebox{0.46\linewidth}{!}{
      \centering
      \pgfdeclarelayer{background}
      \pgfdeclarelayer{foreground}
      \pgfsetlayers{background,main,foreground}
      \begin{tikzpicture}[scale=0.55]
        \tikzset{
          style 0/.style={inner sep=0pt, minimum size=3pt},
          style A/.style={shape=circle, fill=PineGreen, style 0},
          style B/.style={shape=circle, fill=WildStrawberry, style 0},
          style C/.style={shape=circle, fill=Blue!50!RoyalBlue, style 0},
          font={\fontsize{8pt}{9.6}\selectfont}
        }
        
        \begin{pgfonlayer}{foreground}
          \node[draw, style A] (q1) at (-1, -1) {};
          \node[draw, style A] (q2) at (11, -1) {};
          \node[draw, style A] (q3) at (11,  1) {};
          \node[draw, style A] (q4) at ( 9,  1) {};
          \node[draw, style A] (q5) at ( 9,  5) {};
          \node[draw, style A] (q6) at ( 1,  5) {};
          \node[draw, style A] (q7) at ( 1,  1) {};
          \node[draw, style A] (q8) at (-1,  1) {};
          
          \node[text=PineGreen] at ([shift={(225:0.7)}]q1) {$q^1$};
          \node[text=PineGreen] at ([shift={(315:0.5)}]q2) {$q^8$};
          \node[text=PineGreen] at ([shift={( 45:0.7)}]q3) {$q^7$};
          \node[text=PineGreen] at ([shift={( 45:0.7)}]q4) {$q^6$};
          \node[text=PineGreen] at ([shift={( 45:0.7)}]q5) {$q^5$};
          \node[text=PineGreen] at ([shift={(330:0.9)}]q6) {$q^4$};
          \node[text=PineGreen] at ([shift={(135:0.7)}]q7) {$q^3$};
          \node[text=PineGreen] at ([shift={(135:0.5)}]q8) {$q^2$};
        \end{pgfonlayer}
        
        \begin{pgfonlayer}{background}
          \tikzstyle{every path}=[draw, line width=1.0pt, color=black]
          \draw[fill = gray!20] (q1.center) -- (q2.center) -- (q3.center) -- (q4.center) -- (q5.center) -- (q6.center) -- (q7.center) -- (q8.center) -- cycle;
          \node at (8, 2) {$\freeSpace$};
        \end{pgfonlayer}
        
        \begin{pgfonlayer}{foreground}
          \node[draw, style B] (o1) at (3, 1) {};
          \node[draw, style B] (o2) at (7, 1) {};
          \node[draw, style B] (o3) at (7, 3) {};
          \node[draw, style B] (o4) at (3, 3) {};
          
          \node[text=WildStrawberry] at ([shift={( 45:0.6)}]o1) {$o^1$};
          \node[text=WildStrawberry] at ([shift={(135:0.6)}]o2) {$o^2$};
          \node[text=WildStrawberry] at ([shift={(225:0.6)}]o3) {$o^3$};
          \node[text=WildStrawberry] at ([shift={(315:0.7)}]o4) {$o^4$};
        \end{pgfonlayer}
        
        \begin{pgfonlayer}{main}
          \tikzstyle{every path}=[draw, line width=1.0pt, color=black]
          \draw[fill = white] (o1.center) -- (o2.center) -- (o3.center) -- (o4.center) -- cycle;
          \node at (5, 2) {$\obstacle$};
        \end{pgfonlayer}
        
        
        \begin{pgfonlayer}{foreground}
          \node[draw, style C] (p) at (2, 2) {};
          \node[text = Blue!50!RoyalBlue] at ([shift={(135:0.5)}]p)
          {$\genericPoint$}; 
        \end{pgfonlayer}
        
        \begin{pgfonlayer}{main}
          \draw[line width = 0.5pt, color = Violet] (p.center) -- (q6.center);
          \draw[line width = 1.0pt, color = Violet] (q6.center) -- (0.5, 6.5);
          \node[text = Violet] at ([shift={(70:1.5)}]{q6}) {$\gamma(q^4, \genericPoint)$};
          \draw[|-, line width = 0.5pt, color = Violet] ($(q6.center) - (3/10, 1/10)$) -- ($(0.5, 6.5) - (3/10, 1/10)$);
          \node[text = Violet] at ([shift={(165:1.9)}]{q6}) {$r(q^4, \genericPoint)\!=\!\infty$};
        \end{pgfonlayer}
        
        \begin{pgfonlayer}{main}
          \draw[line width = 0.5pt, color = RawSienna] (p.center) -- (o4.center);
          \draw[line width = 1.0pt, color = RawSienna] (o4.center) -- (5.0, 5.0);
          \node[text = RawSienna] at ([shift={(30:2.5)}]{o4}) {$\gamma(o^4, \genericPoint)$};
          \draw[|-|, line width = 0.5pt, color = RawSienna] ($(o4.center) + (-1/5, 1/5)$) -- ($(5, 5) + (-1/5, 1/5)$);
          \node[text = RawSienna] at ([shift={(90:0.25)}]{$(o4)!1.2!(5,5)$}) {$r(o^4, \genericPoint)$};
          
        \end{pgfonlayer}
        
        \begin{pgfonlayer}{main}
          \draw[line width = 0.5pt, color = Red] (p.center) -- (o1.center);
          \draw[line width = 1.0pt, color = Red] (o1.center) -- (5, -1);
          \node[text = Red] at ([shift={(342.5:1.75)}]{o1}) {$\gamma(o^1, \genericPoint)$};
          \draw[line width = 1.0pt, color = Orange!80!Red] (o1.center) -- (4.5, -1);
          \node[text = Orange!80!Red] at ([shift={(332.5:2.75)}]{o1}) {$\gamma_\theta(o^1, \genericPoint)$};
          \draw[->, line width = 0.5pt, color = Orange!80!Red] ($(o1.center) + (1.0, -1.0)$) arc[start angle=-45, end angle=-57, radius=1.0];
          \node[text = Orange!80!Red] at ([shift={(-65:1.5)}]{o1}) {$\theta$};
        \end{pgfonlayer}
        
        \begin{pgfonlayer}{main}
          \draw[line width = 0.5pt, color = RoyalBlue] (p.center) -- (q7.center);
          \draw[line width = 1.0pt, color = RoyalBlue] (q7.center) -- (q1.center);
          \node[text = RoyalBlue] at ([shift={(270:1.25)}]{q7}) {$\gamma(q^3, \genericPoint)$};
          \draw[line width = 1.0pt, color = RawSienna] (q7.center) -- ($(q1.center)+(0, 0.5)$);
          \draw[line width = 1.0pt, color = Violet] (q7.center) -- ($(q1.center)+(0, 1.0)$);
        \end{pgfonlayer}
        
        \begin{pgfonlayer}{background}
          \draw[fill = Red!20] (q7.center) -- ($(q1.center)+(0, 0.5)$) -- ($(q1.center)+(0, 1.0)$) -- cycle;
          \node[text = Red] (l1) at ([shift={(345:1.75)}]{q1}) {$\gamma_{[\phi_1, \phi_2]}(q^3, \genericPoint)$};
          \draw[->, line width = 0.5pt, color = Red] ($(q1.center) + (0.5, 1.0)$) to[out = -60, in = 90] ($(l1.north)+(0, -0.25)$);
        \end{pgfonlayer}
        
        \begin{pgfonlayer}{main}
          \tikzstyle{every path}=[draw, line width=1.0pt, color=black]
          \draw (q1.center) -- (q2.center) -- (q3.center) -- (q4.center) -- (q5.center) -- (q6.center) -- (q7.center) -- (q8.center) -- cycle;
          \draw (o1.center) -- (o2.center) -- (o3.center) -- (o4.center) -- cycle;
        \end{pgfonlayer}
      \end{tikzpicture}
    }
  }
  \caption{(a) For an observer at $\genericPoint$, its visibility
    region $\visibilityPolygon{\genericPoint}$ is shaded in blue.  For
    the free space $\freeSpace$, we have reflex vertices
    $\reflexVertices{\freeSpace} = \{ q^3, q^6, o^1, o^2, o^3, o^4\}$.
    The reflex angles are shown in violet and the non-reflex angles
    are shown in orange. (b) Visualization of notions on projected
    rays and associated distances. The ray cast from $\genericPoint$
    are shown in thin lines, while the projected rays are shown in
    bold ones.}
\end{figure}

\subsubsection*{Projected rays and associated distances}

Given an observer at $\genericPoint$ and a vertex $\vertex \in
\vertices{\freeSpace}$ with $\genericPoint \neq \vertex$, we are
interested in the projection of $\genericPoint$ across the
vertex~$\vertex$.  Analogues of this mathematical structure have been
used for analysis in context of visibility
\cite{LG-RM-PR:92,AG-JC-FB:06-sicon}.  For this purpose, consider the
projected ray $\setRay{\vertex}{\genericPoint}$ that originates at
$\vertex$ and points away from $\genericPoint$ such that it either
ends at a point on the boundary of $\freeSpace$ or extends to
infinity. Formally,
\begin{align*}
  \setRay{\vertex}{\genericPoint}
  & = \{\genericPointQ \in \real^2 \mid \genericPointQ = \vertex +
  \lambda (\vertex - \genericPoint), \foralln 0 \leqslant \lambda
  \leqslant \lambda^*(\vertex, \genericPoint)\}, 
\end{align*}
where $\lambda^*(\vertex, \genericPoint) = \min_{\rho \, \in \,
\realpositive} \left\{\vertex + \rho (\vertex - \genericPoint) \in
\boundary{\freeSpace} \right\}$ if it exists, and $+\infty$ otherwise.
We let $\radiusRay{\vertex}{\genericPoint} := \lambda^*(\vertex,
\genericPoint) \|\vertex - \genericPoint\| \geqslant 0$ denote the
length of the ray.
Next, we define a rotated segment
$\setRotatedRay{\vertex}{\genericPoint}{\theta}$ as
$\setRay{\vertex}{\genericPoint}$ rotated clock-wise around $\vertex$
by $\theta$,
that is,
\begin{align*}
  \setRotatedRay{\vertex}{\genericPoint}{\theta} = \{\genericPointQ
  \in \real^2 \mid \genericPointQ = \vertex + 
  \lambda R(\theta) (\vertex - \genericPoint), \foralln 0 
  \leqslant \lambda \leqslant \lambda^*_\theta(\vertex,
  \genericPoint) \},
\end{align*}
where $R(\theta) \in \real^{2\times 2}$ is the standard rotation
matrix of angle $\theta$ and
\begin{align*}
  \lambda^*_\theta(\vertex, \genericPoint) = \min_{\rho \, \in \,
  \realpositive} \left\{\vertex + \rho R(\theta) (\vertex -
  \genericPoint) \in \boundary{\freeSpace} \right\},
\end{align*}
if it exists, and $+\infty$ otherwise. The corresponding length is
denoted by $\radiusRotatedRay{\vertex}{\genericPoint}{\theta}$. We
further extend this to a closed interval of angles
\begin{align*}
  \setRotatedRaySet{\vertex}{\genericPoint}{\theta_1}{\theta_2}
  & = {\textstyle \bigcup\limits_{\theta \in [\theta_1, \theta_2]}}
  \setRotatedRay{\vertex}{\genericPoint}{\theta}, 
\end{align*}
which corresponds to a bundle of rays cast for angles in $[\theta_1,
\theta_2]$. These notions are illustrated in
\Cref{fig:projected_rays}.


\section{Problem Formulation}\label{sec:problem_formulation}

Our objective is to identify the optimal locations for agents to hide
within obstacle-rich environments with adversaries, and to develop
methodologies to compute them. Motivated by scenarios where there is
no information about the adversary's location, we consider that the
agent's objective is to minimize the risk of line-of-sight detection
from any position in the environment.

Formally, consider a polygonal environment $\environment$ with a set
of polygonal obstacles $\obstacles$ and free space $\freeSpace =
\environment \setminus \interior{\obstacles}$. Suppose that the ego
agent is restricted to move in a polygonal domain $\myDomain \subseteq
\freeSpace$ and aims to minimize the chances of detection by an
adversary whose location is known to be within a polygonal region
$\advDomain \subseteq \freeSpace$. In this case, the ego agent is
visible from $\visibilityPolygon{\genericPoint} \cap \advDomain$.
Therefore, we define the \emph{visibility metric}
$\map{\visibilityMetric}{\myDomain}{\real}$ by
\begin{align}
  \label{eqn:visibility_metric}
  \visibilityMetric(\genericPoint) =
  \|\visibilityPolygon{\genericPoint} \cap \advDomain\|_\mu,
\end{align}
where $\mu$ denotes the Lebesgue measure on $\real^2$.  The visibility
metric quantifies the chances of detection of the ego agent by the
adversary.  Therefore, the ego agent's objective can be cast as the
optimization problem
\begin{equation}
  \label{eqn:optimization_min}
  \min_{\genericPoint \,\in\, \myDomain} \visibilityMetric(\genericPoint),
\end{equation}

If, instead, the ego agent's objective is to identify the most
effective surveillance positions within its domain, the ingredients
introduced above can also be brought to bear. Note that the visibility
metric also quantifies the chances of detection of the adversary by
the ego agent. Therefore, maximizing the chances of detecting the
adversary by line-of-sight corresponds to replacing $\min$ in
\cref{eqn:optimization_min} by $\max$ to yield
\begin{equation}
  \label{eqn:optimization_max}
  \max_{\genericPoint \,\in\, \myDomain} \visibilityMetric(\genericPoint). 
\end{equation}
\Cref{fig:problem_description} illustrates this discussion.

\newpage

\begin{figure}[htb]
  \centering
  \pgfdeclarelayer{background}
\pgfdeclarelayer{foreground}
\pgfsetlayers{background,main,foreground}
\begin{tikzpicture}[scale=0.55]
  \tikzset{
    style 0/.style={inner sep=0pt, minimum size=3pt},
    style A/.style={shape=circle, fill=PineGreen, style 0},
    style B/.style={shape=circle, fill=WildStrawberry, style 0},
    style C/.style={shape=circle, fill=Blue!50!RoyalBlue, style 0},
    font={\fontsize{8pt}{9.6}\selectfont}
  }

  \begin{pgfonlayer}{foreground}
    \node[draw, style A] (q1) at (-1, -1) {};
    \node[draw, style A] (q2) at (11, -1) {};
    \node[draw, style A] (q3) at (11,  1) {};
    \node[draw, style A] (q4) at ( 9,  1) {};
    \node[draw, style A] (q5) at ( 9,  5) {};
    \node[draw, style A] (q6) at ( 1,  5) {};
    \node[draw, style A] (q7) at ( 1,  1) {};
    \node[draw, style A] (q8) at (-1,  1) {};

  \end{pgfonlayer}

  \begin{pgfonlayer}{background}
    \tikzstyle{every path}=[draw, line width=1.0pt, color=black]
    \draw[fill = gray!20] (q1.center) -- (q2.center) -- (q3.center) -- (q4.center) -- (q5.center) -- (q6.center) -- (q7.center) -- (q8.center) -- cycle;
  \end{pgfonlayer}

  \begin{pgfonlayer}{foreground}
    \node[draw, style B] (o1) at (3, 1) {};
    \node[draw, style B] (o2) at (7, 1) {};
    \node[draw, style B] (o3) at (7, 3) {};
    \node[draw, style B] (o4) at (3, 3) {};

  \end{pgfonlayer}

  \begin{pgfonlayer}{background}
    \tikzstyle{every path}=[draw, line width=1.0pt, color=black]
    \draw[fill = white] (o1.center) -- (o2.center) -- (o3.center) -- (o4.center) -- cycle;
    \node at (5, 2) {$\obstacle$};
  \end{pgfonlayer}
  
  \begin{pgfonlayer}{background}
    \tikzstyle{every path}=[draw, line width=0.5pt, color=ForestGreen]
    \draw[fill = ForestGreen!20] (-1,-1) -- (11,-1) -- (11,1) -- (-1,1) -- cycle;
    \node[text = ForestGreen] at (5, 0) {$\myDomain$};
  \end{pgfonlayer}

  \begin{pgfonlayer}{background}
    \tikzstyle{every path}=[draw, line width=0.5pt, color=RedOrange]
    \draw[fill = RedOrange!20] (1,3) -- (9,3) -- (9,5) -- (1,5) -- cycle;
    \node[text = Red!50!RedOrange] at (5, 4) {$\advDomain$};
  \end{pgfonlayer}

  \begin{pgfonlayer}{main}
    \tikzstyle{every path}=[draw, line width=1.0pt, color=black]
    \draw (q1.center) -- (q2.center) -- (q3.center) -- (q4.center) -- (q5.center) -- (q6.center) -- (q7.center) -- (q8.center) -- cycle;
    \draw (o1.center) -- (o2.center) -- (o3.center) -- (o4.center) -- cycle;
  \end{pgfonlayer}

  \begin{pgfonlayer}{foreground}
    \node[draw, shape=circle, fill=Magenta, style 0] (p2) at (0,0) {};
    \node[text = Magenta] at ([shift={(135:0.85)}]{p2}) {$\genericPoint_1$};
  \end{pgfonlayer}

  \begin{pgfonlayer}{background}
    \draw[dashed, color = Magenta, fill = Magenta!20, line width=0.5pt] (p2.center)+(0.5, 0) arc[start angle=0, end angle=360, radius=0.5];
    \node[text = Magenta] at ([shift={(340:2.00)}]{p2}) {$\oBall{\epsilon}{\genericPoint_1} \intersection \myDomain$};
  \end{pgfonlayer}

  \begin{pgfonlayer}{background}
    \tikzstyle{every path}=[draw, line width=0.5pt, color=Magenta]
    \node[text = Magenta] (l3) at (-1,2.25) {$\visibilityPolygon{\genericPoint_1} \intersection \advDomain = \varnothing$};
    \draw (p2.center) -- (o4.center);
  \end{pgfonlayer}

\begin{pgfonlayer}{foreground}
  \node[draw, shape=circle, fill=Violet, style 0] (p1) at (q4.center) {};
  \node[text = Violet] at ([shift={(225:0.85)}]{q4}) {$\genericPoint_2$};
\end{pgfonlayer}

\begin{pgfonlayer}{background}
  \draw[dashed, color = Violet, fill = Violet!20, line width=0.5pt] (q4.center) -- (q4.center)+(0.5, 0) arc[start angle=0, end angle=-180, radius=0.5] -- cycle;
  \node[text = Violet] at ([shift={(340:2.00)}]{q4}) {$\oBall{\epsilon}{\genericPoint_2} \intersection \myDomain$};
\end{pgfonlayer}

\begin{pgfonlayer}{background}
  \tikzstyle{every path}=[draw, line width=0.5pt, color=Violet]
  \draw[fill = Violet!20] (9,3) -- (q5.center) -- (5,5) -- (o3.center) -- cycle;
  \node[text = Violet] (l2) at (10.5,3) {$\visibilityPolygon{\genericPoint_2} \intersection \advDomain$};
  \draw (p1.center) -- (o3.center);
  \draw[->] (7.5,4) to[out = 0, in = 90] (l2.north);
\end{pgfonlayer}

\end{tikzpicture}
\caption{In this example, with $\myDomain$ shaded in green, and
  $\advDomain$ shaded in orange, $\genericPoint_1$ is a local
  minimizer as $0 = \visibilityMetric(\genericPoint_1) \leqslant
  \visibilityMetric(\genericPointQ), \foralln \genericPointQ \in
  \oBall{\epsilon}{\genericPoint_1} \intersection \myDomain$.
  Analogously, $\genericPoint_2$ is a local maximizer as
  $\visibilityMetric(\genericPoint_2) \geqslant
  \visibilityMetric(\genericPointQ), \foralln \genericPointQ \in
  \oBall{\epsilon}{\genericPoint_2} \intersection
  \myDomain$.}\label{fig:problem_description}
\end{figure}

The objectives of this work are to address the following questions:
\begin{enumerate}[label=\textbf{Q\arabic*}]
  \item \textbf{Critical Points:} Can we characterize the set of
  critical points, the observer locations where the shape of the
  visibility polygon changes abruptly?\label{enum:critical_points}%
  \item \textbf{Regularity Properties:} What are the smoothness
  properties of the visibility polygon/metric? Are there any
  closed-form expressions for their generalized gradient?
  \label{enum:regularity_properties}
  \item \textbf{Limited Ranges:} Can the characterization of
  properties of the visibility polygon and visibility metric be
  extended to settings where the ego agent has a limited field of view
  and/or limited range?\label{enum:extension}
  \item \textbf{Local Extrema:} What are the local extrema for the
  visibility metric within the agent's domain? How can these points be
  reached? \label{enum:local_extrema}
\end{enumerate}

In the subsequent sections, we first characterize the critical points
of the visibility polygon by introducing anchors and inflection
segments, and exploring their properties to understand the regularity
of the visibility polygon. Building on these insights, we examine the
visibility metric, study its smoothness properties and extend the
analysis to cases with limited range and field of view.  Finally, we
develop methodologies to find local extrema within the agent's domain.


\section{Mathematical Structures for Visibility}\label{sec:analysis}

In this section, we characterize the critical points of the visibility
polygon to address \labelcref{enum:critical_points} through the
notions of anchors and inflection segments. These mathematical
structures are crucial for the analysis required to address the
remaining questions posed in \Cref{sec:problem_formulation}.

\subsection{Anchors Associated with an
  Observer}\label{sub:visibility_notions}

The regularity properties of a visibility polygon are closely related
to the notion of \emph{anchor}, a concept associated with an observer
introduced for simple polygons in \cite{LG-RM-PR:92}.  Here, we
formalize the notion and characterize its properties.  An
\emph{anchor} $\anchor \in \reflexVertices{\freeSpace}$ is a reflex
vertex that occludes a portion of the environment from
$\genericPoint$. Formally, a reflex vertex $\anchor \in
\reflexVertices{\freeSpace}$ is an anchor for an observer at
$\genericPoint \in \freeSpace$ if $\anchor$ is visible from
$\genericPoint$ \emph{and} the projected ray from $\anchor$ pointing
away from $\genericPoint$ is contained in $\freeSpace$, that is,
$\setRay{\anchor}{\genericPoint} \subset \freeSpace$. This new
definition for anchors requires less steps of computation compared to
existing definitions \cite{LG-RM-PR:92,AG-JC-FB:06-sicon}. We let
$\anchors{\genericPoint}$ denote the set of all anchors for an
observer at $\genericPoint$.  In \Cref{fig:projected_rays}, the
vertices $o^1$, $o^4$, $q^1$, $q^3$ and $q^4$ are visible from
$\genericPoint$. Of these, $o^1$, $o^4$, and $q^3$ are anchors for
$\genericPoint$, while $q^1$ and $q^4$ are not.
An anchor is \emph{positively (resp. negatively) oriented} wrt
$\genericPoint$ if there exists $\delta > 0$ such that
$\setRotatedRaySet{\anchor}{\genericPoint}{-\pi}{0} \intersection
\oBall{\delta}{\anchor} \subset \visibilityPolygon{\genericPoint}$
(resp. $\setRotatedRaySet{\anchor}{\genericPoint}{0}{\pi}
\intersection \oBall{\delta}{\anchor} \subset
\visibilityPolygon{\genericPoint}$). Intuitively, a positively (resp.
negatively) oriented $\anchor$ means that if $\genericPoint$ moves
clockwise (resp. counterclockwise) wrt $\anchor$, it sees some region
that was earlier occluded by $\anchor$ (see \Cref{fig:anchors_1}).
The anchors play a crucial role in characterizing visibility, as
stated next.

\begin{lemma}[Construction of the visibility polygon]
  \label{lem:visibility_construction}
  Let $\genericPoint \in \freeSpace$ be an observer. Suppose
  $\{\vertex_i\}_{i=1}^{n_{\vertex}(\genericPoint)}$ are the vertices
  of $\freeSpace$ visible to $\genericPoint$ and let
  $\anchors{\genericPoint}$ denote the set of anchors from~$x$. Let
  $\{u_i\}_{i \in \anchors{\genericPoint}}$ denote the projections of
  $\genericPoint$ across these anchors on the boundary of
  $\freeSpace$. Then, the vertices of the visibility polygon, \emph{in
  no particular order}, are given by
  \smash{$\{\vertex_i\}_{i=1}^{n_{\vertex}(\genericPoint)} \cup
    \{u_i\}_{i \in \anchors{\genericPoint}}$.}
\end{lemma}

\begin{figure}[htb]
  \centering
  \pgfdeclarelayer{background}
  \pgfdeclarelayer{foreground}
  \pgfsetlayers{background,main,foreground}
  \begin{tikzpicture}[scale=0.55]
    \tikzset{
      style 0/.style={inner sep=0pt, minimum size=3pt},
      style A/.style={shape=circle, fill=PineGreen, style 0},
      style B/.style={shape=circle, fill=WildStrawberry, style 0},
      style C/.style={shape=circle, fill=Blue!50!RoyalBlue, style 0},
      font={\fontsize{8pt}{9.6}\selectfont}
    }
    
    \begin{pgfonlayer}{foreground}
      \node[draw, style A] (q1) at (-1, -1) {};
      \node[draw, style A] (q2) at (11, -1) {};
      \node[draw, style A] (q3) at (11,  1) {};
    \node[draw, style A] (q4) at ( 9,  1) {};
    \node[draw, style A] (q5) at ( 9,  5) {};
    \node[draw, style A] (q6) at ( 1,  5) {};
    \node[draw, style A] (q7) at ( 1,  1) {};
    \node[draw, style A] (q8) at (-1,  1) {};

  \end{pgfonlayer}

  \begin{pgfonlayer}{background}
    \tikzstyle{every path}=[draw, line width=1.0pt, color=black]
    \draw[fill = gray!20] (q1.center) -- (q2.center) -- (q3.center) -- (q4.center) -- (q5.center) -- (q6.center) -- (q7.center) -- (q8.center) -- cycle;
    \node at (7.9, 3.9) {$\freeSpace$};
  \end{pgfonlayer}

  \begin{pgfonlayer}{foreground}
    \node[draw, style B] (o1) at (3, 1) {};
    \node[draw, style B] (o2) at (7, 1) {};
    \node[draw, style B] (o3) at (7, 3) {};
    \node[draw, style B] (o4) at (3, 3) {};

    \node[text=WildStrawberry] at ([shift={( 45:0.6)}]o1) {$o^1$};
    \node[text=WildStrawberry] at ([shift={(135:0.6)}]o2) {$o^2$};
    \node[text=WildStrawberry] at ([shift={(225:0.6)}]o3) {$o^3$};
    \node[text=WildStrawberry] at ([shift={(315:0.6)}]o4) {$o^4$};
  \end{pgfonlayer}

  \begin{pgfonlayer}{main}
    \tikzstyle{every path}=[draw, line width=1.0pt, color=black]
    \draw[fill = white] (o1.center) -- (o2.center) -- (o3.center) -- (o4.center) -- cycle;
    \node at (5, 2) {$\obstacle$};
  \end{pgfonlayer}

  \begin{pgfonlayer}{foreground}
    \node[draw, style C] (p) at (3, 0) {};
    \node[text = Blue!50!RoyalBlue] at ([shift={(270:0.5)}]p) {$\genericPoint$};
  \end{pgfonlayer}

  \begin{pgfonlayer}{foreground}
    \node[draw, shape=circle, fill=Violet, style 0] (p1) at (3.25, 0) {};
    \node[draw, shape=circle, fill=Orange, style 0] (p2) at (2.75, 0) {};
    \node[text = Violet] at ([shift={(330:0.5)}]{p1}) {$\genericPoint_1$};
    \node[text = Orange] at ([shift={(210:0.5)}]{p2}) {$\genericPoint_2$};
  \end{pgfonlayer}

  \begin{pgfonlayer}{background}
    \tikzstyle{every path}=[draw, line width=0pt, color=RoyalBlue]
    \draw[fill = RoyalBlue!20] (q7.center) -- (q6.center) -- (3,5) -- (o1.center) -- (o2.center) -- (9, 1.5) -- (q4.center) -- (q3.center) -- (q2.center) -- (q1.center) -- (q8.center) -- cycle;
    \node[text = Blue!50!RoyalBlue] at (2.5, 4.5) {$\visibilityPolygon{\genericPoint}$};
  \end{pgfonlayer}

  \begin{pgfonlayer}{background}
    \tikzstyle{every path}=[draw, line width=0pt, color=Violet]
    \draw[fill = Violet, fill opacity = 0.1] (q7.center) -- (q6.center) -- (2,5) -- (o1.center) -- (o2.center) -- (9, {1.5+1/30}) -- (q4.center) -- (q3.center) -- (q2.center) -- (q1.center) -- (q8.center) -- cycle;
    \node[text = Violet] at (1.75, 3.25) {$\visibilityPolygon{\genericPoint_1}$};
    \tikzstyle{every path}=[draw, line width=0pt, color=Orange]
    \draw[fill = Orange, fill opacity = 0.1] (q7.center) -- (q6.center) -- ({19/6},5) -- (o4.center) -- (o1.center) -- (o2.center) -- (9, {1.5-1/30}) -- (q4.center) -- (q3.center) -- (q2.center) -- (q1.center) -- (q8.center) -- cycle;
    \node[text = Orange] at (3.75, 3.75) {$\visibilityPolygon{\genericPoint_2}$};
  \end{pgfonlayer}

  \begin{pgfonlayer}{main}
    \node (t1) at ([shift=({atan(0.25)}:0.5)]{o2}) {};
    \node (t2) at ([shift=({atan(0.25)}:-0.5)]{o2}) {};
    \draw[color = PineGreen, fill = PineGreen] (o2.center) -- (t1.center) arc[radius = 0.5, start angle = {atan(0.25)}, end angle = {-180 + atan(0.25)}] (t2.center) -- (o2.center);

    \draw[line width=0.5pt, color=Blue!50!RoyalBlue] (p.center) -- (o2.center);
    \draw[line width=0.5pt, color=Violet] (p1.center) -- (o2.center);
    \draw[line width=0.5pt, color=Orange] (p2.center) -- (o2.center);
  \end{pgfonlayer}

  \begin{pgfonlayer}{main}
    \node (t4) at ([shift=(90:0.5)]{o1}) {};
    \node (t5) at ([shift=(270:0.5)]{o1}) {};
    \draw[color = WildStrawberry, fill = WildStrawberry] (o1.center) -- (t4.center) arc[radius = 0.5, start angle = 90, end angle = 270] (t5.center) -- (o1.center);

    \draw[line width=0.5pt, color=Blue!50!RoyalBlue] (p.center) -- (o1.center);
    \draw[line width=0.5pt, color=Violet] (p1.center) -- (o1.center);
  \end{pgfonlayer}

  \begin{pgfonlayer}{main}
    \node (t7) at ([shift=(90:0.5)]{o4}) {};
    \node (t8) at ([shift=(270:0.5)]{o4}) {};
    \draw[color = WildStrawberry, fill = WildStrawberry] (o4.center) -- (t7.center) arc[radius = 0.5, start angle = 90, end angle = 270] (t8.center) -- (o4.center);

    \draw[line width=0.5pt, color=Blue!50!RoyalBlue] (p.center) -- (o4.center);
    \draw[line width=0.5pt, color=Orange] (p2.center) -- (o4.center);
  \end{pgfonlayer}

  \begin{pgfonlayer}{main}
    \tikzstyle{every path}=[draw, line width=0.5pt, color=RoyalBlue]
    \draw (q7.center) -- (q6.center) -- (3,5) -- (o1.center) -- (o2.center) -- (9, 1.5) -- (q4.center) -- (q3.center) -- (q2.center) -- (q1.center) -- (q8.center) -- cycle;
    \tikzstyle{every path}=[draw, line width=0.5pt, color=Violet]
    \draw (q7.center) -- (q6.center) -- (2,5) -- (o1.center) -- (o2.center) -- (9, {1.5+1/30}) -- (q4.center) -- (q3.center) -- (q2.center) -- (q1.center) -- (q8.center) -- cycle;
    \tikzstyle{every path}=[draw, line width=0.5pt, color=Orange]
    \draw (q7.center) -- (q6.center) -- ({19/6},5) -- (o4.center) -- (o1.center) -- (o2.center) -- (9, {1.5-1/30}) -- (q4.center) -- (q3.center) -- (q2.center) -- (q1.center) -- (q8.center) -- cycle;

    \tikzstyle{every path}=[draw, line width=1.0pt, color=black]
    \draw (q1.center) -- (q2.center) -- (q3.center) -- (q4.center) -- (q5.center) -- (q6.center) -- (q7.center) -- (q8.center) -- cycle;
    \draw (o1.center) -- (o2.center) -- (o3.center) -- (o4.center) -- cycle;

    \node[text = PineGreen] at ([shift=({300}:1.0)]{o2}) {$\setRotatedRaySet{\genericPoint}{o^2}{0}{\pi} \intersection \oBall{\epsilon}{o^2}$};
  \end{pgfonlayer}

\end{tikzpicture}
\caption{An observer at $\genericPoint$ has anchors
  $\anchors{\genericPoint} = \left\{ o^1, o^2, o^4 \right\}$. Anchors
  $o^1, o^4$ are positively oriented (showed in red) and $o^2$ is
  negatively oriented wrt $\genericPoint$ (shown in green).}
\label{fig:anchors_1}
\end{figure}

When an observer moves from one point to another, the visibility
polygon changes: some previously visible vertices of $\freeSpace$ may
become occluded by anchors; some vertices remain visible; some
previously hidden vertices may become visible as they emerge from
behind anchors; the projections of anchors onto the boundary may
shift, altering the polygon's shape.  Crucially, all changes---whether
vertices disappearing, appearing, or moving---are directly tied to the
anchors and their projections.

The characterization~\cite{LG-RM-PR:92,AG-JC-FB:06-sicon} above is
centered around the observer at $\genericPoint$, and requires the
construction of projected rays for each change in position.  We next
develop an alternative approach centered instead on the static
environment $\environment$, which has the advantage that it remains
unchanged as the observer moves.

Given a reflex vertex $\reflexVertex \in \reflexVertices{\freeSpace}$,
choose its neighboring vertices $\vertex_1$ and $\vertex_2$ such that
$\vertex_1, \reflexVertex, \vertex_2$ are vertices of the environment
$\environment$ or an obstacle $\obstacle \in \obstacles$, in the
standard order.
We now define the edge unit vectors $\edgeUVA{\reflexVertex} :=
{(\vertex_1 - \reflexVertex)}/{\|\vertex_1 - \reflexVertex\|_2}$ and
$\edgeUVB{\reflexVertex} := {(\vertex_2 - \reflexVertex)}/{\|\vertex_2
- \reflexVertex\|_2}$, and consider the four quadrants centered at
$\reflexVertex$,
\begin{align*}
  \region_1(\reflexVertex)
  & :=
    \interior{\conicalHull{\{+\edgeUVA{\reflexVertex}, 
    +\edgeUVB{\reflexVertex}\}}},
  & \region_2(\reflexVertex)
  &:= \conicalHull{\{+\edgeUVA{\reflexVertex},
    -\edgeUVB{\reflexVertex}\}},
  \\
  \region_3(\reflexVertex)
  & :=
    \interior{\conicalHull{\{-\edgeUVA{\reflexVertex},
    -\edgeUVB{\reflexVertex}\}}},
  & \region_4(\reflexVertex) &:= \conicalHull{\{-\edgeUVA{\reflexVertex},
    +\edgeUVB{\reflexVertex}\}} . 
\end{align*}
We omit the reflex vertex in the argument if it is clear from the
context (see \Cref{fig:characterization_anchor}).
Alternatively, whether or not a reflex vertex is an anchor can be
characterized by the relative position of the observer with respect to
it, as seen in the following result. 

\begin{lemma}[When is a reflex vertex an anchor?]
  \label{lem:characterization_anchor}
  A reflex vertex $\reflexVertex \in \reflexVertices{\freeSpace}$ is
  an anchor for an observer at $\genericPoint \in \freeSpace$ if and
  only if $\reflexVertex$ is visible from $\genericPoint$ \emph{and}
  $\genericPoint \in \region_2(\reflexVertex) \union
  \region_4(\reflexVertex)$. Furthermore, if this is the case, then
  $\reflexVertex$ is positively (resp. negatively) oriented wrt
  $\genericPoint$ if $\genericPoint \in \region_2(\reflexVertex)$
  (resp.  $\genericPoint \in \region_4(\reflexVertex)$).
\end{lemma}
\begin{proof}
  ($\Rightarrow$) Let $\reflexVertex$ be an anchor for
  $\genericPoint$. Then, $\reflexVertex$ is visible from
  $\genericPoint$, and $\setRay{\reflexVertex}{\genericPoint} \subset
  \freeSpace$. As $\setRay{\reflexVertex}{\genericPoint}$ is in the
  opposite quadrant to $\genericPoint$, it immediately follows that
  $\genericPoint \in \region_2 \union \region_4$.

  ($\Leftarrow$) Suppose $\reflexVertex$ is visible from
  $\genericPoint$ and $\genericPoint \in \region_2 \union \region_4$.
  Again, as $\setRay{\reflexVertex}{\genericPoint}$ is in the opposite
  quadrant to $\genericPoint$, it follows that
  $\setRay{\reflexVertex}{\genericPoint} \subset \freeSpace$. So,
  $\reflexVertex$ is an anchor for $\genericPoint$.
    
  The second part follows from the definition of positively and
  negatively oriented anchors.
\end{proof}

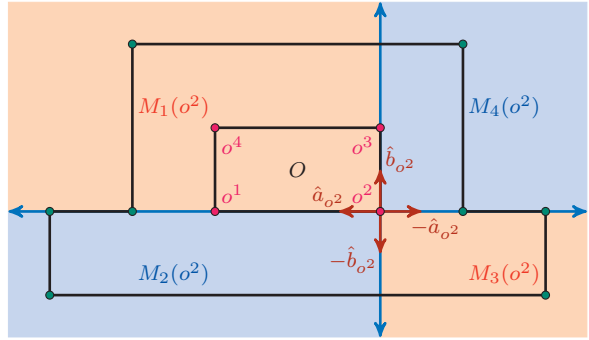
\begin{figure}[htb]
  \centering%
  \pgfdeclarelayer{background} \pgfdeclarelayer{foreground}
  \pgfsetlayers{background,main,foreground}
  \begin{tikzpicture}[scale=0.55]
    \tikzset{
      style 0/.style={inner sep=0pt, minimum size=3pt},
      style A/.style={shape=circle, fill=PineGreen, style 0},
      style B/.style={shape=circle, fill=WildStrawberry, style 0},
      style C/.style={shape=circle, fill=Blue!50!RoyalBlue, style 0},
      font={\fontsize{8pt}{9.6}\selectfont}
    }
    
    \begin{pgfonlayer}{background}
      \tikzstyle{every path}=[draw, line width=1.0pt, color=black]
      \draw[fill = gray!20] (q1.center) -- (q2.center) -- (q3.center) -- (q4.center) -- (q5.center) -- (q6.center) -- (q7.center) -- (q8.center) -- cycle;
      \node at (7.9, 3.9) {$\freeSpace$};
    \end{pgfonlayer}

    \begin{pgfonlayer}{background}
      \tikzstyle{every path}=[draw, line width=1.0pt, color=black]
      \draw[fill = white] (o1.center) -- (o2.center) -- (o3.center) -- (o4.center) -- cycle;
    \end{pgfonlayer}

  \begin{pgfonlayer}{main}
    \tikzstyle{every path}=[draw, line width=0pt, color=RoyalBlue]
    \draw[draw = Orange!30, fill = Orange!30, line width=0pt] (7, 1) -- (7, 6) -- (-2, 6) -- (-2, 1) -- cycle;
    \draw[draw = Orange!30, fill = Orange!30, line width=0pt] (7, 1) -- (12, 1) -- (12, -2) -- (7, -2) -- cycle;
    \draw[draw = RoyalBlue!20, fill = RoyalBlue!20, line width=0pt] (7, 1) -- (7, 6) -- (12, 6) -- (12, 1) -- cycle;
    \draw[draw = RoyalBlue!20, fill = RoyalBlue!20, line width=0pt] (7, 1) -- (-2, 1) -- (-2, -2) -- (7, -2) -- cycle;
  \end{pgfonlayer}

  \begin{pgfonlayer}{foreground}
    \tikzstyle{every path}=[draw, line width=1.0pt, color=black]
    \draw (q1.center) -- (q2.center) -- (q3.center) -- (q4.center) -- (q5.center) -- (q6.center) -- (q7.center) -- (q8.center) -- cycle;
    \draw (o1.center) -- (o2.center) -- (o3.center) -- (o4.center) -- cycle;
  \end{pgfonlayer}

  \begin{pgfonlayer}{foreground}
    \node[text = Orange!20!Red] at (2, 3.5) {$\region_1(o^2)$};
    \node[text = Orange!20!Red] at (10, -0.5) {$\region_3(o^2)$};
    \node[text = Blue!50!RoyalBlue] at (2, -0.5) {$\region_2(o^2)$};
    \node[text = Blue!50!RoyalBlue] at (10, 3.5) {$\region_4(o^2)$};
  \end{pgfonlayer}

  \begin{pgfonlayer}{main}
    \tikzstyle{every path}=[draw, line width=1.0pt, color=RoyalBlue]
    \draw[>=stealth', <->] (7, -2) -- (7, 6);
    \draw[>=stealth', <->] (-2, 1) -- (12, 1);
  \end{pgfonlayer}

  \begin{pgfonlayer}{foreground}
    \tikzstyle{every path}=[draw, line width=1pt, color=BrickRed]
    \draw[>=stealth', <->] (7, 0) -- (7, 2);
    \draw[>=stealth', <->] (6, 1) -- (8, 1);
    \node at (7.5, 2.25) {$\edgeUVB{o^2}$};
    \node at (6.35, -0.15) {$-\edgeUVB{o^2}$};
    \node at (8.35, 0.6) {$-\edgeUVA{o^2}$};
    \node at (5.75, 1.35) {$\edgeUVA{o^2}$};
  \end{pgfonlayer}

  \begin{pgfonlayer}{foreground}
    \node[draw, style A] (q1) at (-1, -1) {};
    \node[draw, style A] (q2) at (11, -1) {};
    \node[draw, style A] (q3) at (11,  1) {};
    \node[draw, style A] (q4) at ( 9,  1) {};
    \node[draw, style A] (q5) at ( 9,  5) {};
    \node[draw, style A] (q6) at ( 1,  5) {};
    \node[draw, style A] (q7) at ( 1,  1) {};
    \node[draw, style A] (q8) at (-1,  1) {};

  \end{pgfonlayer}
  
  \begin{pgfonlayer}{foreground}
    \node[draw, style B] (o1) at (3, 1) {};
    \node[draw, style B] (o2) at (7, 1) {};
    \node[draw, style B] (o3) at (7, 3) {};
    \node[draw, style B] (o4) at (3, 3) {};

    \node[text=WildStrawberry] at ([shift={( 45:0.6)}]o1) {$o^1$};
    \node[text=WildStrawberry] at ([shift={(135:0.6)}]o2) {$o^2$};
    \node[text=WildStrawberry] at ([shift={(225:0.6)}]o3) {$o^3$};
    \node[text=WildStrawberry] at ([shift={(315:0.6)}]o4) {$o^4$};

    \node at (5, 2) {$\obstacle$};
  \end{pgfonlayer}
\end{tikzpicture}
\caption{Here, $o^2$ is a reflex vertex with consecutive vertices
  $o^1, o^2, o^3$ in the standard order, and the corresponding regions
  are shown. If the observer $\genericPoint$ is visible to $o^2$ and
  lies in $\region_2 \union \region_4$, then $o^2$ serves as an anchor
  for $\genericPoint$; otherwise, it does not. When the reflex vertex
  is part of $\environment$ instead of $\obstacle$, the standard
  ordering for environment vertices applies.}
  \label{fig:characterization_anchor}
\end{figure}

\subsection{Critical Points and Inflection Segments}%
\label{sub:critical_point}%

Here, we address~\labelcref{enum:critical_points}. A \emph{critical
point} corresponds to an observer position whose set of anchors
changes under arbitrarily small perturbations. This notion is
illustrated in \Cref{fig:anchors_1} with the point $\genericPoint$. We
denote the set of such points by $ \setCriticalPoints$. Formally,
\begin{align*}
  \setCriticalPoints = \{ \genericPoint \in \freeSpace \mid \forall
  \epsilon>0, \; \existsn
  \genericPointQ_1, \genericPointQ_2 \in
  \oBall{\epsilon}{\genericPoint} \intersection \freeSpace
  \text{ s.t. }
  \anchors{\genericPointQ_1} \neq \anchors{\genericPointQ_2}
  \}.
\end{align*}

We next provide an explicit characterization of the set
$\setCriticalPoints$.  Based on \Cref{lem:characterization_anchor},
there are two possibilities for a reflex vertex $\reflexVertex$ to get
added to or removed from the list of anchors: 
\begin{description}
\item[Case 1:] if the observer moves from $\region_1 \union \region_3$
  to $\region_2 \union \region_4$, then $\reflexVertex$ gets added to
  the list of anchors. Vice versa, if the observer moves from
  $\region_2 \union \region_4$ to $\region_1 \union \region_3$, then
  $\reflexVertex$ gets removed from the list of anchors;
\item[Case~2:] If $\reflexVertex$ gets occluded by another anchor, say
  $\anchor$, then $\reflexVertex$ gets removed from the list of
  anchors. Vice versa, if $\reflexVertex$ is revealed from across
  another anchor, say $\anchor$, then $\reflexVertex$ gets added to
  the list of anchors.
\end{description}
We can use these observations to characterize the set of points in
free space where the set of anchors changes. Case~1 arises on points
along the 4 rays starting from $\reflexVertex$ and pointing towards
$\edgeUVA{\reflexVertex}$, $\edgeUVB{\reflexVertex}$,
$-\edgeUVA{\reflexVertex}$, $-\edgeUVB{\reflexVertex}$. We call these
\emph{Type~I points}. Case~2 arises on points along the rays lying on
the line passing through two anchors. We call these \emph{Type~II
points}.
These points can be collectively characterized via the notion of
\emph{inflection segments}.

\begin{definition}[Inflection Segments]
  \label{def:inflection_segments}
  Let $\reflexVertex \in \reflexVertices{\freeSpace}$ be a reflex
  vertex with neighboring vertices $\vertex_1$ and $\vertex_2$ such
  that $\vertex_1, \reflexVertex, \vertex_2$ are consecutive vertices
  in the standard order.  Then, the rays
  $\setRay{\reflexVertex}{\vertex_1} \intersection \freeSpace$,
  $\setRay{\reflexVertex}{\vertex_2} \intersection \freeSpace$,
  $\setRay{\vertex_1}{\reflexVertex} \intersection \freeSpace$, and
  $\setRay{\vertex_2}{\reflexVertex} \intersection \freeSpace$ are
  \emph{Type~I inflection segments} of $\freeSpace$.  The union of the
  4 rays corresponds to $\boundary{\region_2 \union \region_4}$.

  Let $\reflexVertex, \reflexVertex' \in \reflexVertices{\freeSpace}$
  be such that $\reflexVertex'$ is visible from $\reflexVertex$. Then,
  the rays $\setRay{\reflexVertex}{\reflexVertex'} \intersection
  \freeSpace$ and $\setRay{\reflexVertex'}{\reflexVertex}
  \intersection \freeSpace$ are \emph{Type~II inflection segments} of
  $\freeSpace$. \oprocend
\end{definition}

\Cref{fig:inflection_segments} illustrates this notion. We denote by
$\setInflectionSegments$ the set of all inflection segments, and by
$\setInflectionPoints$ the set of points on an inflection segment,
that is, $\setInflectionPoints := \{ \genericPoint \in \freeSpace \mid
\exists \inflectionSegment \in \setInflectionSegments \text{ s.t. }
\genericPoint \in \inflectionSegment \}$.

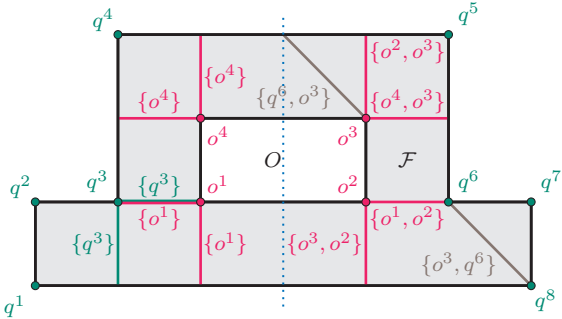
\begin{figure}[htb]
  \centering
  \pgfdeclarelayer{background}
  \pgfdeclarelayer{foreground}
  \pgfsetlayers{background,main,foreground}
  \begin{tikzpicture}[scale=0.55]
    \tikzset{
      style 0/.style={inner sep=0pt, minimum size=3pt},
      style A/.style={shape=circle, fill=PineGreen, style 0},
      style B/.style={shape=circle, fill=WildStrawberry, style 0},
      style C/.style={shape=circle, fill=Blue!50!RoyalBlue, style 0},
      font={\fontsize{8pt}{9.6}\selectfont}
    }
    
    \begin{pgfonlayer}{foreground}
      \node[draw, style A] (q1) at (-1, -1) {};
    \node[draw, style A] (q2) at (11, -1) {};
    \node[draw, style A] (q3) at (11,  1) {};
    \node[draw, style A] (q4) at ( 9,  1) {};
    \node[draw, style A] (q5) at ( 9,  5) {};
    \node[draw, style A] (q6) at ( 1,  5) {};
    \node[draw, style A] (q7) at ( 1,  1) {};
    \node[draw, style A] (q8) at (-1,  1) {};

    \node[text=PineGreen] at ([shift={(225:0.7)}]q1) {$q^1$};
    \node[text=PineGreen] at ([shift={(315:0.5)}]q2) {$q^8$};
    \node[text=PineGreen] at ([shift={( 45:0.7)}]q3) {$q^7$};
    \node[text=PineGreen] at ([shift={( 45:0.7)}]q4) {$q^6$};
    \node[text=PineGreen] at ([shift={( 45:0.7)}]q5) {$q^5$};
    \node[text=PineGreen] at ([shift={(135:0.5)}]q6) {$q^4$};
    \node[text=PineGreen] at ([shift={(135:0.7)}]q7) {$q^3$};
    \node[text=PineGreen] at ([shift={(135:0.5)}]q8) {$q^2$};
  \end{pgfonlayer}

  \begin{pgfonlayer}{background}
    \tikzstyle{every path}=[draw, line width=1.0pt, color=black]
    \draw[fill = gray!20] (q1.center) -- (q2.center) -- (q3.center) -- (q4.center) -- (q5.center) -- (q6.center) -- (q7.center) -- (q8.center) -- cycle;
    \node at (8, 2) {$\freeSpace$};
  \end{pgfonlayer}

  \begin{pgfonlayer}{foreground}
    \node[draw, style B] (o1) at (3, 1) {};
    \node[draw, style B] (o2) at (7, 1) {};
    \node[draw, style B] (o3) at (7, 3) {};
    \node[draw, style B] (o4) at (3, 3) {};

    \node[text=WildStrawberry] at ([shift={( 45:0.6)}]o1) {$o^1$};
    \node[text=WildStrawberry] at ([shift={(135:0.6)}]o2) {$o^2$};
    \node[text=WildStrawberry] at ([shift={(225:0.6)}]o3) {$o^3$};
    \node[text=WildStrawberry] at ([shift={(315:0.6)}]o4) {$o^4$};
  \end{pgfonlayer}

  \begin{pgfonlayer}{main}
    \tikzstyle{every path}=[draw, line width=1.0pt, color=black]
    \draw[fill = white] (o1.center) -- (o2.center) -- (o3.center) -- (o4.center) -- cycle;
    \node at (4.75, 2) {$\obstacle$};
  \end{pgfonlayer}

  \begin{pgfonlayer}{background}
    \tikzstyle{every path}=[draw, line width=1.0pt, color=PineGreen]
    \draw (1, 1) -- (1, -1);
    \draw (1, 1.025) -- (3, 1.025);

    \node at (0.4, 0) {$\{q^3\}$};
    \node at (2, 1.4) {$\{q^3\}$};

    \tikzstyle{every path}=[draw, line width=1.0pt, color=WildStrawberry]
    \draw (3, 0.975) -- (1, 0.975);
    \draw (3, 1) -- (3, -1);
    \draw (3, 3) -- (1, 3);
    \draw (3, 3) -- (3, 5);

    \node at (3.6, 0) {$\{o^1\}$};
    \node at (2, 0.6) {$\{o^1\}$};
    \node at (3.6, 4) {$\{o^4\}$};
    \node at (2, 3.4) {$\{o^4\}$};
  \end{pgfonlayer}

  \begin{pgfonlayer}{background}
    \tikzstyle{every path}=[draw, line width=1.0pt, color=PineGreen!50!WildStrawberry]
    \draw (7, 3) -- (5, 5);
    \draw (9, 1) -- (11, -1);

    \node at (9.25, -0.5) {$\{o^3,q^6\}$};
    \node at (5.25,  3.5) {$\{q^6,o^3\}$};

    \tikzstyle{every path}=[draw, line width=1.0pt, color=WildStrawberry]
    \draw (7, 1) -- (9, 1);
    \draw (7, 1) -- (7, -1);
    \draw (7, 3) -- (9, 3);
    \draw (7, 3) -- (7, 5);

    \node at (6.0, 0) {$\{o^3,o^2\}$};
    \node at (8, 0.6) {$\{o^1,o^2\}$};
    \node at (8, 4.6) {$\{o^2,o^3\}$};
    \node at (8, 3.4) {$\{o^4,o^3\}$};
  \end{pgfonlayer}

  \begin{pgfonlayer}{foreground}
    \tikzstyle{every path}=[draw, line width=0.75pt, color=RoyalBlue]
    \draw[dotted] (5, -1.5) -. (5, 5.5);
  \end{pgfonlayer}

  \begin{pgfonlayer}{main}
    \tikzstyle{every path}=[draw, line width=1.0pt, color=black]
    \draw (q1.center) -- (q2.center) -- (q3.center) -- (q4.center) -- (q5.center) -- (q6.center) -- (q7.center) -- (q8.center) -- cycle;
    \draw (o1.center) -- (o2.center) -- (o3.center) -- (o4.center) -- cycle;
  \end{pgfonlayer}

\end{tikzpicture}
\caption{Here, Type~I and Type~II inflection segments are shown to the
  left and right of the dotted line, respectively, colored by their
  generating vertices. The anchor(s) added or removed when crossing
  each segment is also listed.}
  \label{fig:inflection_segments}
\end{figure}

Note that even though each inflection segment affects only one element
of the set of anchors, multiple inflection segments can physically
overlap. For instance, in \Cref{fig:anchors_1}, $\genericPoint$ lies
on a Type~I and also on a Type~II inflection segment: therefore, when
moving from $\genericPoint_1$ to $\genericPoint_2$, we have the
combined effect of both, and $o^1$ is replaced by $o^4$ in the set of
anchors.

\begin{lemma}[Bound on the number of inflection
  segments]
  \label{lem:bound_inflection_segments}
  One has $|\setInflectionSegments| \leqslant 4n_r + 2
  \binom{n_r}{2}$, where $n_r = |\reflexVertices{\freeSpace}|$. Thus,
  $\setInflectionPoints \subset \real^2$ has zero Lebesgue measure.
\end{lemma}

\begin{proof}
  Based on \Cref{def:inflection_segments}, we can have a maximum of
  $4$ inflection segments of Type~I per reflex vertex, and $2$
  inflection segments of Type~II per pair of reflex vertices.
  The bound follows by noting that there are $n_r$ reflex vertices and
  $\binom{n_r}{2}$ pairs of reflex vertices. Since each inflection
  segment is a line segment embedded in a 2D space, it is a measure
  zero set, and the finite union of zero-measure sets has a zero
  measure.
\end{proof}

This improves on the known bounds \cite{LG-RM-PR:92,HS-MZ:15-ajor}. In
practice, the number of inflection segments is much lower than the
bound in Lemma~\ref{lem:bound_inflection_segments}, as not all reflex
vertices are visible from every observer, not all four Type~I
inflection segments exist for every anchor, not all reflex vertices
are pairwise visible, etc.  We next establish the correspondence
between the set of all inflection segments and the set of all critical
points.

\begin{theorem}[The set of inflection segments represent all critical
  points]
  \label{thm:inflection_segments}
  The set of critical points is precisely the set of points on an
  inflection segment, that is,
  $\setInflectionPoints = \setCriticalPoints$.
\end{theorem}
\begin{proof}
  ($\setInflectionPoints \subset \setCriticalPoints$) Let
  $\genericPoint \in \setInflectionPoints$. We consider 2
  cases---\emph{Case~1:} $\genericPoint$ lies on a Type~I inflection
  segment. Then, there exists a reflex vertex $\reflexVertex$ such
  that $\genericPoint$ lies in $\boundary{\region_2 \union
  \region_4}$. In this case, for any $\epsilon > 0$, there exist
  points $\genericPoint', \genericPoint'' \in
  \oBall{\epsilon}{\genericPoint}$ such that $\genericPoint' \in
  \region_1 \union \region_3$ whereas $\genericPoint'' \in \region_2
  \union \region_4$, and hence, $\genericPoint \in
  \setCriticalPoints$.
  
  \emph{Case~2:} $\genericPoint$ lies on a Type~II inflection
  segment. Then, there exist two anchors
  $\reflexVertex, \reflexVertex'$ such that $\genericPoint$ lies on
  either of the projected
  ray---$\setRay{\reflexVertex}{\reflexVertex'} \intersection
  \freeSpace$ or
  $\setRay{\reflexVertex'}{\reflexVertex} \intersection
  \freeSpace$. In this case, for any $\epsilon$, there exist points
  $\genericPointQ_1, \genericPointQ_2 \in
  \oBall{\epsilon}{\genericPoint} \intersection \freeSpace$ such that
  either
  $\anchors{\genericPointQ_1} \symDiff \anchors{\genericPointQ_2}$
  equals $\{\reflexVertex\}$ or $\{\reflexVertex'\}$ and hence, $\genericPoint
  \in \setCriticalPoints$.

  ($\setCriticalPoints \subset \setInflectionPoints$) Let
  $\genericPoint \in \setCriticalPoints$. Then,
  $\genericPoint \in \freeSpace$ and, for any $\epsilon>0$, there
  exist
  $\genericPointQ_1, \genericPointQ_2 \in
  \oBall{\epsilon}{\genericPoint} \intersection \freeSpace,
  \anchors{\genericPointQ_1} \neq \anchors{\genericPointQ_2}$. Without
  loss of generality, let $\reflexVertex$ be an anchor to
  $\genericPointQ_1$, but not to $\genericPointQ_2$. Using
  \Cref{lem:characterization_anchor}, we deduce that
  $\reflexVertex \in \visibilityPolygon{\genericPointQ_1}$ and
  $\genericPointQ_1 \in \region_2(\reflexVertex) \union
  \region_4(\reflexVertex)$. We then consider 2 cases: either (i)
  $\reflexVertex \in \visibilityPolygon{\genericPointQ_2}$ and
  $\genericPointQ_2 \in \region_1(\reflexVertex) \union
  \region_3(\reflexVertex)$ or (ii)
  $\reflexVertex \not \in \visibilityPolygon{\genericPointQ_2}$. In
  case (i), $\genericPointQ_1$ and $\genericPointQ_2$ lie on opposite
  sides of a Type~I inflection segment of $\reflexVertex$. As there
  exist $\genericPointQ_1, \genericPointQ_2$ for all $\epsilon >0$,
  $\genericPoint$ has to lie on a Type~I inflection segment. Due to
  symmetry of line-of-sight, case (ii) can also be written as
  $\genericPointQ_1 \in \visibilityPolygon{\reflexVertex}$, but
  $\genericPointQ_2 \not \in \visibilityPolygon{\reflexVertex}$. So,
  there must exist an anchor $\anchor$ for $\reflexVertex$ that
  obstructs $\genericPointQ_2$, but not $\genericPointQ_1$, from
  $\reflexVertex$. Then, $\genericPointQ_1$ and $\genericPointQ_2$ lie
  on opposite sides of $\setRay{\anchor}{\reflexVertex}$, the
  projected ray from $\anchor$ which points away from
  $\reflexVertex$. Again, as $\genericPointQ_1, \genericPointQ_2$
  exist for all $\epsilon >0$, $\genericPoint$ has to lie on a Type~II
  inflection segment. Thus, in either case, we deduce
  $\genericPoint \in \setInflectionPoints$.
\end{proof}

From \Cref{thm:inflection_segments}, the set of critical points can be
represented via the set of inflection segments. Therefore, we turn our
attention to the reduced free space $\reducedFreeSpace := \freeSpace
\setminus \setInflectionPoints$ as a topological subspace of
$\real^2$. Two points $\genericPoint, \genericPointQ$ are
\emph{path-connected} in $\reducedFreeSpace$ if there is a path
between them not crossing any inflection segment. Therefore, both
points $x$ and $y$ have the same set of anchors. The set of
path-connected components of $\reducedFreeSpace$ define a partition
$\partition$ of relatively\footnote{The components are not open as the
boundary of $\freeSpace$ is a part of these components.}  open subsets
of~$\reducedFreeSpace$.

\begin{lemma}[Partition of the free space]
  \label{lem:partition}
  Let $\partition := \left\{ \componentTerrain_\beta \right\}_{\beta
  \in B}$ be the partition of the reduced free space
  $\reducedFreeSpace$, where each $\verythinspace
  \componentTerrain_\beta$ is a path-connected component of
  $\reducedFreeSpace$. Then, each $\verythinspace
  \componentTerrain_\beta \in \partition$ is convex.
\end{lemma}
\begin{proof}
  Suppose that there is a component $\componentTerrain$ of
  $\partition$ that is not convex.
  Then, there is a vertex $w$ of $\componentTerrain$ that is reflex in
  $\componentTerrain$. This implies that $w \in
  \reflexVertices{\freeSpace}$, as otherwise $w$ would subtend an
  angle of less than or equal to $\pi$, and hence, will not be reflex.
  Now, let $w_1$, $w$, $w_2$ be consecutive vertices of the
  environment or an obstacle. Then, by definition, $\setRay{w}{w_1},
  \setRay{w}{w_2} \subset \setInflectionSegments$. But
  $\setRay{w}{w_1}, \setRay{w}{w_2}, \lineSegment{w}{w_1},
  \lineSegment{w}{w_2}$ lie on two intersecting lines at $w$, which
  must form an angle strictly less than $\pi$. Hence, $w$ cannot be
  reflex in $\componentTerrain$, contradicting our assumption.
\end{proof}

From \Cref{thm:inflection_segments}, $\setInflectionPoints$ captures
all the critical points, and so the points in each path-connected
component of $\reducedFreeSpace$ are \emph{not} critical. Applying
this to the partition $\partition$ defined in \Cref{lem:partition},
the following characterization for the persistence of an anchor
emerges.

\begin{corollary}[Constancy of set of anchors within a path-connected
  component]
  \label{cor:constancy_anchors}
  Let $\componentTerrain \in \partition$. Then, for all $\genericPoint
  \in \componentTerrain$, the set of anchors, $\anchors{\genericPoint}
  \subseteq \reflexVertices{\freeSpace}$ remains constant. \oprocend
\end{corollary}
Therefore, $\partition$ is the minimal partition such that none of its
component has any critical points in its interior, an improvement over
the \emph{visibility cell decomposition} \cite{LG-RM-PR:92}, the
\emph{back diagonal partition} \cite{HS-MZ:15-ajor}, and the
\emph{inflection segment partition} in \cite{AG-JC-FB:06-sicon}. 

The result also enables us to employ $\anchors{\componentTerrain}$ for
the set of anchors of a path-connected component $\componentTerrain$.
Moreover, it shows that if a path exists between two points that does
not cross any inflection segments, then the points share the same set
of anchors. The next result considers the case when there is a path
between two points on an inflection segment that does not go out of
the inflection segment or cross any other inflection segments. 

\begin{prop}[Constancy of set of anchors on an inflection segment]
  \label{prop:constancy_inflection}
  Let $\inflectionSegment \in \setInflectionSegments$. If two points
  $\genericPoint, \genericPointQ \in \inflectionSegment$ are
  path-connected within $\inflectionSegment$ without crossing any
  other inflection segment, then they share the same set of anchors.
\end{prop}
\begin{proof}
  We consider two cases.  Suppose that the segment
  $\lineSegment{\genericPoint}{\genericPointQ}$ lies on the boundary
  of $\freeSpace$. Then, there is exactly one component of
  $\partition$, say $\componentTerrain$, such that, for sufficiently
  small $\epsilon > 0$, we have
  $\oBall{\epsilon}{\genericPoint} \intersection \freeSpace,
  \oBall{\epsilon}{\genericPointQ} \intersection \freeSpace \subset
  \closure{\componentTerrain}$. The points on the inflection segment
  thus have exactly one additional anchor beyond
  $\anchors{\componentTerrain}$, corresponding to the reflex vertex
  generating~$\inflectionSegment$.  In the other case, let
  $\componentTerrain_1, \componentTerrain_2 \in \partition$ be the two
  components of $\partition$ such that, for sufficiently small
  $\epsilon > 0$, $\oBall{\epsilon}{\genericPoint},
  \oBall{\epsilon}{\genericPointQ} \subset
  \closure{\componentTerrain_1} \cup \closure{\componentTerrain_2}$.
  Then, the anchors of both $\genericPoint$ and $\genericPointQ$ are
  given by $\anchors{\componentTerrain_1} \cup
  \anchors{\componentTerrain_2}$.
\end{proof}

\begin{figure}[htb]
    \centering
    \pgfdeclarelayer{background}
\pgfdeclarelayer{foreground}
\pgfsetlayers{background,main,foreground}
\begin{tikzpicture}[scale=0.55]
  \tikzset{
    style 0/.style={inner sep=0pt, minimum size=3pt},
    style A/.style={shape=circle, fill=PineGreen, style 0},
    style B/.style={shape=circle, fill=WildStrawberry, style 0},
    style C/.style={shape=circle, fill=Blue!50!RoyalBlue, style 0},
    font={\fontsize{8pt}{9.6}\selectfont}
  }

  \begin{pgfonlayer}{foreground}
    \node[draw, style A] (q1) at (-1, -1) {};
    \node[draw, style A] (q2) at (11, -1) {};
    \node[draw, style A] (q3) at (11,  1) {};
    \node[draw, style A] (q4) at ( 9,  1) {};
    \node[draw, style A] (q5) at ( 9,  5) {};
    \node[draw, style A] (q6) at ( 1,  5) {};
    \node[draw, style A] (q7) at ( 1,  1) {};
    \node[draw, style A] (q8) at (-1,  1) {};

    \node[text=PineGreen] at ([shift={(225:0.7)}]q1) {$q^1$};
    \node[text=PineGreen] at ([shift={(315:0.5)}]q2) {$q^8$};
    \node[text=PineGreen] at ([shift={( 45:0.7)}]q3) {$q^7$};
    \node[text=PineGreen] at ([shift={( 45:0.7)}]q4) {$q^6$};
    \node[text=PineGreen] at ([shift={( 45:0.7)}]q5) {$q^5$};
    \node[text=PineGreen] at ([shift={(135:0.5)}]q6) {$q^4$};
    \node[text=PineGreen] at ([shift={(135:0.7)}]q7) {$q^3$};
    \node[text=PineGreen] at ([shift={(135:0.5)}]q8) {$q^2$};
  \end{pgfonlayer}

  \begin{pgfonlayer}{background}
    \tikzstyle{every path}=[draw, line width=1.0pt, color=black]
    \draw[fill = gray!20] (q1.center) -- (q2.center) -- (q3.center) -- (q4.center) -- (q5.center) -- (q6.center) -- (q7.center) -- (q8.center) -- cycle;
  \end{pgfonlayer}

  \begin{pgfonlayer}{foreground}
    \node[draw, style B] (o1) at (3, 1) {};
    \node[draw, style B] (o2) at (7, 1) {};
    \node[draw, style B] (o3) at (7, 3) {};
    \node[draw, style B] (o4) at (3, 3) {};

    \node[text=WildStrawberry] at ([shift={( 45:0.6)}]o1) {$o^1$};
    \node[text=WildStrawberry] at ([shift={(135:0.6)}]o2) {$o^2$};
    \node[text=WildStrawberry] at ([shift={(225:0.6)}]o3) {$o^3$};
    \node[text=WildStrawberry] at ([shift={(315:0.6)}]o4) {$o^4$};
  \end{pgfonlayer}

  \begin{pgfonlayer}{main}
    \tikzstyle{every path}=[draw, line width=1.0pt, color=black]
    \draw[fill = white] (o1.center) -- (o2.center) -- (o3.center) -- (o4.center) -- cycle;
    \node at (5, 2) {$\obstacle$};
  \end{pgfonlayer}

  \begin{pgfonlayer}{background}
    \tikzstyle{every path}=[draw, line width=1.0pt, color=Violet]
    \draw (1, 1) -- (-1, -1);
    \draw (3, 3) -- (5, 5);
    \draw (7, 3) -- (5, 5);
    \draw (9, 1) -- (11, -1);
    \draw (3, 1) -- (1, 1);
    \draw (3, 1) -- (3, -1);
    \draw (3, 3) -- (1, 3);
    \draw (3, 3) -- (3, 5);
    \draw (7, 1) -- (9, 1);
    \draw (7, 1) -- (7, -1);
    \draw (7, 3) -- (9, 3);
    \draw (7, 3) -- (7, 5);
    \draw (1, 1) -- (1, -1);
    \draw (9, 1) -- (9, -1);
  \end{pgfonlayer}

  \begin{pgfonlayer}{main}
    \tikzstyle{every node}=[color=Violet, scale = 0.75]
    \node at (5, 0) {$\{ o^1, o^2 \}$};
    \node at (5, 3.75) {$\{ o^4, o^3 \}$};
    \node at (2, 0) {$\{ o^4, o^2 \}$};
    \node at (8, 0) {$\{ o^3, o^1 \}$};
    \node at (2, 2) {$\{ o^1, o^4, q^3 \}$};
    \node at (8, 2) {$\{ o^3, o^2, q^6 \}$};
    \node at (2, 4) {$\{ o^3, o^1, q^3 \}$};
    \node at (8, 4) {$\{ o^4, o^2, q^6 \}$};
    \node at (-0.25, 0.65) {$\{ o^2, q^3 \}$};
    \node at (10.25, 0.65) {$\{ o^1, q^6 \}$};

    \tikzstyle{every node}=[color=Violet, scale = 0.6]
    \node at (3.78, 4.7) {$\{ o^4, o^3, q^3 \}$};
    \node at (6.22, 4.7) {$\{ o^4, o^3, q^6 \}$};
    \node at (0.22, -0.7) {$\{ o^4, o^2, q^3 \}$};
    \node at (9.78, -0.7) {$\{ o^3, o^1, q^6 \}$};
    
  \end{pgfonlayer}

  \begin{pgfonlayer}{main}
    \tikzstyle{every path}=[draw, line width=1.0pt, color=black]
    \draw (q1.center) -- (q2.center) -- (q3.center) -- (q4.center) -- (q5.center) -- (q6.center) -- (q7.center) -- (q8.center) -- cycle;
    \draw (o1.center) -- (o2.center) -- (o3.center) -- (o4.center) -- cycle;
  \end{pgfonlayer}
\end{tikzpicture}
\caption{The unique partition $\partition$ of
  $\reducedFreeSpace = \freeSpace \setminus \setInflectionPoints$ and
  the set of anchors $\anchors{\componentTerrain}$ for each component
  $\componentTerrain$ of $\partition$.}
    \label{fig:partition}
\end{figure}

\begin{remark}[Generalization to non-polygonal free spaces]
  \label{rem:non_polygonal}
  Although we have kept the presentation of this paper focused on
  polygonal environments for ease of exposition, the notions
  introduced in this section extend naturally to non-polygonal
  environments where obstacles are represented by weakly simple
  curves~\cite{HC-JE-CX:15-siam}.  In this case, the role of reflex
  vertices is played by the points of non-differentiability on the
  curves for the environment and the obstacles. Type~I inflection
  segments are determined by the tangent directions at these points.
  Similarly, Type~II inflection segments arise from lines tangent to
  the curves at more than one point, with corresponding tangent points
  serving as anchors. \oprocend
\end{remark}


\section{How do Visibility Polygon and Visibility Metric Change as the
Observer Moves?\texorpdfstring{\nopunct}{}}\label{sec:derivative}

Here, we build on the results from the previous section to address the
remaining questions regarding the visibility polygon and its
associated metric. We begin by exploring their regularity properties.


\subsection{Regularity Properties of Visibility Polygons and
Metrics}\label{sub:regularity_properties}

We address~\labelcref{enum:regularity_properties} here. We begin by
computing the directional derivative of the visibility polygon's area.
These expressions allow us to characterize the regularity properties
of the visibility polygon, which in turn informs the regularity of the
visibility area and the visibility metric.

Consider the function $\visibilityArea(\genericPoint) := \|
\visibilityPolygon{\genericPoint}\|_\mu$ which represents the area of
the visibility polygon for an observer at $\genericPoint \in
\freeSpace$.  The following result characterizes the directional
derivative of the visibility area $\visibilityArea$ over $\freeSpace
\setminus \reflexVertices{\freeSpace}$.

\begin{theorem}[Directional derivative of visibility area at a point]
  \label{thm:directional_derivative_area}
  For an observer at $\genericPoint \in \freeSpace \setminus
  \reflexVertices{\freeSpace}$, suppose there exist a direction
  $\direction \in \tangentCone{\freeSpace}{\genericPoint}$, a set of
  anchors $\mathrm{Ve}_a \subset \reflexVertices{\freeSpace}$, and a
  sufficiently small $\bar{\delta} > 0$ such that $\genericPoint +
  \delta \direction \in \freeSpace$ and $\anchors{\genericPoint +
  \delta \direction} = \mathrm{Ve}_a$ for all $0 < \delta <
  \bar{\delta}$. Then, the directional derivative of the visibility
  area at $\genericPoint$ in direction $\direction$ is given by
  \begin{align}
    \label{eqn:directional_derivative_area}
    \rightDirectionalDerivative{\visibilityArea}(\genericPoint;\direction)
    = \sum_{\anchor \in 
    \mathrm{Ve}_a}
    \frac{\radiusRay{\anchor}{\genericPoint}^2}{\|\genericPoint - 
    \anchor\|_2} \left[ \frac{\hat{k}}{2} \cdot
    \widehat{(\genericPoint - \anchor)} \times \direction \right],
  \end{align}
  where $\hat{k}$ is the vector perpendicular to the plane, and
  $\times$ and~$\cdot$ denote the cross and dot products,
  respectively.
\end{theorem}
\begin{proof}
  From \Cref{lem:visibility_construction}, we know that the changes in
  the visibility polygon are entirely determined by anchors and their
  projected rays.  Thus, the first step in computing the directional
  derivative of $\visibilityArea$ is to evaluate the impact in the
  area of the visibility polygon associated to changes in a single
  anchor~$\anchor$.

  Given $\genericPoint \in \freeSpace$ and a direction $\direction$
  such that $\genericPoint + \delta \direction \in \freeSpace$ for
  sufficiently small $\delta >0$, let $u_-$ and $u_+$ denote the
  intersection of the projected rays $\setRay{\anchor}{\genericPoint}$
  and $\setRay{\anchor}{\genericPoint + \delta \direction}$ with
  $\boundary{\freeSpace}$, respectively. Then, the change in area of
  the visibility polygon from $\genericPoint$ to $\genericPoint +
  \delta \direction$ is given by
  \begin{align}
    \Delta \visibilityArea(\genericPoint; \direction, \anchor)
    & = \frac{1}{2} \left[ \hat{k} \cdot \left( (u_- - \anchor) \times
    (u_+ - \anchor) \right) \right] \nonumber
    \\
    & = \frac{c_1}{2 c_2} \left[ \hat{k} \cdot (\genericPoint -
    \anchor) \times (\genericPoint + \delta \direction - \anchor)
    \right] \nonumber
    \\ 
    \Rightarrow \Delta \visibilityArea(\genericPoint; \direction,
    \anchor) & = \frac{\delta c_1}{2 c_2} \left[ \hat{k} \cdot
    (\genericPoint - \anchor) \times \direction
    \right], \label{eqn:area_change_single_anchor_comment}  
  \end{align}
  where $c_1 = \radiusRay{\anchor}{\genericPoint} \,
  \radiusRay{\anchor}{\genericPoint + \delta \direction}$ and $c_2 =
  \|\genericPoint - \anchor\|_2 \, \|\genericPoint + \delta \direction
  - \anchor\|_2$.  The directional derivative of the area of the
  visibility polygon at $\genericPoint$ in the direction $\nu$ wrt
  $\anchor$ is given by
  $\rightDirectionalDerivative{\visibilityArea}(\genericPoint;
  \direction, \anchor) = \lim_{\delta \rightarrow 0^+} \symDiff
  \visibilityArea(\genericPoint; \direction, \anchor) / \delta$.

  We claim that the change in area due to one anchor is independent of
  the change in area due to other anchors. The proof is in two
  parts---for $\genericPoint \not \in \setInflectionPoints$, at most
  one anchor lies along each ray, and there exists $\delta$
  sufficiently small such that none of the triangles $\triangle u_-
  \anchor u_+$ intersect each other.  For $\genericPoint \in
  \setInflectionPoints$, if $\direction$ is not along any inflection
  segments, of the two anchors generating each of them, only one is in
  $\mathrm{Ve}_a$, and a similar argument ensues. Finally, if
  $\direction$ points along an inflection segment, then
  $(\genericPoint - \anchor) \times \direction = 0$ for both anchors
  generating the inflection segment, and the remaining changes in
  areas are independent of each other by the above.

  Consider an arbitrary ray emanating from the observer. By rotating
  this ray in a clockwise direction, it sweeps out the entire
  visibility polygon. As the ray rotates, each anchor that becomes
  visible contributes a small, incremental change in the area of the
  polygon. In other words, the change in the area as a function of the
  ray's orientation is built from the individual contributions of all
  anchors that are visible to the observer. Therefore, the overall
  directional derivative of the visibility area is given by the signed
  sum of the directional derivatives corresponding to the individual
  anchors---each of which is expressed by equation
  \cref{eqn:area_change_single_anchor_comment}. Summing these
  contributions over all anchors in $\mathrm{Ve}_a$ yields equation
  \cref{eqn:directional_derivative_area}.
\end{proof}

This result is analogous to \cite[Section 2]{AG-JC-FB:06-sicon} but
provides an alternative proof technique that is generalizable to other
visibility metrics. The next result provides the expression for the
directional derivative of the visibility metric $\visibilityMetric$
over $\freeSpace \setminus \reflexVertices{\freeSpace}$.

\begin{theorem}[Directional derivative of visibility metric at a point]
  \label{thm:directional_derivative_metric}
  Let an observer at $\genericPoint$ satisfy the conditions of
  \Cref{thm:directional_derivative_area} with direction $\direction$
  and set of anchors $\mathrm{Ve}_a$. Then, the directional derivative
  of the visibility metric at $\genericPoint$ in direction
  $\direction$ is given by
  \begin{align}
    \label{eqn:directional_derivative_metric}
    \rightDirectionalDerivative{\visibilityMetric}(\genericPoint;\direction)
    = \sum_{\anchor \in 
    \mathrm{Ve}_a}  c_{\anchor}(\genericPoint; \direction)
    \frac{\radiusRay{\anchor}{\genericPoint}^2}{\|\genericPoint - 
    \anchor\|_2}  \left[ \frac{\hat{k}}{2}  \cdot 
    \widehat{(\genericPoint - \anchor)}  \times  \direction \right],
  \end{align}
  where $\hat{k}$ is the vector perpendicular to the plane, $\times$
  and~$\cdot$ denote the cross and dot products, respectively, and
  $c_{\anchor}(\genericPoint; \direction)$ represents the limiting
  fraction of the change in area due to $\anchor$ that intersects with
  $\advDomain$.
\end{theorem}
\begin{proof}
  With the same conventions and reasoning as in the proof of
  \Cref{thm:directional_derivative_area}, we obtain
  \begin{align*}
    \Delta \visibilityMetric(\genericPoint; \direction, \anchor) =
    \frac{\delta c_1}{2 c_2} c_{\anchor}(\genericPoint; \direction,
    \delta) \left[ \hat{k} \cdot (\genericPoint - \anchor) \times
    \direction \right], 
  \end{align*}
  where $c_{\anchor}(\genericPoint; \direction, \delta)$ is the
  fraction of the change in area due to $\anchor$ that intersects with
  $\advDomain$, that is,
  $c_{\anchor}(\genericPoint; \direction, \delta) = \|\advDomain \cap
  \triangle u_- \anchor u_+\|_\mu / \|\triangle u_- \anchor u_+\|_\mu
  \in [0,1]$.

  By definition, the directional derivative of the visibility metric
  at $\genericPoint$ in the direction $\nu$ with respect to $\anchor$
  is $\rightDirectionalDerivative{\visibilityMetric}(\genericPoint;
  \direction, \anchor) = \lim_{\delta \rightarrow 0^+} \symDiff
  \visibilityMetric(\genericPoint; \direction, \anchor) / \delta$. Let
  $u_0, \dots, u_N$ be equidistant points on the edge containing $u_-$
  and $u_+$. The triangles $\triangle u_i \anchor u_{i+1}$ are
  disjoint and cover all $\triangle u_- \anchor u_+$ for $\delta < 1$.
  Furthermore, by choosing $N$ sufficiently large, $\|\triangle u_i
  \anchor u_{i+1}\|_\mu$ can be made arbitrarily small. This allows us
  to conclude that the limit
  \smash{$c_{\anchor}(\genericPoint; \direction) := \lim_{\delta \to
  0^+} c_{\anchor}(\genericPoint; \direction, \delta)$}
  exists.  The statement now follows by summing over all anchors in
  $\mathrm{Ve}_a$, using the same independence argument as in
  \Cref{thm:directional_derivative_area}.
\end{proof}

In general, it is possible that $\mathrm{Ve}_a \neq
\anchors{\genericPoint}$ in
\Cref{thm:directional_derivative_area,thm:directional_derivative_metric}.
This happens when there is one more/less anchor at $\genericPoint$,
but it is not involved in the change of area from the geometry (for
instance, when $\genericPoint \in \setInflectionPoints$ and
$\direction$ points into a path-connected component $\componentTerrain
\in \partition$). Also, note that from
\cref{eqn:directional_derivative_area} and
\cref{eqn:directional_derivative_metric}, one can see that these
directional derivatives can potentially blow up as $\genericPoint
\rightarrow \anchor$ for any $\anchor \in
\reflexVertices{\freeSpace}$.  This is consistent with the fact that
the visibility polygon changes drastically near anchors.

The combination of \Cref{thm:directional_derivative_area} and
\Cref{thm:directional_derivative_metric} with
\Cref{cor:constancy_anchors} yields the following.

\begin{corollary}[Directional derivative of visibility area and metric
  over a path-connected component]
  \label{cor:directional_derivative_component_area-metric}
  The visibility area $\visibilityArea$ (resp.~visibility metric
  $\visibilityMetric$) is analytic (resp.~continuously differentiable)
  over each path-connected component $\componentTerrain$ of
  $\reducedFreeSpace$. The expression for its directional derivative
  is given by \cref{eqn:directional_derivative_area} (resp.
  \cref{eqn:directional_derivative_metric}) with $\mathrm{Ve}_a :=
  \nobreak \anchors{\componentTerrain}$. \oprocend
\end{corollary}

From this result, we deduce that $ \visibilityArea$ is constant over
$\componentTerrain$ if $|\anchors{\componentTerrain}| = 0$.
Furthermore, when $|\anchors{\componentTerrain}| = \{ \anchor\}$, the
direction of steepest descent/ascent is perpendicular
to~$\genericPoint - \anchor$.
Although $\visibilityMetric$ is continuously differentiable, it is not
analytic due to presence of the fraction $c_{\anchor}(\genericPoint;
\direction)$ in the expression for
$\rightDirectionalDerivative{\visibilityMetric}(\genericPoint;\direction)$,
which is continuous, but not necessarily analytic.

Next, we consider the case when the observer is located on an
inflection segment. 

\begin{corollary}[Directional derivative of the visibility area and
  metric on an inflection segment]
  \label{cor:directional_derivative_inflection_area-metric}
  Given an inflection segment $\inflectionSegment \in
  \setInflectionSegments$, the function $\visibilityArea$ is analytic
  (resp.  $\visibilityMetric$ is continuously differentiable) over
  each path-connected component of $\inflectionSegment \setminus \{
  \genericPoint \in \inflectionSegment \mid \exists
  \inflectionSegment' \in \setInflectionSegments \setminus
  \{\inflectionSegment\}, \genericPoint \in \inflectionSegment' \}$.
  For an observer at $\genericPoint \in \setInflectionPoints$,
  \begin{itemize}
  \item if
    $\{\genericPoint + \delta \direction\} \in \componentTerrain$, for
    all $0 < \delta < \epsilon$, then
    \cref{eqn:directional_derivative_area} (resp.
    \cref{eqn:directional_derivative_metric}) holds with
    $\mathrm{Ve}_a = \anchors{\componentTerrain}$;
  \item if
    $\{\genericPoint + \delta \direction\} \in \inflectionSegment$ for
    all $0 < \delta < \epsilon$ and $\epsilon$ is chosen such that
    $\oBall{\epsilon}{\genericPoint}$ does not intersect with any other
    inflection segment, then \cref{eqn:directional_derivative_area}
    (resp. \cref{eqn:directional_derivative_metric}) holds with
    $\mathrm{Ve}_a = \cap_{\delta > 0} \anchors{\genericPoint + \delta
    \direction}$. \oprocend
  \end{itemize}
\end{corollary}

We next build on these results to establish the regularity of the
visibility polygon in terms of its $\mu$-local Lipschitzness.

\begin{prop}[\texorpdfstring{$\mu$}{µ}-directional derivative of
  visibility polygon]
  \label{prop:directional_derivative_mu}
  For an observer at $\genericPoint \in \freeSpace \setminus
  \reflexVertices{\freeSpace}$, let $\direction \in
  \tangentCone{\freeSpace}{\genericPoint}$ be a direction and
  $\mathrm{Ve}_a \subset \reflexVertices{\freeSpace}$ such that
  $\genericPoint + \delta \direction \in \freeSpace$ and
  $\mathrm{Ve}_a = \anchors{\genericPoint + \delta \direction}$ for
  all $0 < \delta < \epsilon$ for a sufficiently small $\epsilon > 0$.
  Then, the $\mu$-directional derivative of the visibility polygon
  $\visibilityPolygonSymbol$ at $\genericPoint$ in direction
  $\direction$ is given by
  \begin{align}
    \label{eqn:directional_derivative_mu}
    \rightMuDirectionalDerivative{\visibilityPolygonSymbol}(\genericPoint;
    \direction) = 
    \sum_{\anchor \in \mathrm{Ve}_a}
    \frac{\radiusRay{\anchor}{\genericPoint}^2}{\|\genericPoint - 
    \anchor\|_2} \left\| \widehat{(\genericPoint - \anchor)} \times
    \direction \right\|. 
  \end{align}
  Furthermore,
  $|\rightDirectionalDerivative{\visibilityArea}(\genericPoint;\direction)|,
  |\rightDirectionalDerivative{\visibilityMetric}(\genericPoint;\direction)|
  \leqslant
  \rightMuDirectionalDerivative{\visibilityPolygonSymbol}(\genericPoint;
  \direction)$ for any~$\advDomain$.
\end{prop}
\begin{proof}
  The proof mirrors that of \Cref{thm:directional_derivative_area},
  except that the $\mu$-directional derivative is the \emph{unsigned}
  sum of the directional derivatives due to individual anchors in
  \cref{eqn:area_change_single_anchor_comment},
  giving~\cref{eqn:directional_derivative_mu}.  The second part
  follows from the fact that $|\sum a_i | \leqslant \sum |a_i|$ with
  $a_i = \rightDirectionalDerivative{\visibilityArea}(\genericPoint;
  \direction, \anchor^i)$.
\end{proof}

We leverage this result to characterize the regularity of the
visibility polygon and its intersection in terms of $\mu$-local
Lipschitzness.

\begin{theorem}[Regularity of the visibility polygon]
  \label{thm:local_lipschitz_mu}
  Let $\advDomain \subset \freeSpace$ be non-empty. Then, the
  visibility polygon $\visibilityPolygonSymbol$ and the intersection
  $\visibilityPolygonSymbol \cap \advDomain$ are $\mu$-locally
  Lipschitz over $\freeSpace \setminus \reflexVertices{\freeSpace}$.
\end{theorem}
\begin{proof}
  Analogous to the proof of \Cref{thm:directional_derivative_area},
  consider two observers at $\genericPointR, \genericPointR' \in
  \freeSpace \setminus \reflexVertices{\freeSpace}$ that belong to the
  same path-connected component $\componentTerrain \in \partition$ of
  $\reducedFreeSpace$. Then, the distance between the visibility
  polygons at these two points satisfies
  \begin{align}\label{eqn:local_lipschitz_mu}
    \|
    \visibilityPolygonSymbol(\genericPointR) \symDiff
    \visibilityPolygonSymbol(\genericPointR') \|_\mu \leqslant \|
    \genericPointR - \genericPointR' \| \sum_{\anchor \in
    \anchors{\componentTerrain}}
    \frac{\radiusRay{\anchor}{\genericPointR} \,
    \radiusRay{\anchor}{\genericPointR'}}{2 \|\genericPointR' -
    \anchor\|} \left| (\widehat{\genericPoint - \anchor}) \times
    (\widehat{\genericPointR' - \genericPointR})
    \right|.
  \end{align}
  Note that as $\genericPointR' \not \in \reflexVertices{\freeSpace}$,
  the denominator $\|\genericPointR' - \anchor\|$ is bounded away from
  zero, and the sum is finite.
  Now, for all points in the free space that are at least $2\delta >
  0$ away from the reflex vertices, the sum can be upper bounded by
  $L_{\componentTerrain} := |\reflexVertices{\componentTerrain}|
  D^2/4\delta$, where $D$ is the diameter of the environment.

  Let $\genericPoint \in \freeSpace \setminus
  \reflexVertices{\freeSpace}$ and $\delta > 0 $ sufficiently small
  such that there are no reflex vertices in
  $\oBall{2\delta}{\genericPoint}$. Consider two points
  $\genericPointQ_1, \genericPointQ_2 \in
  \oBall{\delta}{\genericPoint}$.  Then, the line segment
  $\lineSegment{\genericPointQ_1}{\genericPointQ_2}$ strictly
  crosses\footnotemark~the set $\setInflectionPoints$ of critical
  points at most finitely many times.\footnotetext{If
  $\genericPointQ_1, \genericPointQ_2$ lie on the same inflection
  segment, then they do not strictly cross the inflection segment.}
  Let $\{\bar{\genericPoint}_i\}_{i=0}^K$ be the $K \geqslant 2$
  distinct points on
  $\lineSegment{\genericPointQ_1}{\genericPointQ_2}$ such that
  $\bar{\genericPoint}_0 = \genericPointQ_1$, $\bar{\genericPoint}_K =
  \genericPointQ_2$, and $\bar{\genericPoint}_i \in
  \setInflectionPoints$ for $0 < i < K$, in order from
  $\genericPointQ_1$ to $\genericPointQ_2$.  The line segments
  $\lineSegment{\bar{\genericPoint}_j}{\bar{\genericPoint}_{j+1}}$
  each lie in a different connected component, say
  $\componentTerrain_j \in \partition$, for $0 \leqslant j < K$. Using
  the triangle inequality for the metric and iteratively applying
  \cref{eqn:local_lipschitz_mu} to each of these segments, we have
  \begin{align*}
    & \| \visibilityPolygonSymbol(\genericPointQ_1) \symDiff
      \visibilityPolygonSymbol(\genericPointQ_2) \|_\mu
      \leqslant
      \sum_{j=0}^{K-1} \| \visibilityPolygonSymbol(\bar{\genericPoint}_j)
      \symDiff \visibilityPolygonSymbol(\bar{\genericPoint}_{j+1})
      \|_\mu \leqslant \sum_{j = 0}^{K-1} L_{\componentTerrain_j} 
      \| \bar{\genericPoint}_j - \bar{\genericPoint}_{j+1} \|_\mu \leqslant L_\genericPoint \| \genericPointQ_1 -
      \genericPointQ_2 \|_\mu,
  \end{align*}
  where $L_\genericPoint = \max_{0 \leqslant j < K}
  L_{\componentTerrain_j}$, establishing the $\mu$-local Lipschitzness
  of $\visibilityPolygonSymbol$ at~$x$.  Finally, the $\mu$-local
  Lipschitzness of $\visibilityPolygonSymbol \cap \advDomain$ follows
  from \Cref{lem:intersection}
  (note that $\visibilityPolygonSymbol \cap \advDomain$
  has the same Lipschitz constant as $\visibilityPolygonSymbol$). 
\end{proof}

Given the definition of $\visibilityArea$ and $\visibilityMetric$,
\Cref{thm:local_lipschitz_mu} combined with \Cref{lem:mu_LL} yields
the following result.

\begin{corollary}[Regularity of the visibility metric]
  \label{cor:local_lipschitz}
  The functions $\visibilityArea$ and $\visibilityMetric$ are locally
  Lipschitz over $\freeSpace \setminus \reflexVertices{\freeSpace}$.
  \oprocend
\end{corollary}

Finally, the expressions for the generalized gradients of the
visibility area and metric follow from the definition (the functions
are continuously differentiable everywhere except on
$\setCriticalPoints$ by
\Cref{cor:directional_derivative_component_area-metric}).

\begin{corollary}[Generalized gradients of visibility area and metric]
  \label{prop:generalized_gradient}
  The generalized gradient of $f \in
  \{\visibilityArea,\visibilityMetric\}$ at $x \in \freeSpace
  \setminus \reflexVertices{\freeSpace}$ is
  \begin{align}\label{eqn:generalized_gradient_area} 
    \partial f(\genericPoint)
    & = \convexHull{\nabla
      f(\genericPoint')
      \mid \genericPoint' \to
      \genericPoint,
      \genericPoint' \in
      \freeSpace \setminus
      \setCriticalPoints}, 
  \end{align}
  where $\nabla f = [\rightDirectionalDerivative{f}(\genericPoint;
  e_1), \rightDirectionalDerivative{f}(\genericPoint; e_2)]$ and
  $\{e_1, e_2\}$ denotes the standard basis of $\real^2$. \oprocend
\end{corollary}


\subsection{Limited Range and Field of
  View}\label{sub:limited_range_limited_field_of_view}

We next turn our attention to \labelcref{enum:extension} and consider
the cases with limited range and limited field of view, cf.
\cite{CT:02-cg,CT:00-cg,BG-ST-GG:06}. For these cases, here we
characterize the regularity properties of the visibility area and
metric. We first show that the visibility with limited range is
$\mu$-locally Lipschitz.

\begin{theorem}[Regularity properties of visibility polygon with
  limited range]
  \label{thm:smoothness_limited_range}
  Let $\advDomain \subset \freeSpace$ be non-empty and let $R > 0$ be
  the sensor range. Then, the visibility polygon with limited range
  $\genericPoint \mapsto \visibilityPolygon{\genericPoint}
  \intersection \advDomain \intersection \oBall{R}{\genericPoint}$ is
  $\mu$-locally Lipschitz over
  $\freeSpace \setminus \reflexVertices{\freeSpace}$. Furthermore,
  $\genericPoint \mapsto
  \visibilityMetric_{R}(\genericPoint) :=
  \|\visibilityPolygon{\genericPoint} \intersection \advDomain
  \intersection \oBall{R}{\genericPoint}\|_\mu$ is locally Lipschitz over
  $\freeSpace \setminus \reflexVertices{\freeSpace}$.
\end{theorem}
\begin{proof}
  Note that
  $\|\oBall{R}{\genericPointQ} \symDiff
  \oBall{R}{\genericPointQ'}\|_\mu$ only depends on the distance
  $\|\genericPointQ - \genericPointQ'\|$. Let $f_R(\distanceCenters)$
  be the area of the symmetric difference of two $R$-balls with
  centers $\distanceCenters$ apart. The expression for $f_R$ for
  $\distanceCenters \in [0, 2R]$ is given by
  \begin{align}
    \label{eqn:lipschitz_ball}
    f_R(\distanceCenters) = 4R^2 \arcsin \left(
    \frac{\distanceCenters}{2R} \right) + 
    \distanceCenters \sqrt{4R^2 - \distanceCenters^2} .
  \end{align}
  Now, since
  $ f_R'(\distanceCenters) = 2 \sqrt{4R^2 - \distanceCenters^2}
  \leqslant 4R$ and $f_R(0) = 0$, we deduce that
  $f_R(\distanceCenters) \leqslant 4R \distanceCenters$.
  It follows that $\mathbb{B}_R$ is $\mu$-locally Lipschitz with
  constant $4R$. Since $\advDomain$ is trivially $\mu$-locally
  Lipschitz with constant $0$ and $\visibilityPolygonSymbol$ is
  $\mu$-locally Lipschitz, cf. \Cref{thm:local_lipschitz_mu}, say with
  Lipschitz constant $L$, we deduce from \Cref{lem:intersection} that
  $\advDomain \intersection \mathbb{B}_R \intersection
  \visibilityPolygonSymbol$ is $\mu$-locally Lipschitz with constant
  $L + 4R$.  From here, \Cref{lem:mu_LL} allows us to establish the
  local Lipschitzness of $\visibilityMetric_{\operatorname{range}}$
  with the same constant.
\end{proof}

\begin{remark}[Inflection curves for visibility with limited range]
  \label{rem:inflection_curves_limited_range}
  The points of non-differentiability of the area of the visibility
  metric under limited range comprise parts of inflection segments and
  arcs of circles centered at the anchor with radius~$R$.  This is
  illustrated in \Cref{fig:inflection_curves_limited_range}. In fact,
  a reflex vertex is \emph{invisible} outside an $R$-ball around it.
  Therefore, as a critical point is where the set of anchors changes,
  there are possible critical points on
  $\boundary{\oBall{R}{\reflexVertex}}$.  However, as the reflex vertex
  is not an anchor for observers in $M_1(\reflexVertex)$ and
  $M_3(\reflexVertex)$, the critical points can be characterized as
  follows. For every reflex vertex, the boundary of
  $M_2(\reflexVertex) \intersection \oBall{R}{\reflexVertex}
  \intersection \advDomain$ is a set of critical points in
  $M_2(\reflexVertex)$ because within $\advDomain$, $\reflexVertex$ is
  anchor for the points inside the set, and not an anchor for the
  points outside the set. In an analogous fashion, the boundary of
  $M_4(\reflexVertex) \intersection \oBall{R}{\reflexVertex}
  \intersection \advDomain$ is also a set of critical points.
  \oprocend
\end{remark}

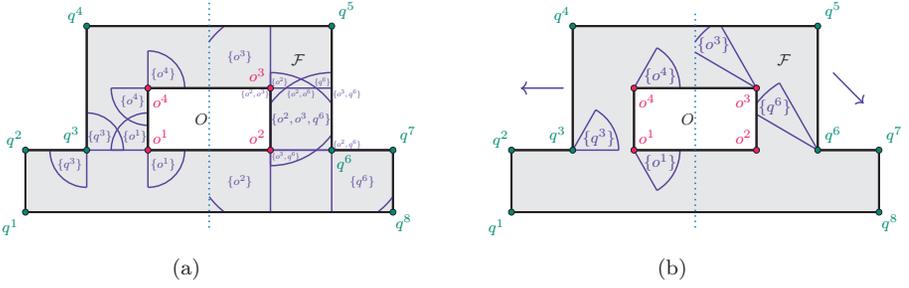
\begin{figure}[htb]
  \centering
  \subfigure[]{
    \label{fig:inflection_curves_limited_range}
    \resizebox{0.47\linewidth}{!}{
      \centering
      \pgfdeclarelayer{background}
      \pgfdeclarelayer{foreground}
      \pgfsetlayers{background,main,foreground}
      \begin{tikzpicture}[scale=0.55]
        \tikzset{
          style 0/.style={inner sep=0pt, minimum size=3pt},
          style A/.style={shape=circle, fill=PineGreen, style 0},
          style B/.style={shape=circle, fill=WildStrawberry, style 0},
          style C/.style={shape=circle, fill=Blue!50!RoyalBlue, style 0},
          font={\fontsize{8pt}{9.6}\selectfont}
        }
        
        \begin{pgfonlayer}{foreground}
          \node[draw, style A] (q1) at (-1, -1) {};
          \node[draw, style A] (q2) at (11, -1) {};
          \node[draw, style A] (q3) at (11,  1) {};
          \node[draw, style A] (q4) at ( 9,  1) {};
          \node[draw, style A] (q5) at ( 9,  5) {};
          \node[draw, style A] (q6) at ( 1,  5) {};
          \node[draw, style A] (q7) at ( 1,  1) {};
          \node[draw, style A] (q8) at (-1,  1) {};

          \node[text=PineGreen] at ([shift={(225:0.7)}]q1) {$q^1$};
          \node[text=PineGreen] at ([shift={(315:0.5)}]q2) {$q^8$};
          \node[text=PineGreen] at ([shift={( 45:0.7)}]q3) {$q^7$};
          \node[text=PineGreen] at ([shift={(-45:0.6)}]q4) {$q^6$};
          \node[text=PineGreen] at ([shift={( 45:0.7)}]q5) {$q^5$};
          \node[text=PineGreen] at ([shift={(135:0.5)}]q6) {$q^4$};
          \node[text=PineGreen] at ([shift={(135:0.7)}]q7) {$q^3$};
          \node[text=PineGreen] at ([shift={(135:0.5)}]q8) {$q^2$};
        \end{pgfonlayer}

        \begin{pgfonlayer}{background}
          \tikzstyle{every path}=[draw, line width=1.0pt, color=black]
          \draw[fill = gray!20] (q1.center) -- (q2.center) -- (q3.center) -- (q4.center) -- (q5.center) -- (q6.center) -- (q7.center) -- (q8.center) -- cycle;
          \node at (7.9, 3.9) {$\freeSpace$};
        \end{pgfonlayer}

        \begin{pgfonlayer}{foreground}
          \node[draw, style B] (o1) at (3, 1) {};
          \node[draw, style B] (o2) at (7, 1) {};
          \node[draw, style B] (o3) at (7, 3) {};
          \node[draw, style B] (o4) at (3, 3) {};

          \node[text=WildStrawberry] at ([shift={( 45:0.6)}]o1) {$o^1$};
          \node[text=WildStrawberry] at ([shift={(135:0.6)}]o2) {$o^2$};
          \node[text=WildStrawberry] at ([shift={(135:0.6)}]o3) {$o^3$};
          \node[text=WildStrawberry] at ([shift={(315:0.6)}]o4) {$o^4$};
        \end{pgfonlayer}

        \begin{pgfonlayer}{main}
          \tikzstyle{every path}=[draw, line width=1.0pt, color=black]
          \draw[fill = white] (o1.center) -- (o2.center) -- (o3.center) -- (o4.center) -- cycle;
          \node at (4.75, 2) {$\obstacle$};
        \end{pgfonlayer}
        
        \begin{pgfonlayer}{background}
          \tikzstyle{every node}=[color=Violet, scale = 0.75]


          \draw[-, Violet, line width=0.75pt] (2.2, 1) arc[start angle=0,
          end angle=90, radius=1.2];
          \draw[-, Violet, line width=0.75pt] (1, 1) -- (2.2, 1);
          \draw[-, Violet, line width=0.75pt] (1, 1) -- (1,
          2.2);
          \node at (1.425, 1.45) {$\{q^3\}$};

          \draw[-, Violet, line width=0.75pt] (-0.2, 1) arc[start angle=-180,
          end angle=-90, radius=1.2];
          \draw[-, Violet, line width=0.75pt] (1, 1) -- (-0.2, 1);
          \draw[-, Violet, line width=0.75pt] (1, 1) -- (1,
          -0.2);
          \node at (0.45, 0.55) {$\{q^3\}$};

          \draw[-, Violet, line width=0.75pt] (4.2, 1) arc[start
          angle=0, end angle = -90, radius=1.2];
          \draw[-, Violet, line width=0.75pt] (3, 1) -- (4.2, 1);
          \draw[-, Violet, line width=0.75pt] (3, 1) -- (3,
          -0.2);
          \node at (3.5, 0.55) {$\{o^1\}$};

          \draw[-, Violet, line width=0.75pt] (1.8, 1) arc[start
          angle=180, end angle=90, radius=1.2];
          \draw[-, Violet, line width=0.75pt] (3, 1) -- (1.8, 1);
          \draw[-, Violet, line width=0.75pt] (3, 1) -- (3,
          2.2);
          \node at (2.575, 1.45) {$\{o^1\}$};

          \draw[-, Violet, line width=0.75pt] (4.2, 3) arc[start
          angle=0, end angle = 90, radius=1.2];
          \draw[-, Violet, line width=0.75pt] (3, 3) -- (4.2, 3);
          \draw[-, Violet, line width=0.75pt] (3, 3) -- (3,
          4.2);
          \node at (3.5, 3.45) {$\{o^4\}$};

          \draw[-, Violet, line width=0.75pt] (1.8, 3) arc[start
          angle=-180, end angle=-90, radius=1.2];
          \draw[-, Violet, line width=0.75pt] (3, 3) -- (1.8, 3);
          \draw[-, Violet, line width=0.75pt] (3, 3) -- (3,
          1.8);
          \node at (2.5, 2.6) {$\{o^4\}$};

        \end{pgfonlayer}

        \begin{pgfonlayer}{main}
          \tikzstyle{every node}=[color=Violet, scale = 0.75]


          \draw[-, Violet, line width=0.75pt] (11, -0.5) arc[start
          angle=-asin(0.6), end angle=-asin(0.8), radius=2.5];
          \draw[-, Violet, line width=0.75pt] (9, 1) -- (9, -1);
          \draw[-, Violet, line width=0.75pt] (9, 1) -- (11, 1);
          \node at (10, 0) {$\{q^6\}$};

          \draw[-, Violet, line width=0.75pt] (7, 2.5) arc[start
          angle=180-asin(0.6), end angle=90, radius=2.5];
          \draw[-, Violet, line width=0.75pt] (9, 1) -- (9, 3.5);
          \draw[-, Violet, line width=0.75pt] (9, 1) -- (7, 1);

          \draw[-, Violet, line width=0.75pt] (7, 3.5) arc[start angle=90,
          end angle=asin(0.6), radius=2.5];
          \draw[-, Violet, line width=0.75pt] (9, 1) -- (7, 1);
          \draw[-, Violet, line width=0.75pt] (7, 1) -- (7, 3.5);

          \draw[-, Violet, line width=0.75pt] (5, -0.5) arc[start
          angle=180+asin(0.6), end angle=180+asin(0.8), radius=2.5];
          \draw[-, Violet, line width=0.75pt] (7, 1) -- (5, 1);
          \draw[-, Violet, line width=0.75pt] (7, 1) -- (7, -1);
          \node at (6, 0) {$\{o^2\}$};

          \draw[-, Violet, line width=0.75pt] (7, 0.5) arc[start angle=-90,
          end angle=-asin(0.6), radius=2.5];
          \draw[-, Violet, line width=0.75pt] (7, 3) -- (9, 3);
          \draw[-, Violet, line width=0.75pt] (7, 3) -- (7, 0.5);

          \draw[-, Violet, line width=0.75pt] (5, 4.5) arc[start angle=180-asin(0.6),
          end angle=180-asin(0.8), radius=2.5];
          \draw[-, Violet, line width=0.75pt] (7, 3) -- (5, 3);
          \draw[-, Violet, line width=0.75pt] (7, 3) -- (7, 5);
          \node at (6, 4) {$\{o^3\}$};
          
          \node at (8, 2) {$\{o^2, o^3, q^6\}$};

          \tikzstyle{every node}=[color=Violet, scale = 0.50]
          \node at (7.5, 0.8) {$\{o^3, q^6\}$};
          \node at (9.5, 1.2) {$\{o^2, q^6\}$};
          \node at (9.5, 2.8) {$\{o^3, q^6\}$};
          \node at (8.7, 3.2) {$\{q^6\}$};
          \node at (7.3, 3.2) {$\{o^2\}$};
          \node at (6.5, 2.8) {$\{o^2, o^3\}$};
          \node at (8, 2.8) {$\{o^2, o^6\}$};

        \end{pgfonlayer}

        \begin{pgfonlayer}{foreground}
          \tikzstyle{every path}=[draw, line width=0.75pt, color=RoyalBlue]
          \draw[dotted] (5, -1.5) -. (5, 6);
        \end{pgfonlayer}

        \begin{pgfonlayer}{main}
          \tikzstyle{every path}=[draw, line width=1.0pt, color=black]
          \draw (q1.center) -- (q2.center) -- (q3.center) -- (q4.center) -- (q5.center) -- (q6.center) -- (q7.center) -- (q8.center) -- cycle;
          \draw (o1.center) -- (o2.center) -- (o3.center) -- (o4.center) -- cycle;
        \end{pgfonlayer}
        
      \end{tikzpicture}
    }%
  }
  \subfigure[]{
    \label{fig:inflection_curves_limited_fov_range}
    \resizebox{0.47\linewidth}{!}{
      \centering
      \pgfdeclarelayer{background}
      \pgfdeclarelayer{foreground}
      \pgfsetlayers{background,main,foreground}
      \begin{tikzpicture}[scale=0.55]
        \tikzset{
          style 0/.style={inner sep=0pt, minimum size=3pt},
          style A/.style={shape=circle, fill=PineGreen, style 0},
          style B/.style={shape=circle, fill=WildStrawberry, style 0},
          style C/.style={shape=circle, fill=Blue!50!RoyalBlue, style 0},
          font={\fontsize{8pt}{9.6}\selectfont}
        }
        
        \begin{pgfonlayer}{foreground}
          \node[draw, style A] (q1) at (-1, -1) {};
          \node[draw, style A] (q2) at (11, -1) {};
          \node[draw, style A] (q3) at (11,  1) {};
          \node[draw, style A] (q4) at ( 9,  1) {};
          \node[draw, style A] (q5) at ( 9,  5) {};
          \node[draw, style A] (q6) at ( 1,  5) {};
          \node[draw, style A] (q7) at ( 1,  1) {};
          \node[draw, style A] (q8) at (-1,  1) {};

          \node[text=PineGreen] at ([shift={(225:0.7)}]q1) {$q^1$};
          \node[text=PineGreen] at ([shift={(315:0.5)}]q2) {$q^8$};
          \node[text=PineGreen] at ([shift={( 45:0.7)}]q3) {$q^7$};
          \node[text=PineGreen] at ([shift={( 45:0.7)}]q4) {$q^6$};
          \node[text=PineGreen] at ([shift={( 45:0.7)}]q5) {$q^5$};
          \node[text=PineGreen] at ([shift={(135:0.5)}]q6) {$q^4$};
          \node[text=PineGreen] at ([shift={(135:0.7)}]q7) {$q^3$};
          \node[text=PineGreen] at ([shift={(135:0.5)}]q8) {$q^2$};
        \end{pgfonlayer}

        \begin{pgfonlayer}{background}
          \tikzstyle{every path}=[draw, line width=1.0pt, color=black]
          \draw[fill = gray!20] (q1.center) -- (q2.center) -- (q3.center) -- (q4.center) -- (q5.center) -- (q6.center) -- (q7.center) -- (q8.center) -- cycle;
          \node at (7.9, 3.9) {$\freeSpace$};
        \end{pgfonlayer}

        \begin{pgfonlayer}{foreground}
          \node[draw, style B] (o1) at (3, 1) {};
          \node[draw, style B] (o2) at (7, 1) {};
          \node[draw, style B] (o3) at (7, 3) {};
          \node[draw, style B] (o4) at (3, 3) {};

          \node[text=WildStrawberry] at ([shift={( 45:0.6)}]o1) {$o^1$};
          \node[text=WildStrawberry] at ([shift={(135:0.6)}]o2) {$o^2$};
          \node[text=WildStrawberry] at ([shift={(225:0.6)}]o3) {$o^3$};
          \node[text=WildStrawberry] at ([shift={(315:0.6)}]o4) {$o^4$};
        \end{pgfonlayer}

        \begin{pgfonlayer}{main}
          \tikzstyle{every path}=[draw, line width=1.0pt, color=black]
          \draw[fill = white] (o1.center) -- (o2.center) -- (o3.center) -- (o4.center) -- cycle;
          \node at (4.75, 2) {$\obstacle$};
        \end{pgfonlayer}
        
        \begin{pgfonlayer}{background}
          \draw[->, Violet, line width=0.75pt] (0.7, 3) -- (-0.7, 3);

          \draw[-, Violet, line width=0.75pt] (2.5, 1) arc[start angle=0,
          end angle=60, radius=1.5];
          \draw[-, Violet, line width=0.75pt] (1, 1) -- ({1+1.5*cos(60)},
          {1+1.5*sin(60)});
          \draw[-, Violet, line width=0.75pt] (1, 1) -- (2.5, 1);
          \node[text=Violet] at (1.85, 1.4) {$\{q^3\}$};

          \draw[-, Violet, line width=0.75pt] (4.5, 1) arc[start
          angle=0, end angle = -60, radius=1.5];
          \draw[-, Violet, line width=0.75pt] (3, 1) -- ({3+1.5*cos(-60)},
          {1+1.5*sin(-60)});
          \draw[-, Violet, line width=0.75pt] (3, 1) -- (4.5, 1);
          \node[text=Violet] at (3.85, 0.6) {$\{o^1\}$};

          \draw[-, Violet, line width=0.75pt] (4.5, 3) arc[start
          angle=0, end angle=60, radius=1.5];
          \draw[-, Violet, line width=0.75pt] (3, 3) -- ({3+1.5*cos(60)},
          {3+1.5*sin(60)});
          \draw[-, Violet, line width=0.75pt] (3, 3) -- (4.5, 3);
          \node[text=Violet] at (3.85, 3.4) {$\{o^4\}$};
          
        \end{pgfonlayer}

        \begin{pgfonlayer}{background}
          \draw[->, Violet, line width=0.75pt] (9.5, 3.5) -- (10.5, 2.5);

          \draw[-, Violet, line width=0.75pt] (7, 2.5) arc[start
          angle=180-asin(0.6), end angle=120, radius=2.5];
          \draw[-, Violet, line width=0.75pt] (9, 1) --
          ({9-2.5*sin(30)}, {1+2.5*cos(30)});
          \draw[-, Violet, line width=0.75pt] (9, 1) -- (7, {1 + 2*tan(30)});
          \node[text=Violet] at (7.6, 2.4) {$\{q^6\}$};


          \draw[-, Violet, line width=0.75pt] (5, 4.5) arc[start
          angle=180-asin(0.6), end angle=180-asin(0.8), radius=2.5];
          \draw[-, Violet, line width=0.75pt] (7, 3) -- ({7 - 2*tan(30)},
          5);
          \draw[-, Violet, line width=0.75pt] (7, 3) -- (5, {3 +
          2*tan(30)});
          \node[text=Violet] at (5.6, 4.4) {$\{o^3\}$};
        \end{pgfonlayer}

        \begin{pgfonlayer}{foreground}
          \tikzstyle{every path}=[draw, line width=0.75pt, color=RoyalBlue]
          \draw[dotted] (5, -1.5) -. (5, 5.5);
        \end{pgfonlayer}

        \begin{pgfonlayer}{main}
          \tikzstyle{every path}=[draw, line width=1.0pt, color=black]
          \draw (q1.center) -- (q2.center) -- (q3.center) -- (q4.center) -- (q5.center) -- (q6.center) -- (q7.center) -- (q8.center) -- cycle;
          \draw (o1.center) -- (o2.center) -- (o3.center) -- (o4.center) -- cycle;
        \end{pgfonlayer}

      \end{tikzpicture}
    }%
  }
  \caption{(a) Inflection curves for visibility for different ranges.
  The range of the observer is $R = 1.2$ to the left of the dotted
  line and $R = 2.5$ to the right.  (b) Inflection curves for
  visibility with different ranges and fields of view.  The observer
  is facing west to the left of the dotted line, with range $R=1.5$
  and field of view angle $\varphi=120^\circ$.  The observer is facing
  south-east to the right of the dotted line, with range $R=2.5$ and
  field of view angle $\varphi=30^\circ$.}
\end{figure}

Next, we study the case of limited range and limited field of view. In
this case, the visibility depends not just on the position of the
observer, but also its heading, which therefore needs to be included
in the state. Let the state $\state$ be the position $\genericPoint
\in \real^2$ augmented with its heading angle $\heading \in \real$,
with positive angles taken counter-clockwise from the reference, set
to the positive horizontal axis. Suppose the field of view of the
observer is determined by $\fov > 0$.  Then, the $R$-visibility cone
\smash{$\fovCone{\radius}{\fov}{\state}$} is defined by $\{
\genericPointQ \in \real^2 \mid |\heading -
\arctanTwo(\genericPointQ_2-\genericPoint_2,
\genericPointQ_1-\genericPoint_1)| \leqslant {\fov}/{2},
\|\genericPointQ - \genericPoint\|_2 \leqslant \radius \}$, that is,
the set of all points in $\oBall{R}{\genericPoint}$ that make an angle
smaller than ${\fov}/{2}$ with the heading $\theta$. We denote by
$\fovConeInfinite{\fov}{\genericPointQ}$ the infinite range visibility
cone.

We first establish the $\mu$-global Lipschitzness of the visibility
polygon with limited field of view.

\begin{prop}[Regularity of visibility polygon with limited field of view]
  \label{prop:smoothness_fov}
  Given a bounded $\freeSpace$, the limited field of view cone
  $\fovConeInfinite{\varphi}{\genericPointR}$ is $\mu$-globally
  Lipschitz over $\freeSpace \times \real$.
\end{prop}
\begin{proof}
  Given points $\genericPointR', \genericPointR'' \in \freeSpace
  \times \real$, choose $\bar{\genericPointR} :=
  (\genericPointR_\genericPoint', \genericPointR_\heading'')$. Then,
  by the triangle inequality
  \begin{align*}
    \|\fovConeInfinite{\varphi}{\genericPointR'} \symDiff
    \fovConeInfinite{\varphi}{\genericPointR''}\|_\mu 
    & \leqslant \|\fovConeInfinite{\varphi}{\genericPointR'} \symDiff
      \fovConeInfinite{\varphi}{\bar{\genericPointR}}\|_\mu 
      + \|\fovConeInfinite{\varphi}{\bar{\genericPointR}} \symDiff
      \fovConeInfinite{\varphi}{\genericPointR''}\|_\mu
    \\
    & \leqslant 2D^2 \|\genericPointR_\heading' -
      \genericPointR_\heading''\| 
      + 2D
      \|\genericPointR_\genericPoint'-\genericPointR_\genericPoint''\|
      \leqslant 2D \sqrt{D^2+1} \,\,
      \|\genericPointR'-\genericPointR''\| ,
  \end{align*}
  where $D$ is the diameter of the free space.
\end{proof}

The regularity of the visibility polygon under both limited range and
limited field of view follows immediately.

\begin{theorem}[Regularity of visibility polygon with
  limited range \& limited field of
  view]
  \label{thm:smoothness_limited_fov_range}
  Let $\advDomain \subset \freeSpace$ be non-empty and let $R > 0$ be
  the sensor range.  Then, for any $0 < \fov \leqslant 2 \pi$, the
  visibility polygon with limited field of view and limited range
  \smash{$\state \mapsto \visibilityPolygon{\genericPoint}
  \intersection \advDomain \intersection
  \fovCone{\radius}{\fov}{\state}$} is $\mu$-locally Lipschitz over
  \smash{$ \left( \freeSpace \setminus \reflexVertices{\freeSpace}
  \right) \times \real$}.  Furthermore, \smash{$\state \mapsto
  \visibilityMetric_{\textrm{FOV}}(\state) :=
  \|\visibilityPolygon{\genericPoint} \intersection \advDomain
  \intersection \fovCone{\radius}{\fov}{\state}\|_\mu$} is locally
  Lipschitz over \smash{$\left( \freeSpace \setminus
  \reflexVertices{\freeSpace} \right) \times \real$}.
\end{theorem}
\begin{proof}
  The first statement follows by invoking
  \Cref{lem:intersection,lem:mu_LL} on $\visibilityPolygonSymbol
  \intersection \mathbb{B}_R \allowbreak \intersection
  \mathfrak{C}^\infty_\fov \intersection \advDomain$, where the first
  three are $\mu$-locally Lipschitz from
  \Cref{thm:local_lipschitz_mu,thm:smoothness_limited_range} and
  \Cref{prop:smoothness_fov} respectively. The second statement then
  follows from \Cref{lem:mu_LL}.
\end{proof}

\begin{remark}[Inflection curves for visibility with limited field of
  view and range]
  \label{rem:inflection_curves_limited_fov_range}
  The critical points of the visibility metric with limited range and
  field of view comprise parts of inflection segments and parts of the
  $R$-visibility cone at an anchor, pointing opposite to the heading
  of the observer.  This is illustrated in
  \Cref{fig:inflection_curves_limited_fov_range}.
  To determine them, we are interested in the points from where an
  anchor lies in the $R$-visibility cone of the observer. From the
  properties of a cone, this is the cone with vertex at the anchor,
  with range $R$, subtending an angle of $\fov$ at the anchor, and
  pointing in the opposite direction of the observer's heading. With a
  reasoning mirroring the one in
  \Cref{rem:inflection_curves_limited_range}, the critical points can
  then be characterized as follows. For every reflex vertex
  $\reflexVertex$, the boundary of $M_i(\reflexVertex) \intersection
  \fovCone{\radius}{\fov}{\reflexVertex; -\heading} \intersection
  \advDomain$ is a set of critical points for $i \in \{2, 4\}$
  because, within $\advDomain$, $\reflexVertex$ is anchor for the
  points inside the set, and not an anchor for the points outside it.
  \oprocend
\end{remark}


\subsection{Finding Best Hiding/Surveillance
  Spots}\label{sec:optimization}%

Here, we address \labelcref{enum:local_extrema}. We first transform
the constrained optimization
problems~\cref{eqn:optimization_min,eqn:optimization_max} into
unconstrained problems and design a normalized generalized
gradient-based algorithm to find local extrema.

We start by defining augmented visibility metrics.  For $\genericPoint
\in \real^2$, let
\begin{align}
  \label{eqn:augmented_visibility_metric}
  \augmentedVisibilityMetric{\myDomain}(\genericPoint)
  & := \visibilityMetric(\proj{\myDomain}{\genericPoint}) + \|\genericPoint -
    \proj{\myDomain}{\genericPoint}\|, 
  \\
  \label{eqn:augmented_visibility_metric_max}
  \augmentedVisibilityMetricMax{\myDomain}(\genericPoint)
  & := \visibilityMetric(\proj{\myDomain}{\genericPoint}) -
    \|\genericPoint - \proj{\myDomain}{\genericPoint}\|, 
\end{align}
where
\smash{$\proj{\myDomain}{\genericPoint} \in \argmin_{\genericPointQ
  \in \myDomain} \|\genericPointQ - \genericPoint\|$}
is a projection of $\genericPoint$ on $\myDomain$.  The local
minimizers of $\augmentedVisibilityMetric{\myDomain}$ (resp.
maximizers of $\augmentedVisibilityMetricMax{\myDomain}$) correspond
to the solutions to the constrained problems
\cref{eqn:optimization_min} (resp. \cref{eqn:optimization_max}).

\begin{lemma}[Correspondence between local extrema of visibility and
  augmented visibility metrics]
  \label{lem:common_local_minima}
  The point $\genericPoint \in \freeSpace$ is a local minimizer (resp.
  maximizer) of $\visibilityMetric$ iff $\genericPoint$ is a local
  minimizer of $\mathcal{V}_{\myDomain}$ (resp. maximizer of
  $\mathcal{W}_{\myDomain}$).
\end{lemma}
\begin{proof}
  We only present the proof for the minimization case---the argument
  for maximization is analogous.  By construction, $\visibilityMetric$
  is the restriction of $\mathcal{V}_{\myDomain}$ to $\myDomain$, so
  the local minimizers coincide on $\interior{\myDomain}$. We next
  show that there are no local minimizers of $\mathcal{V}_{\myDomain}$
  in $\exterior{\myDomain}$. By contradiction, let $\genericPoint^*
  \in \exterior{\myDomain}$ be a local minimizer of
  $\mathcal{V}_{\myDomain}$. Then, for any point $\genericPoint'$ on
  the open segment $(\genericPoint^*,
  \proj{\myDomain}{\genericPoint^*})$, we have
  $\augmentedVisibilityMetric{\myDomain}(\genericPoint') <
  \augmentedVisibilityMetric{\myDomain}(\genericPoint^*)$ as
  $\proj{\myDomain}{\genericPoint'} =
  \proj{\myDomain}{\genericPoint^*}$.  Hence, $\genericPoint^*$ is not
  a local minimizer, leading to a contradiction.%
  
  Finally, consider a point $\genericPoint^*$ on the boundary of
  $\myDomain$. Suppose $\genericPoint^*$ is a local minimum for
  $\visibilityMetric$, then for a sufficiently small $\epsilon >0$,
  $\augmentedVisibilityMetric{\myDomain}(\genericPoint^*) \leqslant
  \augmentedVisibilityMetric{\myDomain}(\genericPoint)$ for all
  $\genericPoint \in \oBall{\epsilon}{\genericPoint^*} \intersection
  \myDomain$, and
  $\augmentedVisibilityMetric{\myDomain}(\genericPoint^*) \leqslant
  \augmentedVisibilityMetric{\myDomain}(\proj{\myDomain}{\genericPoint})
  < \augmentedVisibilityMetric{\myDomain}(\genericPoint)$ for all
  $\genericPoint \in \oBall{\epsilon}{\genericPoint^*} \setminus
  \myDomain$. So, $\genericPoint^*$ is also a local minimum for
  $\augmentedVisibilityMetric{\myDomain}$. On the other hand, if
  $\genericPoint^*$ is a local minimum of
  $\augmentedVisibilityMetric{\myDomain}$, then for a sufficiently
  small $\epsilon > 0$, we have $\visibilityMetric(\genericPoint^*)
  \leqslant \visibilityMetric(\genericPoint)$ for all $\genericPoint
  \in \oBall{\epsilon}{\genericPoint^*} \intersection \myDomain$ as
  $\visibilityMetric$ and $\augmentedVisibilityMetric{\myDomain}$
  agree on the domain, and $\genericPoint^*$ is a local minimum for
  $\visibilityMetric$ too.
\end{proof}

To minimize~$\mathcal{V}_{\myDomain}$, we strive for an algorithm with
the following properties:
\begin{enumerate}[label={(P\arabic*)}]
  \item applicable to non-convex non-smooth
    objectives; \label{enum:algo_type}
  \item allowing for non-monotonic decrease in the
    objective; \label{enum:algo_non_mono}
  \item almost surely converging to local minimizers, while avoiding
    saddle points and local maximizers. \label{enum:algo_avoid+conv}
\end{enumerate}

Requirement~\labelcref{enum:algo_type} is justified by the fact that
$\mathcal{V}_{\myDomain}$ is non-smooth, non-convex. Without
requirement~\labelcref{enum:algo_non_mono}, the monotonic decrease of
the objective would prevent us from avoiding saddle points in general,
cf. requirement \labelcref{enum:algo_avoid+conv}.

Our proposed solution is the \hyperlink{alg:Norcent}{Norcent
(\textbf{Nor}malized Des\textbf{cent}) Algorithm}. The pseudocode is
presented in the accompanying algorithmic environment\footnote{In line
5, $e_1, e_2$ are the standard orthonormal basis for $\real^2$. Note
that $\nabla_s \augmentedVisibilityMetric{\myDomain}(\genericPoint)
\in \partial \augmentedVisibilityMetric{\myDomain}(\genericPoint)$
because $\mathcal{V}_{\myDomain}'(\genericPoint;v) \in \partial
\augmentedVisibilityMetric{\myDomain}(\genericPoint)$ exists for all
$v \in \real^2 \setminus \{0\}$ from
\Cref{cor:directional_derivative_component_area-metric,cor:directional_derivative_inflection_area-metric}.
This makes sure that irrespective of the choice of coordinates, the
numerical gradient always lies in the generalized gradient of the
augmented visibility metric.}.  The strategy leverages normalized
gradient
algorithms~\cite{LB:72-kibernetika,JC:06-auto,VM-AG-VN:87,NS-KK-AR:85}
and hence satisfies~\labelcref{enum:algo_type}.  Moreover, the
\NorcentTO~displays randomization at stationary points, leading to
non-monotonic decrease in the objective, in line with
\labelcref{enum:algo_non_mono}.
Non-monotonicity also plays a role in the possible increase in the
objective when crossing non-smooth points, which we have shown to be
constrained to lines, and hence have zero measure.
The rest of this section is devoted to show that \NorcentTO{}
satisfies \labelcref{enum:algo_avoid+conv}.

\begin{algorithm}[!t]
  \caption*{\hypertarget{alg:Norcent}{\textbf{Norcent Algorithm}}}
  \label{alg:Norcent}
  \begin{algorithmic}[1]
    \Require $x_0 \in \myDomain + \oBall{\stepSize_0}{0},
    \left\{\stepSize_k\right\}_{k=0}^\infty, 
    \left\{\stepSizeB_k\right\}_{k=0}^\infty,
    \left\{\delta^{\mathrm{th}}_k\right\}_{k=0}^\infty,
    \delta_{\mathrm{tol}}> 0$
    \Ensure $\sum_{k=0}^{\infty} \stepSize_k = \infty$, 
    $\sum_{k=0}^{\infty} \stepSize_k^2 < \infty$, 
    $\stepSize_k \downarrow 0$,
    $\sum_{k=0}^{\infty} \stepSizeB_k^2 = \infty$, 
    $\stepSizeB_k \downarrow 0$,
    $\delta^{\mathrm{th}}_k \downarrow 0$,
    $\delta^{\mathrm{th}}_0 \ll 1$
    \State $x \gets x_0$
    \State $k \gets 0$
    \While{True}
      \State $x_{\mathrm{old}} \gets x$
      \State $\nu \gets \nabla_s
      \augmentedVisibilityMetric{\myDomain}(\genericPoint) = \left[
  \rightDirectionalDerivative{\augmentedVisibilityMetric{\myDomain}}(\genericPoint;
  e_1),
  \rightDirectionalDerivative{\augmentedVisibilityMetric{\myDomain}}(\genericPoint;
  e_2) \right]$
      \label{line:gen_grad}
      \If{$\|\nu\| \leqslant \delta^{\mathrm{th}}_k$ \& $\genericPoint
      \in \myDomain$}
        \State choose $\nu$ uniformly randomly in
        $\boundary{\oBall{1}{0}}$
        \State $x \gets x - \stepSizeB_k \hat{\nu}$
      \Else
        \State $x \gets x - \stepSize_k \hat{\nu}$
      \EndIf
      \If{$\|\augmentedVisibilityMetric{\myDomain}(x_\mathrm{old}) - \augmentedVisibilityMetric{\myDomain}
      (x)\| < \delta_{\mathrm{tol}}$}
        \State break
      \EndIf  
      \State $k \gets k + 1$
    \EndWhile
  \end{algorithmic}
\end{algorithm}

\begin{remark}\longthmtitle{Comparison with the literature}
  Most algorithms that employ normalized gradient descent converge to
  Clarke stationary points
  \cite{JB-AL-MO:05-siamjo,JC:06-auto,KK:07-siamjo,VM-AG-VN:87}.  The
  variant in \cite{RM-BS-SK:19-tac,BS-RM-HP-SK:22-jmlr} converges to
  local minimizers, but requires a strong assumption that the function
  is thrice differentiable at saddle points, which is violated in our
  case for saddle points lying on inflection segments or the boundary
  of the free space. \cite{LB:72-kibernetika,NS-KK-AR:85} show
  convergence to isolated local minimizers, but in our case, the
  function $\augmentedVisibilityMetric{\myDomain}$ can have a
  continuum of local minimizers; for instance, all points in a convex
  polygonal environment have a constant value of the metric, and are
  trivially local minimizers. Therefore, these algorithms cannot be
  directly applied to our problem, justifying the design of the
  \Norcent. \oprocend
\end{remark}

\begin{remark}[Stepsizes]
  \label{rem:step_size}
  In addition to the standard assumption of the stepsize sequence
  $\{\stepSize_k\}$, being square summable but not summable, we
  require that they are monotonically decreasing to $0$. Most step
  size sequences used in optimization, such as $\stepSize_k = c_2 /
  \left(1 + c_3 k^{c_1} \right)$ or $\stepSize_k = c_2 / k^{c_1}$,
  with $c_1 \in (0, 1]$, satisfy these assumptions.  The use of the
  stepsize sequence $\{ \stepSizeB_k \}$ that is not square summable
  and yet monotonically decreasing to zero is less standard. We use it
  to guarantee escape from undesired stationary points
  \oprocend
\end{remark}

\begin{remark}[Non-uniqueness of projection]
  \label{rem:non_uniqueness}
  We do not require $\myDomain$ to be convex. The function
  $\augmentedVisibilityMetric{\myDomain}$ is continuous outside
  $\myDomain$, but may be non-smooth due to the non-uniqueness of the
  projection on to non-convex sets. However, the set of points where
  this happens has zero measure. These points are escaped with
  probability $1$ (cf. Case~4 in proof of \Cref{thm:non_stopping}).
  \oprocend
\end{remark}

To show the convergence of \Norcent, we first show that the distance
to $\myDomain$ decreases at every iteration of the algorithm.

\begin{lemma}[Distance to \texorpdfstring{$\myDomain$}{D\textsubscript{1}}]
  \label{lem:distance_decrease}
  For $\genericPoint_k \in \exterior{\myDomain}$, let
  $\genericPoint_{k+1}$ be the next iterate of the \Norcent.  Then,
  $\|\genericPoint_{k+1}-\proj{\myDomain}{\genericPoint_{k+1}}\| <
  \|\genericPoint_{k}-\proj{\myDomain}{\genericPoint_{k}}\|$.
\end{lemma}
\begin{proof}
  Let the set of all unique projections of $\genericPoint_k$ on
  $\myDomain$ by given by
  \smash{$\{\pi_{\myDomain}^{(i)}(\genericPoint_k)\}_{i=1}^{n_{\pi}}$}.
  Then, \smash{$\|\genericPoint_k -
  \pi_{\myDomain}^{(i)}(\genericPoint_k)\|$} have equal values for all
  $i \in [n_{\pi}]$ and
  \begin{align*}
    \partial \augmentedVisibilityMetric{\myDomain} (\genericPoint_k) =
    \convexHull{\{ 
    \normalization(\genericPoint_k -
    \pi_{\myDomain}^{(i)}(\genericPoint_k)) \}_{i=1}^{n_{\pi}}},
  \end{align*} 
  by choosing $E_f = E$ as the set of points where the distance to
  $\myDomain$ fails to be differentiable, and the gradients in the
  neighborhood are due to \cite[Corollary 3.4.5]{PC-CS:04}.  This
  means that, for every $\hat{\nu} \in \partial
  \augmentedVisibilityMetric{\myDomain}(\genericPoint_k)$, there
  exists $\lambda \in \realpositive^{n_{\pi}}$ with $\lambda^\top
  \mathbbm{1} = 1$ such that
  \smash{$\hat{\nu} = \sum_{i=1}^{n_{\pi}} \lambda_i \normalization(
  \genericPoint_k - \pi_{\myDomain}^{(i)}(\genericPoint_k))$}.
  It follows that
  \begin{align*}
    \|\genericPoint_{k+1}-\proj{\myDomain}{\genericPoint_{k+1}}\| 
    & =  \|
      \genericPoint_{k} - \stepSize_k \hat{\nu} -
      \proj{\myDomain}{\genericPoint_{k+1}} \|
    &&
       \text{(\hyperlink{alg:Norcent}{Norcent} Line 12)}
    \\
    & \leqslant
      \| \genericPoint_{k} - \stepSize_k \hat{\nu} -
      \sum_{i=1}^{n_{\pi}} \lambda_i \pi_{\myDomain}^{(i)}
      (\genericPoint_k) \|
    &&
       \text{(definition of $\pi_{\myDomain}$)}
    \\
    & \leqslant \sum_{i=1}^{n_{\pi}} 
      \lambda_i \| \genericPoint_{k} 
      \! - \! 
      \pi_{\myDomain}^{(i)}(\genericPoint_k) 
      \! - \!
      \stepSize_k \normalization (\genericPoint_{k} 
      \! - \!
      \pi_{\myDomain}^{(i)}(\genericPoint_k)) \|
    &&
       \text{($\triangle$ inequality)}
    \\
    & = - \stepSize_k + \sum_{i=1}^{n_{\pi}}
      \lambda_i
      \| \genericPoint_{k} -
      \pi_{\myDomain}^{(i)}(\genericPoint_k) \| && \text{(parallel vectors)}
    \\
    & < \| \genericPoint_{k} - \proj{\myDomain}{\genericPoint_k}
      \|, && (\stepSize_k > 0)
  \end{align*}
  for any projection $\proj{\myDomain}{\genericPoint_k}$, which proves
  the result.
\end{proof}

We use this result to show that the set $\myDomain +
\oBall{\min\{\stepSize_0,\stepSizeB_0\}}{0}$ is forward invariant
under the \NorcentTO.

\begin{proposition}[Invariance under \NorcentTO]
  \label{prop:invariant_set}
  The set $\tilde{D}_1 := \myDomain +
  \oBall{\min\{\stepSize_0,\stepSizeB_0\}}{0}$ is forward invariant
  under \NorcentTO: if \smash{$\genericPoint_0 \in \tilde{D}_1$}, then
  \smash{$\genericPoint_k \in \tilde{D}_1$} for all $k \in
  \integernonnegative$.
\end{proposition}
\begin{proof}
  Given $\genericPoint_k \in \tilde{D}_1$, we consider two cases: (1)
  if $\genericPoint_k \in \myDomain$, by the update rule in Lines 8
  and 10 of \Norcent, we have \smash{$\genericPoint_{k+1} \in
  \myDomain + \oBall{\min\{\stepSize_k, \stepSizeB_k\}}{0} \subseteq$}
  \smash{$\myDomain +\oBall{\min\{\stepSize_0, \stepSizeB_0\}}{0}
  \subset \tilde{D}_1 $} due to monotonicity of the two stepsize
  sequences; (2) if $\genericPoint_k \in \tilde{D}_1 \setminus
  \myDomain \subset \exterior{\myDomain}$, by
  \Cref{lem:distance_decrease}, $\genericPoint_{k+1} \in \tilde{D}_1$.
  Thus, by induction, the claim holds for all $k \in
  \mathbb{Z}_{\geqslant 0}$.
\end{proof}

Next, we show how the randomization in Line 7 of \Norcent{} allows us
to establish the non-stopping property
\labelcref{enum:algo_avoid+conv} for the iterates: that is, the
algorithm does not get stuck at local maximizers or saddle points, but
almost surely escapes them.

\begin{theorem}[Non-stopping property]
  \label{thm:non_stopping}
  Consider the sequence of iterates $\{\genericPoint_k\}$ generated
  by~\NorcentTO~initialized at $\genericPoint_0 \in
  \tilde{D}_1$. Then, for any convergent subsequence
  $\{\genericPoint_{k_s}\}$ whose limit point $\tilde{\genericPoint}$ is not in
  $\solutions$, there exists
  $\bar{\delta}$ such that, for all
  sufficiently large $s$ and all
  $\delta \in (0, \bar{\delta}\verythinspace]$,
  \begin{align}
    \label{eqn:non_stopping_1}
    \kappa_\delta(s) = \min\{r \mid \|x_r - x_{k_s}\| > \delta, r > k_s\} < \infty,
  \end{align}
  and 
  \begin{align}
    \label{eqn:non_stopping_2}
    \limsup_{s \to \infty}
    \augmentedVisibilityMetric{\myDomain}(\genericPoint_{\kappa_\delta(s)})  
    < \lim_{s \to \infty}
    \augmentedVisibilityMetric{\myDomain}(\genericPoint_{k_s}) =
    \augmentedVisibilityMetric{\myDomain}(\tilde{\genericPoint}),
  \end{align}
  hold with probability 1. 
\end{theorem}
\begin{proof}
  Let $\setStationaryPoints := \{ \genericPoint \in \myDomain \mid 0
  \in \partial \augmentedVisibilityMetric{\myDomain}(\genericPoint)
  \}$ be the set of all Clarke stationary points of
  $\augmentedVisibilityMetric{\myDomain}$ over the ego domain
  $\myDomain$ and $\setNonsmoothPoints:= \{ \genericPoint \in \real^2
  \setminus \myDomain \mid |\partial
  \augmentedVisibilityMetric{\myDomain}(\genericPoint)| > 1\}$
  be the set of points outside the ego domain where the distance
  function to $\myDomain$ is non-smooth, identified by the cardinality
  of the generalized gradient.
  We consider the following cases for the limit point
  $\tilde{\genericPoint}$. For the reader's convenience, we provide
  summary statements of the steps of the proof.

  \textbf{Case 1: $\tilde{\genericPoint} \in \myDomain
    \setminus \setStationaryPoints$}. 

  \emph{(Bounding gradient vectors at $\tilde{\genericPoint}$):}
  Suppose $\tilde{\genericPoint}$ lies on $n_\inflectionSegment
  \geqslant 0$ inflection segments and has $n_\componentTerrain$
  neighboring components
  $\{\componentTerrain_j\}_{j=1}^{n_\componentTerrain} \subset
  \partition$. 
  Note that the $n_\inflectionSegment$ is bounded by
  \Cref{lem:bound_inflection_segments} and $n_\inflectionSegment = 0$
  corresponds to the case where $\tilde{\genericPoint}$ lies in the
  interior of a connected component $\componentTerrain_1$ of the
  partition $\partition$. 
  Then, either $n_\componentTerrain = 1$ when $n_\inflectionSegment =
  0$, or $n_\componentTerrain \leqslant 2 n_\inflectionSegment$, as
  each inflection segment is a line segment from
  \Cref{def:inflection_segments}, and a line segment divides the
  entire $2$-D plane into at most 2 regions. Thus, the generator
  $\gen{} \partial \augmentedVisibilityMetric{\myDomain}
  (\tilde{\genericPoint})$ of the generalized gradient is finite (cf.
  \Cref{lem:cone_lemma}), given by $\gen{} \partial
  \augmentedVisibilityMetric{\myDomain} (\tilde{\genericPoint}) =
  \left\{ \nabla^{(j)}
  \augmentedVisibilityMetric{\myDomain}(\tilde{\genericPoint})
  \right\}_{j=1}^{n_\componentTerrain}$, where each
  \begin{align*}
    \nabla^{(j)}
    \augmentedVisibilityMetric{\myDomain}(\tilde{\genericPoint}) :=
    \lim_{{\genericPoint_i \verythinspace \to \verythinspace
    \tilde{\genericPoint}}, \, {\genericPoint_i \verythinspace \in
    \verythinspace \componentTerrain_j}} \nabla
    \augmentedVisibilityMetric{\myDomain} (\genericPoint_i),
  \end{align*}
  exists and is unique due to
  \Cref{cor:directional_derivative_component_area-metric}.
  Furthermore, by choosing $E_f = \setCriticalPoints$ in the
  definition of the generalized gradient (cf. \Cref{sub:non_smooth}),
  it is clear from
  \Cref{cor:directional_derivative_component_area-metric} that
  $\partial \augmentedVisibilityMetric{\myDomain}
  (\tilde{\genericPoint}) = \convexHull{\gen{} \partial
  \augmentedVisibilityMetric{\myDomain} (\tilde{\genericPoint})}$.
  By \Cref{lem:cone_lemma}, there exists a unit vector $\hat{u} \in
  \real^2$ and an angle $\alpha \in (0, \sfrac{\pi}{2})$ such that
  $\partial \augmentedVisibilityMetric{\myDomain}
  (\tilde{\genericPoint})$ is completely contained in
  $\fovConeInfinite{2\alpha}{\hat{u}}$.

  \sloppy
  \emph{(Defining neighborhoods and bounding gradient vectors):}
  Let
  \smash{$\nabla_{\text{min}}\augmentedVisibilityMetric{\myDomain}(\tilde{\genericPoint})
  = \allowbreak \argmin_{g \in \partial
  \augmentedVisibilityMetric{\myDomain}(\tilde{\genericPoint})} \|g\|
  $}
  and choose $\tilde{\epsilon} = \|
  \nabla_{\text{min}}\augmentedVisibilityMetric{\myDomain}(\tilde{\genericPoint})\|
  /2$. Note that $\tilde{\epsilon} > 0$ as $0 \not \in \partial
  \augmentedVisibilityMetric{\myDomain}(\tilde{\genericPoint})$, which
  in turn is due to $\tilde{\genericPoint} \not \in \mathcal{S}$. For
  each $\componentTerrain_j$, there exists $\tilde{\delta}^{j} > 0$
  such that $\nabla
  \augmentedVisibilityMetric{\myDomain}(\genericPointQ) \in
  \oBall{\tilde{\epsilon}}{\nabla^{(j)}\augmentedVisibilityMetric{\myDomain}(\tilde{\genericPoint})}$
  for all $\genericPointQ \in
  \oBall{\tilde{\delta}^{j}}{\tilde{\genericPoint}} \intersection
  \componentTerrain_j$. Set $\tilde{\delta} =
  \min_{j=1}^{n_\componentTerrain} \delta^{j}$, and we get $\partial
  \augmentedVisibilityMetric{\myDomain} (\genericPointQ) \subset
  \fovConeInfinite{\sfrac{\pi}{2} + \alpha}{\hat{u}}$ for all
  $\genericPointQ \in \oBall{\tilde{\delta}}{\tilde{\genericPoint}}$.
  So, $\hat{\nu}$ in Line 10 of \NorcentTO{} lies in a cone of
  aperture $\left(\sfrac{\pi}{2} + \alpha\right) \in (\sfrac{\pi}{2},
  \pi)$, centered at $\hat{u}$.

  \emph{(Lower bound on progress made by iterates at each step):}
  Note that there exists $\tilde{\delta}_2 > 0$ such that there are no
  other inflection segments in
  \smash{$\oBall{\tilde{\delta}_2}{\tilde{\genericPoint}}$}.
  Let $\tilde{\delta} = \min\{ \tilde{\delta}_1, \tilde{\delta}_2\}$.
  We choose $\bar{\delta} = \tilde{\delta}/3$ and $\bar{s}$
  sufficiently large such that (i) $\stepSize_{k_{\bar{s}}} <
  \bar{\delta}$ and (ii) ${\genericPoint_{k_{\bar{s}}}} \in {
  \oBall{\bar{\delta}}{\tilde{\genericPoint}}}$.
  Then, given an index $s \geqslant \bar{s}$, any $\delta \in (0,
  \bar{\delta}]$, for every $k \geqslant k_s$, we have
  \begin{align}
    \label{eqn:iterate_progress}
    \left( \genericPoint_{k+1} - \genericPoint_{k_s} \right) \cdot
    (-\hat{u}) \geqslant \left( \genericPoint_{k} -
    \genericPoint_{k_s} \right) \cdot (-\hat{u}) + \stepSize_k
    \cos\big(\sfrac{\pi}{4} + \sfrac{\alpha}{2} \big), 
  \end{align}
  as long as $\genericPoint_k, \genericPoint_{k+1} \in
  \oBall{\tilde{\delta}}{\tilde{\genericPoint}}$. In other words, we
  make at least $\stepSize_k \cos(\sfrac{\pi}{4} + \sfrac{\alpha}{2})
  > 0$ progress at each $k$\textsuperscript{th} step in the
  $(-\hat{u})$ direction before exiting the $\delta$-ball. By summing
  \cref{eqn:iterate_progress} from $k_s$ to $k$, we get
  \begin{align}
    \label{eqn:iterate_cone_1}
    \left( \genericPoint_k - \genericPoint_{k_s} \right) \cdot
    (-\hat{u}) \geqslant \cos\left(\sfrac{\pi}{4} +
    \sfrac{\alpha}{2}\right) \Big(\sum_{l=k_s}^k \stepSize_l \Big).
  \end{align}
  Now, as $\cos\left(\sfrac{\pi}{4} + \sfrac{\alpha}{2}\right)$ is a
  positive constant and the stepsize sequence $\{\stepSize_l\}$ is not
  summable, there must exist an index $\bar{r} < \infty$ such that
  $\sum_{l=k_s}^{\bar{r}} \stepSize_l > \delta /
  \cos\left(\sfrac{\pi}{4} + \sfrac{\alpha}{2}\right)$. Note that the
  iterates for all $l$ from $k_s$ to $\bar{r}$ remain in the
  $\tilde{\delta}$-ball of $\tilde{\genericPoint}$, as for all such
  points, $\genericPoint_l \in \oBall{\delta +
  \stepSize_{\bar{r}}}{\genericPoint_{k_s}}$, $\genericPoint_{k_s} \in
  \oBall{\bar{\delta}}{\tilde{\genericPoint}}$, and $\delta +
  \stepSize_{\bar{r}} + \bar{\delta} < \tilde{\delta}$. By invoking
  the Cauchy-Schwarz inequality, $\|\genericPoint_{\bar{r}} -
  \genericPoint_{k_s}\| \geqslant \left( \genericPoint_{\bar{r}} -
  \genericPoint_{k_s} \right) \cdot (-\hat{u}) > \delta$. This implies
  $\genericPoint_{\bar{r}} \not \in
  \oBall{\delta}{\genericPoint_{k_s}}$. Thus, $\kappa_\delta(s)
  \leqslant \bar{r} < \infty$, and \cref{eqn:non_stopping_1} follows. 
  
  \emph{(Monotonic decrease in value of augmented visibility metric at
  each iteration---construction of sets for a splitting favorable to
  analysis):} Next, we show \cref{eqn:non_stopping_2}. The key insight
  here is that in the non-smooth case, given that the stepsize is
  sufficiently small, if we use a normalized gradient based update,
  the function value increases only when the iterates move across an
  inflection segment. To formalize this, we need to define certain
  sets within $\freeSpace$. First, note that
  $\bar{\setCriticalPoints}$ can be expressed as a union of inflection
  segments combined with the set of line segments demarcating the
  boundary of the free space, which we call \emph{boundary segments}.
  We denote by $\bar{\mathcal{I}}$ these set of line segments, with
  elements $\bar{I}_j$'s, each of which have at most an endpoint in
  common.

  We define the (open) expansion of segment $\bar{I}_j$ by a stepsize
  $\stepSize_k$ as \smash{$\bar{I}_j^{\stepSize_k} = \bar{I}_j +
  \oBall{\stepSize_k}{0}$}. These are collected into
  \smash{$\bar{\mathcal{I}}^{\stepSize_k} = \cup_j \,
  \bar{I}_j^{\stepSize_k}$}.
  Second, we collect the closed connected components of
  $\oBall{\tilde{\delta}}{\tilde{\genericPoint}} \setminus
  \bar{\mathcal{I}}^{\stepSize_k}$, which we represent by
  $\component_j$'s.  These are shrunken versions of the connected
  components $\componentTerrain_j$'s in the partition $\partition$ as
  parts of the set are taken away by the expanded inflection segments.
  We use the sets defined above to split the difference
  \smash{$\augmentedVisibilityMetric{\myDomain}(\genericPoint_{k_s})
  -\augmentedVisibilityMetric{\myDomain}(\genericPoint_{\kappa_\delta(s)})$}
  into multiple parts -- one each for the time the iterates stay in an
  expanded segment $\bar{I}_j^{\stepSize_k}$ or a shrunken connected
  component \smash{$\component_j$}, and another for each transition
  from an expanded segment or a shrunken component to another.  Let
  the transitions from one set to another happen at iterations $l_1,
  l_2, \cdots, l_T$, that is, all iterates from $k_s$ to $l_1$ lie in
  one of the sets defined above, iterates $(l_1 + 1)$ through $l_2$
  lie in another set, and so on.
  This gives us the following splitting for the change in the metric
  from $k_s$ to $\kappa_\delta(s)$,
  \begin{equation}
    \label{eqn:diff_split}
    \begin{aligned}
      \augmentedVisibilityMetric{\myDomain}(\genericPoint_{k_s}) -
      \augmentedVisibilityMetric{\myDomain}(\genericPoint_{\kappa_\delta(s)})
      = \sum_{i=0}^{T} &
      (\augmentedVisibilityMetric{\myDomain}(\genericPoint_{l_i+1})
      -
      \augmentedVisibilityMetric{\myDomain}(\genericPoint_{l_{i+1}}))
      \\
      & + \sum_{i=1}^{T-1}
      (\augmentedVisibilityMetric{\myDomain}(\genericPoint_{l_i}) -
      \augmentedVisibilityMetric{\myDomain}(\genericPoint_{l_i+1})),
    \end{aligned}
  \end{equation}
  where we set $l_0 = k_s-1$ and $l_{T+1} = \kappa_\delta(s)$. Here,
  the first term collects the changes in the metric while moving
  within one of our sets, and the second summation does the same for
  the changes when moving from one of our sets to another.

  \emph{(Note on the non-uniqueness of the splitting):} As we chose
  inflection segments which may have intersections, when we expand
  them into $\bar{I}_i^{\stepSize_k}$, they no longer constitute a
  partition, as the expanded segments $\bar{I}_i^{\stepSize_k}$'s may
  have non-empty intersections.
  We therefore choose a splitting in \cref{eqn:diff_split} that leads
  to the least number of transition $T$. This may still not be unique,
  but, as we see later, the rest of the argument works irrespective of
  the choice of splitting.

  \emph{(Lower bound for the first term of \cref{eqn:diff_split}):} We
  first bound the first term.
  For each transition index $i$, the points
  $\genericPoint_{l_i+1}, \dots, \genericPoint_{l_{i+1}}$ lie in
  either a shrunken component $\component_j$, or an expansion of a
  segment \smash{$\bar{I}^{\stepSize_k}_j$}. Either way, we index the sets
  with the transition index~$i$ for convenience, as $\component_i$ or
  $\bar{I}_i^{\stepSize_k}$.  In both cases, we show that
  $\augmentedVisibilityMetric{\myDomain}(\genericPoint_{l_i+1})-
  \augmentedVisibilityMetric{\myDomain}(\genericPoint_{l_{i+1}}) > 0$.
  For the case $\component_i$, this follows from the Descent Lemma
  \cite[Proposition A.24]{DPB:99}, which can be applied due to
  \Cref{cor:directional_derivative_component_area-metric}. We consider
  the case $\bar{I}_i^{\stepSize_k}$ next.
  
  \emph{(Establishing decrease in metric at points along an inflection
  or a boundary segment):}
  The gradients on either side of $\bar{I}_i$ are well defined due to
  \Cref{cor:directional_derivative_component_area-metric}, and their
  components perpendicular to $\bar{I}_i$ must point towards
  $\bar{I}_i$, as otherwise, the iterates would just escape
  $\bar{I}_i^{\stepSize_k}$. Moreover, the parallel components must
  result in a net direction of iterates along $\bar{I}_i$, otherwise
  there would be a critical point in
  $\oBall{\tilde{\delta}}{\tilde{\genericPoint}}$. Let this direction
  be~$\hat{e}$. Now, the gradient on at least one side must have a
  parallel component along~$-\hat{e}$, and for any transition from a
  point on this side $\genericPoint_k$ to a point on the other side
  $\genericPoint_{k+1}$, we define the point $\genericPointQ_k$ on
  $\bar{I}_i$ as the point where it intersects
  $\lineSegment{\genericPoint_k}{\genericPoint_{k+1}}$. Let the
  sequence of indices where such transitions happen be $l_{i_j}$ with
  cardinality $J$.  We break down the change in metric when moving
  within $\bar{I}_i^{\stepSize_k}$ as follows
  \begin{equation}
    \label{eqn:diff_split_2}
    \begin{aligned}
      \augmentedVisibilityMetric{\myDomain}(\genericPoint_{l_i+1})
      & ~
      - \augmentedVisibilityMetric{\myDomain}(\genericPoint_{l_{i+1}}) =
      (\augmentedVisibilityMetric{\myDomain}(\genericPoint_{l_i+1})
      -
      \augmentedVisibilityMetric{\myDomain}(\genericPointQ_{l_i+1}))
      \\
      & \quad + \sum_{j=1}^{J}
      (\augmentedVisibilityMetric{\myDomain}(\genericPointQ_{i_j}) -
      \augmentedVisibilityMetric{\myDomain}(\genericPointQ_{i_{j+1}}))
      +
      (\augmentedVisibilityMetric{\myDomain}(\genericPointQ_{l_{i_J}})
      -
      \augmentedVisibilityMetric{\myDomain}(\genericPoint_{l_{i+1}})). 
    \end{aligned}
  \end{equation}
  The change in the first and the last step is lower bounded by
  $-\stepSize_{l_i} L$, where $L$ is the Lipschitz constant over the
  $\tilde{\epsilon}$ ball around $\tilde{\genericPoint}$, which must
  be finite from \Cref{cor:local_lipschitz} because if
  $\tilde{\genericPoint} \in \reflexVertices{\freeSpace}$, then $0 \in
  \partial\augmentedVisibilityMetric{\myDomain}(\tilde{\genericPoint})$.
  The middle term needs to be analyzed carefully as we can not apply
  the mean value theorem \cite[Theorem 2.3.7]{FHC:90} and
  \Cref{cor:local_lipschitz} directly to
  $(\augmentedVisibilityMetric{\myDomain}(\genericPointQ_{i_j}) -
  \augmentedVisibilityMetric{\myDomain}(\genericPointQ_{i_{j+1}}))$.
  Instead, we define $\genericPointR_k = \alpha \genericPoint_k +
  (1-\alpha) \genericPointQ_k$, for some $\alpha \in (0,1)$, and apply
  the mean value theorem to
  $(\augmentedVisibilityMetric{\myDomain}(\genericPointR_{i_j}) -
  \augmentedVisibilityMetric{\myDomain}(\genericPointR_{i_{j+1}}))$.
  Therefore, there exists $\tilde{\genericPointR}_{i_j} \in
  \lineSegment{\genericPointR_{i_j}}{\genericPointR_{i_{j+1}}}$ such
  that
  \begin{align}
    \augmentedVisibilityMetric{\myDomain}(\genericPointR_{i_j}) -
    \augmentedVisibilityMetric{\myDomain}(\genericPointR_{i_{j+1}}) \in \langle \partial \augmentedVisibilityMetric{\myDomain} (\tilde{\genericPointR}_{i_j}), \genericPointR_{i_j} - \genericPointR_{i_{j+1}} \rangle. 
  \end{align}
  By noting that for any index $k$, $\lim_{\alpha \to 0^+}
  \genericPointR_k = \genericPointQ_k$ and that, by definition,
  $\genericPointR_k$ lies on the same side of $\bar{I}_i$ as
  $\genericPoint_k$, we deduce that $\partial
  \augmentedVisibilityMetric{\myDomain}
  (\tilde{\genericPointR}_{i_j})$ is a singleton, comprising the
  gradient due to
  \Cref{cor:directional_derivative_component_area-metric}. Moreover,
  the component of the gradient along $(\genericPointR_{i_j} -
  \genericPointR_{i_{j+1}})$ is positive as $\genericPointR_{i_j} -
  \genericPointR_{i_{j+1}} = -\hat{e} \|\genericPointR_{i_j} -
  \genericPointR_{i_{j+1}}\|$ and we chose the side where the gradient
  had a component along $-\hat{e}$ earlier. It follows that the inner
  product is strictly positive. Taking limits $\alpha \to 0^+$, and by
  continuity of $\augmentedVisibilityMetric{\myDomain}$ from
  \Cref{cor:local_lipschitz}, it follows that
  $\augmentedVisibilityMetric{\myDomain}(\genericPointQ_{l_i+1})-
  \augmentedVisibilityMetric{\myDomain}(\genericPointQ_{l_{i+1}}) >
  0$.

  \emph{(Lower bound for the second term of \cref{eqn:diff_split} and
    putting everything together):}
  It follows that the first term of \cref{eqn:diff_split} is strictly
  lower bounded by $0$, and taking $\limsup$, the term can not
  decrease to $0$ as that would require that the gradient at each
  $\tilde{\genericPointR}_{i_j}$'s and the corresponding gradients in
  the Descent Lemma be $0$, which is excluded by the assumption that
  there are no critical point in the ball. We next focus on the second
  term.
  The change across each transition across sets is lower bounded by
  $-\stepSize_{l_i} L$, analogous to the handling of the first and
  last term in \cref{eqn:diff_split_2}.
  Thus, the summation is lower bounded by $-T\stepSize_{k_s} L$ and
  $T$ is finite as $T \leqslant \kappa_\delta(s) - k_s < \infty$.
  Taking $\limsup$, the second summation in \cref{eqn:diff_split}
  vanishes as $\stepSize_{k_s} \to 0$. Combining both,
  \cref{eqn:non_stopping_2} follows.

  \textbf{Case 2: $\tilde{\genericPoint} \in \setStationaryPoints$}.
  \emph{($\mathcal{S}$ is at most $1$-dimensional around
    $\tilde{\genericPoint}$):}
  First, note that the connected component of $\mathcal{S}$ containing
  $\tilde{\genericPoint}$ can not have more than 1 degree of freedom
  at $\tilde{\genericPoint}$, since if $\tilde{\genericPoint}$ lies on
  a $2$-dimensional surface, then $\tilde{\genericPoint}$ must also be
  a local minimum, which is precluded by the fact that
  $\tilde{\genericPoint} \not \in \solutions$.
  Let $\delta_1 > 0$ be sufficiently small such that there are no
  $2$-dimensional regions in the connected component of $\mathcal{S}
  \intersection \oBall{\delta_1}{\tilde{\genericPoint}}$ containing
  $\tilde{\genericPoint}$. It is clear that the set $\mathcal{S}
  \intersection \oBall{\delta_1}{\tilde{\genericPoint}}$ has zero
  measure. Analogously to Case 1, let $\delta_2 > 0$ be sufficiently
  small such that there are no other inflection segments in
  $\oBall{\delta_2}{\tilde{\genericPoint}}$ other than the ones
  passing through $\tilde{\genericPoint}$. Define $\tilde{\delta} :=
  \min\{\delta_1, \delta_2\}$ and choose $\bar{\delta} =
  \tilde{\delta}/3$.

  \emph{(Construction of sets on which Line 8 is used):} Consider the
  sequence $\{\delta_{k}^{\mathrm{th}}\}_{k\geqslant 0}$ in \Norcent{}
  with $ \delta_k^{\mathrm{th}} \downarrow 0$ and
  $\delta_0^{\mathrm{th}} \ll 1$. For each $k \geqslant 0$, we
  construct the set $X^*_k$ as the connected component of $\{
  \genericPoint \mid \existsn g \in \partial
  \augmentedVisibilityMetric{\myDomain}(\genericPoint) \text{ s.t. }
  \| g \| \leqslant \delta_k^{\mathrm{th}} \}$ containing
  $\tilde{\genericPoint}$.  Note that $X^* := \lim_{k \to \infty}
  X^*_k = S$ is the connected component of $\mathcal{S}$ that contains
  $\tilde{\genericPoint}$. It follows that the sequence of sets
  $\{X^*_k\}_{k\geqslant0}$ monotonically shrink to $X^*$, and thus,
  $\mu(X^*_k) \downarrow 0$.

  \emph{(Iterates almost surely escape $X^*_{k_s}\intersection
  \oBall{\delta}{\genericPoint_{k_s}}$ in finite time):} 
  Next, we show that
  $\mathbb{P}(\existsn \bar{k} > k_s \text{ s.t. }
  \genericPoint_{\bar{k}} \not \in X^*_{k_s}\intersection
  \oBall{\delta}{\genericPoint_{k_s}}) = 1$. We reason by
  contradiction. Suppose this is not the case. Then, the event $E
  := \{\genericPoint_{\bar{k}} \in X^*_{k_s}\intersection
  \oBall{\delta}{\genericPoint_{k_s}}, \foralln \bar{k} > k_s\}$ has a
  non-zero probability, that is, $\mathbb{P}(E) > 0$. Next, notice
  that the iterates $\{\genericPoint_l\}_{l \verythinspace \geqslant
    \verythinspace k_s}$ form a martingale~\cite{PB:12}
  as $\mathbb{E}(\stepSizeB_l \hat{\nu}) = 0$ from the update in Line
  8 of \Norcent{} and the fact that $\hat{\nu}$ is chosen uniformly
  randomly on the unit circle.  As the martingale is bounded on the
  event $E$, it converges almost surely on $E$ due to Doob's
  martingale convergence theorem \cite{JD:53,PB:12} restricted to $E$.
  But almost sure convergence on a set of positive probability implies
  that the second moments on that set are uniformly bounded,
  contradicting the fact that
  $\mathbb{E}(\|\genericPoint_n - \genericPoint_{k_s}\|^2) =
  \sum_{l=k_s}^{n} b_k^2 \to \infty$.  Therefore, we conclude the
  finite-time escape from
  $X^*_{k_s}\intersection \oBall{\delta}{\genericPoint_{k_s}}$ with
  probability $1$.

  \emph{(Iterates almost surely escape connected components of
    $\oBall{\delta}{\genericPoint_{k_s}} \setminus X^*_{k_s}$ in finite
    time):}
  Given $\epsilon > 0$, there must exist $\tilde{r} \geqslant \bar{k}$
  such that $\mathbb{P}(\genericPoint_{\tilde{r}} \not\in
  X^*_{\tilde{r}}) = 1 - \epsilon$, as $\{X^*_k\}$ is a monotonically
  decreasing sequence of sets.  Then, either
  $\mathbb{P}(\genericPoint_{\tilde{r}} \not \in
  \oBall{\delta}{\tilde{\genericPoint}}) = 1 - \epsilon$ or
  $\mathbb{P}(\genericPoint_{\tilde{r}} \in
  \oBall{\delta}{\tilde{\genericPoint}} \setminus X^*_{\tilde{r}}) =
  1-\epsilon$.
  In the former case, $\kappa_\delta(s) \leqslant \tilde{r}$, and
  \cref{eqn:non_stopping_1} follows almost surely as the choice of
  $\epsilon$ is arbitrary. 
  In the latter case, the directional derivatives in
  $\oBall{\delta}{\tilde{\genericPoint}} \setminus X^*_{\tilde{r}}$
  are bounded away from 0. In fact $\|g\| >
  \delta_{\tilde{r}}^{\mathrm{th}}$ for all $g \in \partial
  \augmentedVisibilityMetric{\myDomain}(\genericPointQ)$ where
  $\genericPointQ \in \oBall{\delta}{\tilde{\genericPoint}} \setminus
  X^*_{\tilde{r}}$. This means that from $\tilde{r}$ onwards, we could
  follow the reasoning of Case 1, except for the possibility that the
  iterates enter $X^*_{k}$ again, for some $k \geqslant \tilde{r}$, in
  which case, they escape again to
  $\oBall{\delta}{\tilde{\genericPoint}} \setminus X^*_{\tilde{r}}$
  after a finite number of steps. In both cases, we leverage the fact
  that the choice of $\epsilon$ is arbitrary, and the statements hold
  with probability $1$.

  \emph{(Proof that the number of exits from and entering $X^*_k$
  are finite):}
  Next, we next show that the iterates can escape and enter $X^*_k$'s
  only a finite number of times. We reason by contradiction. Then,
  given $k \geqslant k_s$, for any point $\genericPoint_k \in
  \oBall{\delta}{\tilde{\genericPoint}} \setminus X^*_k$, there exists
  a smallest $\tilde{r} > k$ such that $\genericPoint_{\tilde{r}} \in
  X^*_{\tilde{r}}$.
  Following Case 1 over the connected component of
  $\oBall{\delta}{\tilde{\genericPoint}} \setminus X^*_k$ containing
  $\genericPoint_k$ which does not contain any stationary points by
  definition,
  $\augmentedVisibilityMetric{\myDomain}(\genericPoint_k) >
  \augmentedVisibilityMetric{\myDomain}(\genericPoint_{\tilde{r}})
  \geqslant 
  \augmentedVisibilityMetric{\myDomain}(\tilde{\genericPoint})  
  - \delta^{\mathrm{th}}_{\tilde{r}}
  d_{X^*_{\tilde{r}}}(\genericPoint_{\tilde{r}},
  \tilde{\genericPoint})$, where
  $d_{X^*_{\tilde{r}}}(\genericPoint_{\tilde{r}},
  \tilde{\genericPoint})$, the distance of $\tilde{\genericPoint}$
  from $\genericPoint_{\tilde{r}}$ along $X^*_{\tilde{r}}$, is
  bounded. As this must hold for all $k \geqslant k_s$ and
  $\delta^{\mathrm{th}}_{\tilde{r}} \to 0$, it follows that
  $\augmentedVisibilityMetric{\myDomain}(\genericPoint_k) >
  \augmentedVisibilityMetric{\myDomain}(\tilde{\genericPoint})$.  Note
  that the inequality remains strict even under $\limsup$ due to the
  strict decrease from $\genericPoint_k$ to $\genericPoint_{\bar{r}}$
  from Case~1.

  But then $\tilde{\genericPoint}$ is a local minimum, which
  contradicts the fact that $\tilde{\genericPoint} \not \in
  \solutions$.

  \emph{(Combining the above arguments to show that the iterates leave
  the $\delta$-ball in finite number of steps with probability $1$):}
  From our reasoning above, it follows that every time the iterates
  enter $X^*_k$, they exit after a finite number of steps with
  probability~$1$, and only a finite number of such entries and exits
  occur. Eventually, the iterates end up in connected component of
  $\oBall{\delta}{\genericPoint_{k_s}} \setminus X^*_k$ for some $k
  \geqslant k_s$ such that it never enters $X^*_k$ again with
  probability $1$.  But the elements of the generalized gradient are
  bounded away from zero within the connected component, and hence,
  the iterates must leave the set, analogously to Case 1. However, the
  only way for the iterates to leave the set is via the boundary of
  $\oBall{\delta}{\genericPoint_{k_s}}$, as the iterates never enter
  $X^*_k$ again. This means that there exists a finite index $\bar{r}$
  such that $\genericPoint_{\bar{r}} \not \in
  \oBall{\delta}{\genericPoint_{k_s}}$ with probability $1$, and
  \cref{eqn:non_stopping_1} follows, again with probability~$1$.
  
  \emph{(Combining the above arguments to show
    \cref*{eqn:non_stopping_2} holds with probability $1$):}
  Suppose we enter $\{X^*_l\}$ at iterations $l \in \{m_t\}_{t=0}^T$
  and exit at iterations $l \in \{n_t\}_{t=0}^T$, and that the local
  Lipschitz constant in $\oBall{\delta}{\genericPoint_{k_s}}$ is $L$.
  Then, following Case 1 for iterates over
  $\oBall{\delta}{\genericPoint_{k_s}} \setminus X^*_{(\cdot)}$ and
  bounding the change in the function values at transitions and the
  variation in the function values over $X^*_{(\cdot)}$,
  \begin{align*}
    \augmentedVisibilityMetric{\myDomain}(\genericPoint_{m_0})
    \! > \!
    \augmentedVisibilityMetric{\myDomain}(\genericPoint_{n_0}) \! - \! L \stepSize_{n_0} \! - \!
    \delta_{m_0}^{\mathrm{th}} \sum_{l=m_0}^{n_0} \stepSizeB_l 
    \! > \!
    \dots
    \! > \!
    \augmentedVisibilityMetric{\myDomain}(\genericPoint_{n_T}) 
    \! - \! L \sum_{t=0}^T \stepSize_{n_t} \! - \! \sum_{t=0}^T
    \delta_{m_t}^{\mathrm{th}} \sum_{l=m_t}^{n_t} \stepSizeB_l
  \end{align*}
  holds with probability $1$. 
  Combining this with
  \smash{$\augmentedVisibilityMetric{\myDomain}(\genericPoint_{k_s})
  \geqslant \augmentedVisibilityMetric{\myDomain}(\genericPoint_{m_0})
  - L \sum_{j={k_s}}^{m_0=1} \stepSize_j$} on one end and
  \smash{$\augmentedVisibilityMetric{\myDomain}(\genericPoint_{n_T})
  \geqslant
  \augmentedVisibilityMetric{\myDomain}(\genericPoint_{k_\delta(s)}) -
  L\sum_{j={n_T}}^{k_\delta(s)} \stepSize_j$} on the other end, we get
  \begin{align*}
    \augmentedVisibilityMetric{\myDomain}(\genericPoint_{k_s}) > 
    \augmentedVisibilityMetric{\myDomain}(\genericPoint_{k_\delta(s)}) 
    - L \sum_{j={k_s}}^{m_0-1} \stepSize_j - L \sum_{t=0}^T 
    \stepSize_{n_t} - \sum_{t=0}^T \delta_{m_t}^{\mathrm{th}} 
    \sum_{l=m_t}^{n_t} \stepSizeB_l - L\sum_{j={n_T}+1}^{k_\delta(s)} 
    \stepSize_j.
  \end{align*}
  Taking $\limsup$ throughout, \cref{eqn:non_stopping_2} follows with
  probability~$1$ from the fact that the last four terms go to zero
  due to stepsizes $\stepSize_l \downarrow 0$, $\delta^\mathrm{th}_l
  \downarrow 0$, $(n_t - m_t)$ being finite for every $t$, and $m_0$,
  $\kappa_\delta(s)$ and $T$ being finite as shown above. Note that
  the inequality remains strict even under $\limsup$ due to strict
  decrease in connected components of
  $\oBall{\delta}{\genericPoint_{k_s}} \setminus X^*_{(\cdot)}$ from
  Case 1. 

  \textbf{Case 3: $\tilde{\genericPoint} \in \real^2 \setminus
  (\myDomain \union \setNonsmoothPoints)$}.
  \emph{(Smooth case -- \cref*{eqn:non_stopping_1}):}
  Choose $\tilde{\delta} > 0$ sufficiently small such that
  $\oBall{\tilde{\delta}}{\tilde{\genericPoint}} \intersection
  \myDomain \intersection \setNonsmoothPoints = \varnothing$. Then,
  $\partial \augmentedVisibilityMetric{\myDomain}(\genericPointQ)$ is
  a singleton comprising its gradient, its projection
  $\proj{\myDomain}{\genericPointQ}$ is unique \cite[Corollary
  3.4.5]{PC-CS:04}, and $\langle \nabla
  \augmentedVisibilityMetric{\myDomain}(\genericPointQ),
  \genericPointQ - \proj{\myDomain}{\genericPointQ} \rangle = -1$ for
  all $\genericPointQ \in \oBall{\bar{\delta}}{\tilde{\genericPoint}}$
  as $\langle \nabla
  \visibilityMetric(\proj{\myDomain}{\genericPointQ}), \genericPointQ
  - \proj{\myDomain}{\genericPointQ} \rangle = 0$. Choose
  $\bar{\delta} = \tilde{\delta} / 2$ and $\bar{s}$ sufficiently large
  such that $\genericPoint_{k_s} \in
  \oBall{\bar{\delta}}{\tilde{\genericPoint}}$ for all $s \geqslant
  \bar{s}$.
  It follows that for every $\delta \in [0, \bar{\delta})$, the
  iterates follow $\genericPoint_{l+1} = \genericPoint_l + \stepSize_l
  \normalization(\proj{\myDomain}{\genericPoint_l} - \genericPoint_l)$
  for all $l \geqslant k_s$ as long as they remain in
  $\oBall{\delta}{\genericPoint_{k_s}}$. Using that
  $\proj{\myDomain}{\genericPoint_{l+1}} =
  \proj{\myDomain}{\genericPoint_l}$ and summing over $l$ from $k_s$
  to $k \geqslant k_s$, we get $\genericPoint_k = \genericPoint_{k_s}
  + \normalization(\proj{\myDomain}{\genericPoint_{k_s}} -
  \genericPoint_{k_s}) \sum_{l=k_s}^{k} \stepSize_l$. As the stepsize
  sequence $\stepSize_l$ is not summable, there is $\bar{r} > k_s$
  with $\genericPoint_{\bar{r}} \not \in
  \oBall{\delta}{\genericPoint_{k_s}}$, and \cref{eqn:non_stopping_1}
  holds with $\kappa_\delta(s) = \bar{r}$.

  \emph{(Smooth case -- \cref*{eqn:non_stopping_2}):}
  From the update above and $\langle \nabla
  \augmentedVisibilityMetric{\myDomain}(\genericPointQ),
  \genericPointQ - \proj{\myDomain}{\genericPointQ} \rangle = -1$, we
  have $\augmentedVisibilityMetric{\myDomain}(\genericPoint_{l+1}) =
  \augmentedVisibilityMetric{\myDomain}(\genericPoint_l) -
  \stepSize_l$. Again, summing over $l$ from $k_s$ to
  $\kappa_\delta(s)$, we get
  \smash{$\augmentedVisibilityMetric{\myDomain}(\genericPoint_{\kappa_\delta(s)})
  = \augmentedVisibilityMetric{\myDomain}(\genericPoint_{k_s}) -
  \sum_{l=k_s}^{\kappa_\delta(s)}\stepSize_l$}, and using
  \smash{$\sum_{l=k_s}^{\kappa_\delta(s)} \stepSize_l \geqslant
  \delta$}, we have
  $\augmentedVisibilityMetric{\myDomain}(\genericPoint_{\kappa_\delta(s)})
  \leqslant \augmentedVisibilityMetric{\myDomain}(\genericPoint_{k_s})
  - \delta$.  Taking $\limsup$ on both sides, $\delta >0$ stays, and
  \cref{eqn:non_stopping_2} follows.

  \textbf{Case 4: $\tilde{\genericPoint} \in \setNonsmoothPoints$}.
  \emph{(Non-smooth case):}
  Note that the iterates leave $\setNonsmoothPoints$ in a single step
  with probability~$1$, as $\mu(\setNonsmoothPoints) = 0$. This is
  because the distance function to a closed set (which $\myDomain$ is)
  is Lipschitz with constant $1$ \cite[Proposition 2.4.1]{FHC:90},
  which means that by Rademacher theorem \cite{HR:19}, the function is
  differentiable almost everywhere.
  Next, once the iterates leave $\setNonsmoothPoints$, they never
  enter the set again because for any $\genericPoint \not \in
  \myDomain$ and any of its projections
  $\proj{\myDomain}{\genericPoint}$, the distance function is
  differentiable along the open segment $(\genericPoint,
  \proj{\myDomain}{\genericPoint})$ \cite[Corollary 3.4.5]{PC-CS:04}. 
  It follows that with probability $1$, the iterates move along the
  segment $[\genericPoint, \proj{\myDomain}{\genericPoint}]$, and by
  following an analogous argument to Case~3, the result follows.
\end{proof}

Notice that the result does not rely on the monotonic decrease in the
function values along the whole sequence, in line with
\labelcref{enum:algo_non_mono}. We now prove that the optimal values
of the visibility metric are isolated.

\begin{lemma}[Optimal values are isolated]
  \label{lem:isolated_optimal_values}
  The set $\augmentedVisibilityMetric{\myDomain}^* = \{
  \augmentedVisibilityMetric{\myDomain}(\genericPoint) \mid
  \genericPoint \in \solutions \}$ has empty interior.
\end{lemma}
\begin{proof}
  Suppose, by contradiction, that
  $\augmentedVisibilityMetric{\myDomain}^*$ contains a non-trivial
  open interval $(a, b) \subseteq \real$. Then for every $\beta \in
  (a, b)$, there exists a point $\genericPoint_\beta \in \Gamma$ such
  that $\augmentedVisibilityMetric{\myDomain}(\genericPoint_\beta) =
  \beta$. Since each $\genericPoint_\beta \in \Gamma$ is assumed to be
  a local minimum, there exists an open neighborhood $U_\beta$ of
  $\genericPoint_\beta$ such that
  $\augmentedVisibilityMetric{\myDomain}(\genericPoint_\beta)
  \leqslant \augmentedVisibilityMetric{\myDomain}(\genericPoint)$ for
  all $\genericPoint \in U_\beta$. Since
  $\augmentedVisibilityMetric{\myDomain}(\genericPoint_\beta) = \beta$
  and $\beta$ varies over $(a, b)$, the points $\genericPoint_\beta$
  must all attain different values and are therefore distinct with
  disjoint open neighborhoods around them.  Thus, we obtain an
  uncountable collection $\{ U_\beta \}_{\beta \in (a, b)}$ of
  disjoint open subsets of $\real^2$. However, this contradicts the
  fact that $\real^2$ is a second-countable space and hence has at
  most countably many disjoint open subsets. Therefore,
  $\augmentedVisibilityMetric{\myDomain}^*$ cannot contain a
  non-trivial interval.
\end{proof}

Finally, we show the almost sure convergence of
\NorcentTO~\labelcref{enum:algo_avoid+conv} to the set of local
minimizers of $\visibilityMetric$.

\begin{theorem}[Convergence of \NorcentTO]
  \label{thm:convergence_norcent}
  \NorcentTO~converges to $\solutions$, the set of local minimizers of
  $\visibilityMetric$ with probability one. 
\end{theorem}
\begin{proof}
  The result follows by invoking the convergence result in \cite[\S
  4]{VM-AG-VN:87} (see also \cite{EN:79}) whose hypothesis we justify
  next. The normalization in Lines 8 and 10 of \NorcentTO~implies that
  $\| \genericPoint_{k+1} - \genericPoint_k \| \leqslant a_k
  \downarrow 0$. For any sequence $\{\genericPoint_k\}_{k \geqslant
  0}$ generated by initializing at $\genericPoint_0 \in \real^2$ and
  iterating~\Norcent, the remaining conditions are satisfied due to
  \Cref{prop:invariant_set}, \Cref{thm:non_stopping}, and
  \Cref{lem:isolated_optimal_values}. 
\end{proof}

For the maximization case, we can devise an analogous normalized
generalized gradient ascent algorithm by replacing `$-$' in Line 10 of
\Norcent~with~`$+$'.  The visibility metric under limited range
$\visibilityMetric_R$ can also be handled in an analogous fashion, by
replacing $\visibilityMetric$ in
\cref{eqn:augmented_visibility_metric,eqn:augmented_visibility_metric_max}
with $\visibilityMetric_R$, though care needs to be taken while
handling the stationary points $\setStationaryPoints$ in
\Cref{thm:non_stopping} which now lie on inflection curves as opposed
to inflection segments (cf.
\Cref{rem:inflection_curves_limited_fov_range}).  \Cref{fig:UL-RL}
shows the result of the execution of the algorithm for maximization in
a non-convex environment in the case of unlimited and limited
visibility range.

\begin{figure}[!h]
  \centering%
  \subfigure[Unlimited
  range]{\includegraphics[width=.46\linewidth]{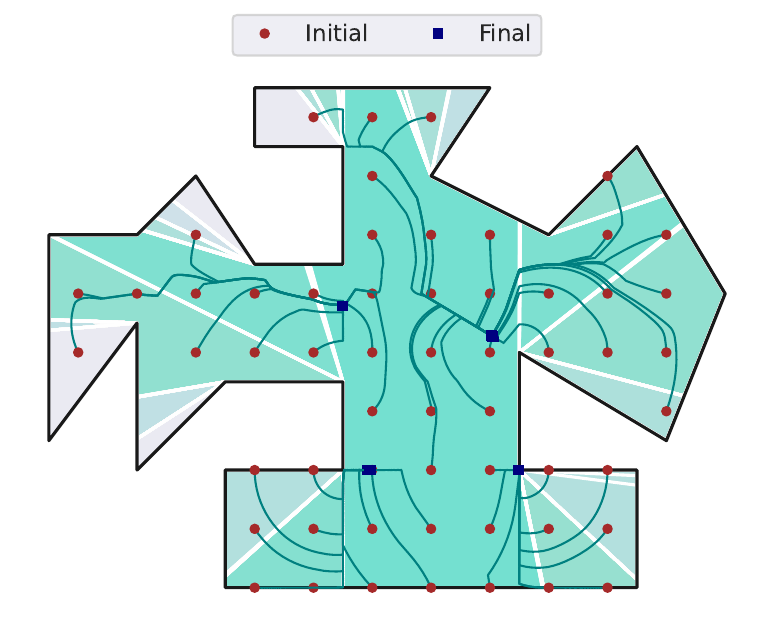}}
  \subfigure[Limited
  range]{\includegraphics[width=.46\linewidth]{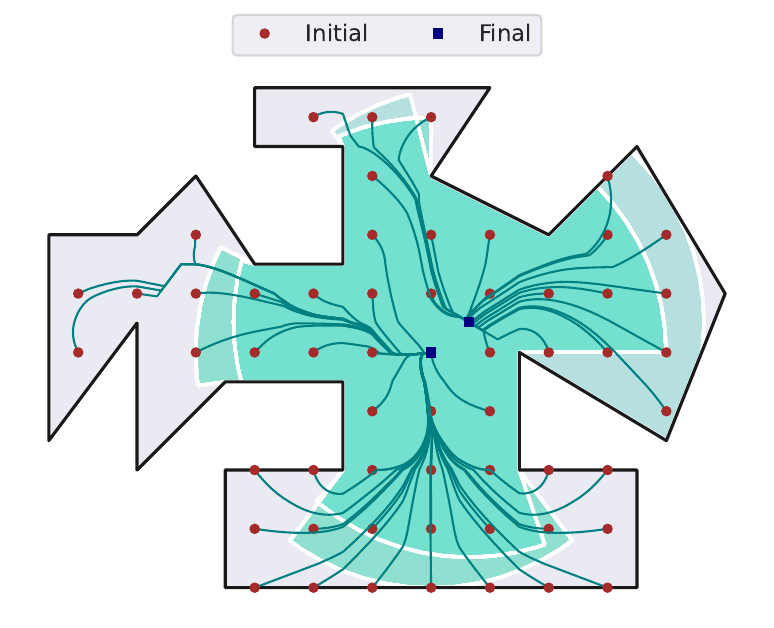}}
  \caption{Visibility maximization in a non-convex environment with
  the \NorcentTO. Red dots correspond to initial locations and blue
  squares to the attained local maximizers. The visibility regions
  from each maximizer are displayed in turquoise.}
  \label{fig:UL-RL}
\end{figure}


\section{Simulations}\label{sec:examples}

We have developed a python-based environment to simulate the proposed
\Norcent~using \emph{shapely} \cite{SG-CW-JB-MT-JA-BW:25} and
\emph{CGAL} \cite{MH-KH-FB-NX:24-cgal} for geometric computations.  We
consider two scenarios.

\subsection{Climbers hiding from a quadrotor}\label{sub:example_1}

In the first scenario, we consider expert mountain climbers trying to
evade a stealthy quadrotor. The key idea is that climbers can use the
visibility metric to perform evasive maneuvers without knowing the
position of the quadrotor. As seen in \Cref{fig:example_1}, we assume
that the mountain climber can move freely close to the ground
($\myDomain$ in green) while the quadrotor has height restrictions,
that is, its height $h \in [h_1, h_2]$ ($\advDomain$ in red).  The
climbers' implementation of \NorcentTO~to find good hiding spots is
shown in \Cref{fig:solution_1}. For one of the climbers, we show in
\Cref{fig:example_1_1} the regions in $\advDomain$ from where the
climbers can be seen at the beginning and the end of the algorithm
evolution, showing the decrease in the visibility metric until an
optimal hiding spot is found.

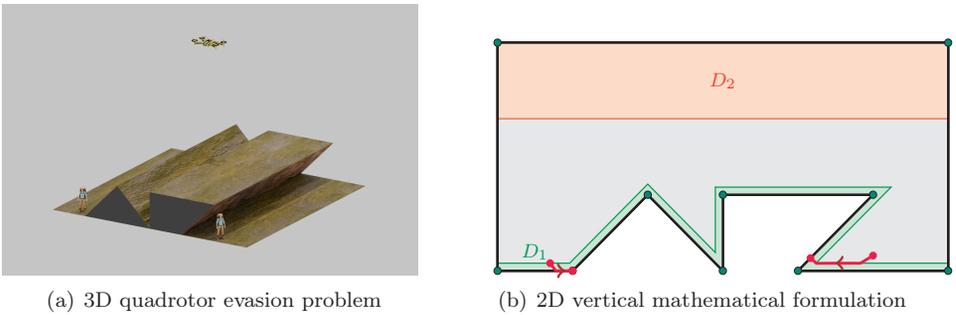
\begin{figure}[!h]
  \centering
  \subfigure[3D quadrotor evasion problem]{
    \includegraphics[width=0.425\linewidth]{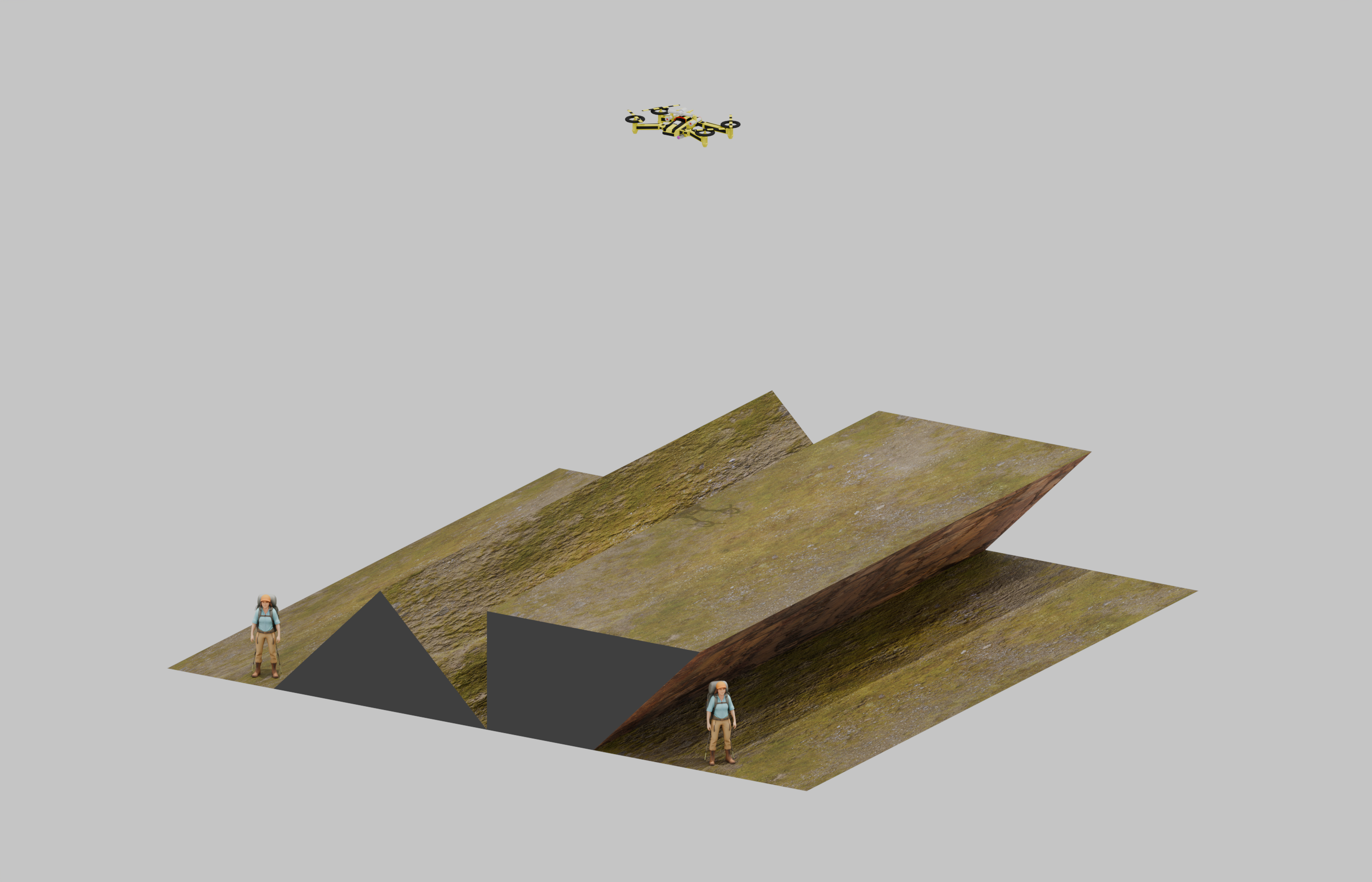}
    \label{fig:problem_1}
  }
  \subfigure[2D vertical mathematical formulation]{
    \label{fig:solution_1}
    \resizebox{0.51\linewidth}{!}{
      \centering
      \pgfdeclarelayer{background}
      \pgfdeclarelayer{foreground}
      \pgfsetlayers{background,main,foreground}
      \begin{tikzpicture}[scale=1.10]
        \tikzset{
          style 0/.style={inner sep=0pt, minimum size=3pt},
          style A/.style={shape=circle, fill=PineGreen, style 0},
          style B/.style={shape=circle, fill=WildStrawberry, style 0},
          style C/.style={shape=circle, fill=WildStrawberry!50!Red, inner sep=0pt, minimum size=2pt},
          font={\fontsize{8pt}{9.6}\selectfont}
        }

        \begin{pgfonlayer}{foreground}
          \node[draw, style A] (q1) at (0, 0) {};
          \node[draw, style A] (q2) at (1, 0) {};
          \node[draw, style A] (q3) at (2, 1) {};
          \node[draw, style A] (q4) at (3, 0) {};
          \node[draw, style A] (q5) at (3, 1) {};
          \node[draw, style A] (q6) at (5, 1) {};
          \node[draw, style A] (q7) at (4, 0) {};
          \node[draw, style A] (q8) at (6, 0) {};
          \node[draw, style A] (q9) at (6, 3) {};
          \node[draw, style A] (q0) at (0, 3) {};
        \end{pgfonlayer}

        \begin{pgfonlayer}{background}
          \tikzstyle{every path}=[draw, line width=1.0pt, color=black]
          \draw[fill = gray!20] (q1.center) -- (q2.center) -- (q3.center) --
          (q4.center) -- (q5.center) -- (q6.center) -- (q7.center) --
          (q8.center) -- (q9.center) -- (q0.center) -- cycle;
        \end{pgfonlayer}
        
        \begin{pgfonlayer}{background}
          \tikzstyle{every path}=[draw, line width=0.5pt, color=ForestGreen]
          \draw[fill = ForestGreen!20] (0.00, 0.00) -- (0.00, 0.10) -- (0.96, 0.10) -- (2.00, 1.14) -- (2.90, 0.24) -- (2.90, 1.10) -- (5.24, 1.10) -- (4.24, 0.10) -- (6.00, 0.10) -- (6.00, 0.00) -- (4.00, 0.00) -- (5.00, 1.00) -- (3.00, 1.00) -- (3.00, 0.00) -- (2.00, 1.00) -- (1.00, 0.00) -- (0.00, 0.00) -- cycle;
          \node[text = ForestGreen] at (0.5, 0.25) {$\myDomain$};
        \end{pgfonlayer}

        \begin{pgfonlayer}{background}
          \tikzstyle{every path}=[draw, line width=0.5pt, color=RedOrange]
          \draw[fill = RedOrange!20] (0.00, 2.00) -- (6.00, 2.00) -- (6.00, 3.00) -- (0.00, 3.00) -- cycle;
          \node[text = Red!50!RedOrange] at (3, 2.5) {$\advDomain$};
        \end{pgfonlayer}

        \begin{pgfonlayer}{main}
          \tikzstyle{every path}=[draw, line width=1.0pt, color=black]
          \draw (q1.center) -- (q2.center) -- (q3.center) -- (q4.center) --
          (q5.center) -- (q6.center) -- (q7.center) -- (q8.center) --
          (q9.center) -- (q0.center) -- cycle;
        \end{pgfonlayer}

        \begin{pgfonlayer}{foreground}
          \tikzstyle{every path}=[draw, line width=1.2pt, color=WildStrawberry!50!Red]
          \draw (0.7065, 0.0926) -- (0.7094, 0.0887) -- (0.7123,
          0.0847) -- (0.7182, 0.0769) -- (0.7239, 0.0690) -- (0.7297,
          0.0612) -- (0.7354, 0.0534) -- (0.7412, 0.0457) -- (0.7470,
          0.0381) -- (0.7528, 0.0305) -- (0.7586, 0.0230) -- (0.7645,
          0.0155) -- (0.7704, 0.0081) -- (0.7767, 0.0033) -- (0.7833,
          0.0009) -- (0.7900, -0.0014) -- (0.7971, -0.0011) --
          (0.8041, -0.0009) -- (0.8111, -0.0006) -- (0.8181, -0.0004)
          -- (0.8251, -0.0002) -- (0.8321, 0.0001) -- (0.8391, 0.0004)
          -- (0.8461, 0.0006) -- (0.8527, -0.0015) -- (0.8597,
          -0.0012) -- (0.8666, -0.0009) -- (0.8736, -0.0006) --
          (0.8805, -0.0003) -- (0.8874, -0.0000) -- (0.8943, 0.0003)
          -- (0.9010, -0.0017) -- (0.9079, -0.0014) -- (0.9147,
          -0.0011) -- (0.9216, -0.0007) -- (0.9285, -0.0004) --
          (0.9354, 0.0000) -- (0.9423, 0.0004) -- (0.9489, -0.0015) --
          (0.9557, -0.0012) -- (0.9626, -0.0008) -- (0.9694, -0.0004)
          -- (0.9762, 0.0000) -- (0.9831, 0.0004) -- (0.9899, 0.0008)
          -- (0.9940, 0.0015) -- (0.9982, 0.0022) -- (0.9996, 0.0032)
          -- (1.0011, 0.0041) -- (0.9998, 0.0053) -- (1.0010, 0.0041)
          -- (0.9998, 0.0053) -- (1.0010, 0.0041) -- (0.9998, 0.0053)
          -- (1.0010, 0.0040) -- (0.9997, 0.0052) -- (1.0009, 0.0040)
          -- (0.9997, 0.0052) -- (1.0009, 0.0040) -- (0.9997, 0.0052)
          -- (1.0009, 0.0039) -- (0.9996, 0.0051) -- (1.0008, 0.0039)
          -- (0.9996, 0.0051) -- (1.0008, 0.0039) -- (0.9996, 0.0051)
          -- (1.0008, 0.0039) -- (0.9996, 0.0050) -- (1.0007, 0.0038)
          -- (0.9995, 0.0050) -- (1.0007, 0.0038) -- (0.9995, 0.0050)
          -- (1.0007, 0.0038) -- (0.9995, 0.0049) -- (1.0006, 0.0037)
          -- (0.9994, 0.0049) -- (1.0006, 0.0037) -- (0.9994, 0.0049)
          -- (1.0006, 0.0037) -- (0.9994, 0.0048) -- (1.0005, 0.0037)
          -- (0.9994, 0.0048) -- (1.0005, 0.0036) -- (0.9993, 0.0048)
          -- (1.0005, 0.0036) -- (0.9993, 0.0047) -- (1.0004, 0.0036)
          -- (0.9993, 0.0047) -- (1.0004, 0.0035) -- (0.9992, 0.0047)
          -- (1.0004, 0.0035) -- (0.9992, 0.0047) -- (1.0003, 0.0035)
          -- (0.9992, 0.0046) -- (1.0003, 0.0035) -- (0.9992, 0.0046)
          -- (1.0003, 0.0034) -- (0.9991, 0.0046) -- (1.0003, 0.0034)
          -- (0.9991, 0.0045) -- (1.0002, 0.0034) -- (0.9991, 0.0045)
          -- (1.0002, 0.0034) -- (0.9991, 0.0045) -- (1.0002, 0.0033)
          -- (0.9990, 0.0044) -- (1.0001, 0.0033) -- (0.9990, 0.0044)
          -- (1.0001, 0.0033) -- (0.9990, 0.0044) -- (1.0001, 0.0033)
          -- (0.9990, 0.0044) -- (1.0001, 0.0032) -- (0.9989, 0.0043)
          -- (1.0000, 0.0032) -- (0.9989, 0.0043) -- (1.0000, 0.0032)
          -- (0.9989, 0.0043) -- (1.0000, 0.0032) -- (0.9989, 0.0042)
          -- (0.9999, 0.0031) -- (0.9988, 0.0042) -- (0.9999, 0.0031)
          -- (0.9988, 0.0042) -- (0.9999, 0.0031) -- (0.9988, 0.0042)
          -- (0.9999, 0.0030) -- (0.9988, 0.0041) -- (0.9998, 0.0030)
          -- (0.9987, 0.0041) -- (0.9998, 0.0030) -- (0.9987, 0.0041)
          -- (0.9998, 0.0030) -- (0.9987, 0.0040) -- (0.9997, 0.0029)
          -- (0.9986, 0.0040) -- (0.9997, 0.0029) -- (0.9986, 0.0040)
          -- (0.9997, 0.0029) -- (0.9986, 0.0040) -- (0.9997, 0.0029)
          -- (0.9986, 0.0039) -- (0.9996, 0.0029) -- (0.9986, 0.0039)
          -- (0.9996, 0.0028) -- (0.9985, 0.0039) -- (0.9996, 0.0028)
          -- (0.9985, 0.0039) -- (0.9995, 0.0028) -- (0.9985, 0.0038)
          -- (0.9995, 0.0028) -- (0.9985, 0.0038) -- (0.9995, 0.0027)
          -- (0.9984, 0.0038) -- (0.9995, 0.0027) -- (0.9984, 0.0037)
          -- (0.9994, 0.0027) -- (0.9984, 0.0037) -- (0.9994, 0.0027)
          -- (0.9984, 0.0037) -- (0.9994, 0.0026) -- (0.9983, 0.0037)
          -- (0.9994, 0.0026) -- (0.9983, 0.0036) -- (0.9993, 0.0026)
          -- (0.9983, 0.0036) -- (0.9993, 0.0026) -- (0.9983, 0.0036)
          -- (0.9993, 0.0025) -- (0.9982, 0.0036) -- (0.9993, 0.0025)
          -- (0.9982, 0.0035) -- (0.9992, 0.0025) -- (0.9982, 0.0035)
          -- (0.9992, 0.0025) -- (0.9982, 0.0035) -- (0.9992, 0.0024)
          -- (0.9981, 0.0035) -- (0.9992, 0.0024) -- (0.9981, 0.0034)
          -- (0.9991, 0.0024) -- (0.9981, 0.0034) -- (0.9991, 0.0024)
          -- (0.9981, 0.0034) -- (0.9991, 0.0024) -- (0.9981, 0.0034)
          -- (0.9991, 0.0023) -- (0.9980, 0.0033) -- (0.9990, 0.0023)
          -- (0.9980, 0.0033) -- (0.9990, 0.0023) -- (0.9980, 0.0033)
          -- (0.9990, 0.0023) -- (0.9980, 0.0033) -- (0.9990, 0.0022)
          -- (0.9979, 0.0032) -- (0.9989, 0.0022) -- (0.9979, 0.0032)
          -- (0.9989, 0.0022) -- (0.9979, 0.0032) -- (0.9989, 0.0022)
          -- (0.9979, 0.0032) -- (0.9989, 0.0022) -- (0.9978, 0.0031)
          -- (0.9988, 0.0021) -- (0.9978, 0.0031) -- (0.9988, 0.0021)
          -- (0.9978, 0.0031) -- (0.9988, 0.0021) -- (0.9978, 0.0031)
          -- (0.9988, 0.0021) -- (0.9978, 0.0030) -- (0.9987, 0.0020)
          -- (0.9977, 0.0030) -- (0.9987, 0.0020) -- (0.9977, 0.0030)
          -- (0.9987, 0.0020) -- (0.9977, 0.0030) -- (0.9987, 0.0020)
          -- (0.9977, 0.0030) -- (0.9986, 0.0020) -- (0.9977, 0.0029)
          -- (0.9986, 0.0019) -- (0.9976, 0.0029) -- (0.9986, 0.0019)
          -- (0.9976, 0.0029) -- (0.9986, 0.0019) -- (0.9976, 0.0029)
          -- (0.9985, 0.0019) -- (0.9976, 0.0028) -- (0.9985, 0.0018)
          -- (0.9975, 0.0028) -- (0.9985, 0.0018) -- (0.9975, 0.0028)
          -- (0.9985, 0.0018) -- (0.9975, 0.0028) -- (0.9985, 0.0018)
          -- (0.9975, 0.0027) -- (0.9984, 0.0018) -- (0.9975, 0.0027)
          -- (0.9984, 0.0017) -- (0.9974, 0.0027) -- (0.9984, 0.0017)
          -- (0.9974, 0.0027) -- (0.9984, 0.0017) -- (0.9974, 0.0027)
          -- (0.9983, 0.0017) -- (0.9974, 0.0027) -- (0.9985, 0.0018)
          -- (0.9976, 0.0028) -- (0.9986, 0.0019) -- (0.9977, 0.0029)
          -- (0.9986, 0.0019) -- (0.9977, 0.0029) -- (0.9986, 0.0019)
          -- (0.9976, 0.0029) -- (0.9986, 0.0019) -- (0.9976, 0.0028)
          -- (0.9986, 0.0019) -- (0.9976, 0.0028) -- (0.9985, 0.0019)
          -- (0.9976, 0.0028) -- (0.9985, 0.0018) -- (0.9976, 0.0028)
          -- (0.9985, 0.0018) -- (0.9975, 0.0027) -- (0.9985, 0.0018)
          -- (0.9980, 0.0023) -- (0.9988, 0.0015) -- (1, 0);
          \node[draw, shape=circle, fill=WildStrawberry!50!Red, style
          C] () at (0.7, 0.1) {};
          \node[draw, shape=circle, fill=WildStrawberry!50!Red, style
          C] () at (1, 0) {};
          \draw[-{to[round, scale=0.5]}, color=Maroon] (0.85, 0) -- (0.9, 0);

          \draw (4.9915, 0.1949) -- (4.9873, 0.1924) -- (4.9831,
          0.1899) -- (4.9747, 0.1848) -- (4.9664, 0.1798) -- (4.9580,
          0.1748) -- (4.9498, 0.1699) -- (4.9415, 0.1649) -- (4.9333,
          0.1600) -- (4.9251, 0.1550) -- (4.9169, 0.1501) -- (4.9088,
          0.1453) -- (4.9007, 0.1404) -- (4.8926, 0.1355) -- (4.8845,
          0.1307) -- (4.8765, 0.1259) -- (4.8684, 0.1211) -- (4.8605,
          0.1163) -- (4.8525, 0.1115) -- (4.8445, 0.1067) -- (4.8364,
          0.1033) -- (4.8280, 0.1011) -- (4.8197, 0.0990) -- (4.8111,
          0.0982) -- (4.8024, 0.0987) -- (4.7937, 0.0994) -- (4.7852,
          0.0987) -- (4.7766, 0.0994) -- (4.7680, 0.1002) -- (4.7596,
          0.0997) -- (4.7513, 0.0993) -- (4.7428, 0.1003) -- (4.7346,
          0.0999) -- (4.7263, 0.0996) -- (4.7181, 0.0993) -- (4.7099,
          0.0990) -- (4.7018, 0.0988) -- (4.6937, 0.0987) -- (4.6856,
          0.0986) -- (4.6774, 0.1000) -- (4.6694, 0.0999) -- (4.6614,
          0.0999) -- (4.6534, 0.1000) -- (4.6455, 0.1001) -- (4.6377,
          0.0986) -- (4.6299, 0.0988) -- (4.6220, 0.0990) -- (4.6142,
          0.0992) -- (4.6065, 0.0995) -- (4.5987, 0.0998) -- (4.5910,
          0.1001) -- (4.5834, 0.0988) -- (4.5757, 0.0992) -- (4.5682,
          0.0980) -- (4.5606, 0.0985) -- (4.5530, 0.0990) -- (4.5455,
          0.0995) -- (4.5380, 0.1000) -- (4.5306, 0.1004) -- (4.5232,
          0.1008) -- (4.5161, 0.0995) -- (4.5088, 0.0999) -- (4.5018,
          0.0984) -- (4.4947, 0.0988) -- (4.4877, 0.0991) -- (4.4807,
          0.0994) -- (4.4738, 0.0997) -- (4.4671, 0.0999) -- (4.4603,
          0.1001) -- (4.4539, 0.0983) -- (4.4473, 0.0984) -- (4.4407,
          0.0986) -- (4.4343, 0.0987) -- (4.4279, 0.0989) -- (4.4216,
          0.0989) -- (4.4154, 0.0990) -- (4.4092, 0.0991) -- (4.4031,
          0.0991) -- (4.3970, 0.0992) -- (4.3911, 0.0992) -- (4.3852,
          0.0992) -- (4.3793, 0.0992) -- (4.3736, 0.0992) -- (4.3679,
          0.0991) -- (4.3622, 0.0991) -- (4.3567, 0.0990) -- (4.3511,
          0.0990) -- (4.3457, 0.0990) -- (4.3403, 0.0989) -- (4.3350,
          0.0988) -- (4.3297, 0.0988) -- (4.3245, 0.0987) -- (4.3193,
          0.0986) -- (4.3142, 0.0986) -- (4.3092, 0.0985) -- (4.3042,
          0.0984) -- (4.2992, 0.0984) -- (4.2943, 0.0983) -- (4.2895,
          0.0982) -- (4.2847, 0.0981) -- (4.2800, 0.0981) -- (4.2753,
          0.0980) -- (4.2706, 0.0979) -- (4.2660, 0.0979) -- (4.2615,
          0.0978) -- (4.2570, 0.0977) -- (4.2525, 0.0977) -- (4.2472,
          0.1002) -- (4.2420, 0.1027) -- (4.2360, 0.1078) -- (4.2300,
          0.1129) -- (4.2240, 0.1181) -- (4.2181, 0.1233) -- (4.2123,
          0.1285) -- (4.2064, 0.1338) -- (4.2007, 0.1391) -- (4.1949,
          0.1444) -- (4.1893, 0.1497) -- (4.1836, 0.1551) -- (4.1780,
          0.1605) -- (4.1746, 0.1638) -- (4.1713, 0.1670) -- (4.1702,
          0.1681) -- (4.1691, 0.1692) -- (4.1702, 0.1681) -- (4.1691,
          0.1692) -- (4.1701, 0.1681) -- (4.1690, 0.1692) -- (4.1701,
          0.1681) -- (4.1690, 0.1691) -- (4.1701, 0.1680) -- (4.1690,
          0.1691) -- (4.1701, 0.1680) -- (4.1690, 0.1691) -- (4.1700,
          0.1680) -- (4.1690, 0.1691) -- (4.1700, 0.1680) -- (4.1689,
          0.1691) -- (4.1700, 0.1680) -- (4.1689, 0.1690) -- (4.1700,
          0.1679) -- (4.1689, 0.1690) -- (4.1700, 0.1679) -- (4.1689,
          0.1690) -- (4.1699, 0.1679) -- (4.1689, 0.1690) -- (4.1699,
          0.1679) -- (4.1688, 0.1689) -- (4.1699, 0.1679) -- (4.1688,
          0.1689) -- (4.1699, 0.1679) -- (4.1688, 0.1689) -- (4.1699,
          0.1678) -- (4.1688, 0.1689) -- (4.1698, 0.1678) -- (4.1688,
          0.1689) -- (4.1698, 0.1678) -- (4.1687, 0.1688) -- (4.1698,
          0.1678) -- (4.1687, 0.1688) -- (4.1698, 0.1678) -- (4.1687,
          0.1688) -- (4.1697, 0.1677) -- (4.1687, 0.1688) -- (4.1697,
          0.1677) -- (4.1687, 0.1688) -- (4.1697, 0.1677) -- (4.1687,
          0.1687) -- (4.1697, 0.1677) -- (4.1686, 0.1687) -- (4.1697,
          0.1677) -- (4.1686, 0.1687) -- (4.1696, 0.1677) -- (4.1686,
          0.1687) -- (4.1696, 0.1676) -- (4.1686, 0.1687) -- (4.1696,
          0.1676) -- (4.1686, 0.1686) -- (4.1696, 0.1676) -- (4.1685,
          0.1686) -- (4.1696, 0.1676) -- (4.1685, 0.1686) -- (4.1695,
          0.1676) -- (4.1685, 0.1686) -- (4.1695, 0.1675) -- (4.1685,
          0.1686) -- (4.1695, 0.1675) -- (4.1685, 0.1685) -- (4.1695,
          0.1675) -- (4.1685, 0.1685) -- (4.1695, 0.1675) -- (4.1684,
          0.1685) -- (4.1694, 0.1675) -- (4.1684, 0.1685) -- (4.1694,
          0.1675) -- (4.1684, 0.1685) -- (4.1694, 0.1674) -- (4.1684,
          0.1684) -- (4.1694, 0.1674) -- (4.1684, 0.1684) -- (4.1694,
          0.1674) -- (4.1684, 0.1684) -- (4.1694, 0.1674) -- (4.1683,
          0.1684) -- (4.1693, 0.1674) -- (4.1683, 0.1684) -- (4.1693,
          0.1674) -- (4.1683, 0.1683) -- (4.1693, 0.1673) -- (4.1683,
          0.1683) -- (4.1693, 0.1673) -- (4.1683, 0.1683) -- (4.1693,
          0.1673) -- (4.1683, 0.1683) -- (4.1692, 0.1673) -- (4.1682,
          0.1683) -- (4.1692, 0.1673) -- (4.1682, 0.1682) -- (4.1692,
          0.1673) -- (4.1682, 0.1682) -- (4.1692, 0.1672) -- (4.1682,
          0.1682) -- (4.1692, 0.1672) -- (4.1682, 0.1682) -- (4.1691,
          0.1672) -- (4.1682, 0.1682) -- (4.1691, 0.1672) -- (4.1681,
          0.1682) -- (4.1691, 0.1672) -- (4.1681, 0.1681) -- (4.1691,
          0.1672) -- (4.1681, 0.1681) -- (4.1691, 0.1671) -- (4.1681,
          0.1681) -- (4.1691, 0.1671) -- (4.1681, 0.1681) -- (4.1690,
          0.1671) -- (4.1681, 0.1681) -- (4.1690, 0.1671) -- (4.1680,
          0.1680) -- (4.1690, 0.1671) -- (4.1680, 0.1680) -- (4.1690,
          0.1671) -- (4.1680, 0.1680) -- (4.1690, 0.1670) -- (4.1680,
          0.1680) -- (4.1690, 0.1670) -- (4.1680, 0.1680) -- (4.1689,
          0.1670) -- (4.1680, 0.1680) -- (4.1689, 0.1670) -- (4.1680,
          0.1679) -- (4.1689, 0.1670) -- (4.1679, 0.1679) -- (4.1689,
          0.1670) -- (4.1679, 0.1679) -- (4.1689, 0.1669) -- (4.1679,
          0.1679) -- (4.1688, 0.1669) -- (4.1679, 0.1679) -- (4.1688,
          0.1669) -- (4.1679, 0.1679) -- (4.1688, 0.1669) -- (4.1679,
          0.1678) -- (4.1688, 0.1669) -- (4.1678, 0.1678) -- (4.1688,
          0.1669) -- (4.1678, 0.1678) -- (4.1688, 0.1669) -- (4.1678,
          0.1678) -- (4.1687, 0.1668) -- (4.1678, 0.1678) -- (4.1687,
          0.1668) -- (4.1678, 0.1678) -- (4.1687, 0.1668) -- (4.1678,
          0.1677) -- (4.1687, 0.1668) -- (4.1678, 0.1677) -- (4.1687,
          0.1668) -- (4.1677, 0.1677) -- (4.1687, 0.1668) -- (4.1677,
          0.1677) -- (4.1686, 0.1668) -- (4.1677, 0.1677) -- (4.1686,
          0.1667) -- (4.1677, 0.1677) -- (4.1686, 0.1667) -- (4.1677,
          0.1676) -- (4.1686, 0.1667) -- (4.1677, 0.1676) -- (4.1686,
          0.1667) -- (4.1681, 0.1671) -- (4.1689, 0.1664);
          \node[draw, shape=circle, fill=WildStrawberry!50!Red, style
          C] () at (5, 0.2) {};
          \node[draw, shape=circle, fill=WildStrawberry!50!Red, style
          C] () at (4.1667, 0.1667) {};
          \draw[-{to[round, scale=0.5]}, color=Maroon] (4.55, 0.1) -- (4.5, 0.1);
        \end{pgfonlayer}

      \end{tikzpicture}
    }%
  }
  \caption{Hiding from an quadrotor. The climbers can freely move
    along the terrain (green), while the quadrotor is restricted to a
    certain height range (red). (b) also shows the trajectories of the
    \NorcentTO~for two different initial positions of the climber.  }
    \label{fig:example_1}
\end{figure}

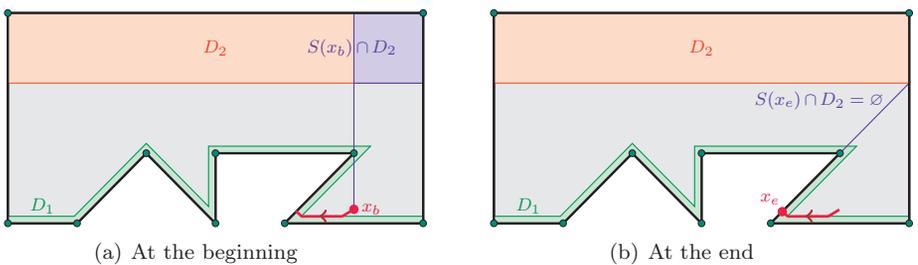
\begin{figure}[!h]
  \centering
  \subfigure[At the beginning]{
    \label{fig:example_1_before}
    \resizebox{0.47\linewidth}{!}{
      \centering
      \pgfdeclarelayer{background}
      \pgfdeclarelayer{foreground}
      \pgfsetlayers{background,main,foreground}
      \begin{tikzpicture}[scale=1.10]
        \tikzset{
          style 0/.style={inner sep=0pt, minimum size=3pt},
          style A/.style={shape=circle, fill=PineGreen, style 0},
          style B/.style={shape=circle, fill=WildStrawberry, style 0},
          style C/.style={shape=circle, fill=WildStrawberry!50!Red, inner sep=0pt, minimum size=2pt},
          font={\fontsize{8pt}{9.6}\selectfont}
        }

        \begin{pgfonlayer}{foreground}
          \node[draw, style A] (q1) at (0, 0) {};
          \node[draw, style A] (q2) at (1, 0) {};
          \node[draw, style A] (q3) at (2, 1) {};
          \node[draw, style A] (q4) at (3, 0) {};
          \node[draw, style A] (q5) at (3, 1) {};
          \node[draw, style A] (q6) at (5, 1) {};
          \node[draw, style A] (q7) at (4, 0) {};
          \node[draw, style A] (q8) at (6, 0) {};
          \node[draw, style A] (q9) at (6, 3) {};
          \node[draw, style A] (q0) at (0, 3) {};
        \end{pgfonlayer}

        \begin{pgfonlayer}{background}
          \tikzstyle{every path}=[draw, line width=1.0pt, color=black]
          \draw[fill = gray!20] (q1.center) -- (q2.center) -- (q3.center) --
          (q4.center) -- (q5.center) -- (q6.center) -- (q7.center) --
          (q8.center) -- (q9.center) -- (q0.center) -- cycle;
        \end{pgfonlayer}
        
        \begin{pgfonlayer}{background}
          \tikzstyle{every path}=[draw, line width=0.5pt, color=ForestGreen]
          \draw[fill = ForestGreen!20] (0.00, 0.00) -- (0.00, 0.10) -- (0.96, 0.10) -- (2.00, 1.14) -- (2.90, 0.24) -- (2.90, 1.10) -- (5.24, 1.10) -- (4.24, 0.10) -- (6.00, 0.10) -- (6.00, 0.00) -- (4.00, 0.00) -- (5.00, 1.00) -- (3.00, 1.00) -- (3.00, 0.00) -- (2.00, 1.00) -- (1.00, 0.00) -- (0.00, 0.00) -- cycle;
          \node[text = ForestGreen] at (0.5, 0.25) {$\myDomain$};
        \end{pgfonlayer}

        \begin{pgfonlayer}{background}
          \tikzstyle{every path}=[draw, line width=0.5pt, color=RedOrange]
          \draw[fill = RedOrange!20] (0.00, 2.00) -- (6.00, 2.00) -- (6.00, 3.00) -- (0.00, 3.00) -- cycle;
          \node[text = Red!50!RedOrange] at (3, 2.5) {$\advDomain$};
        \end{pgfonlayer}

        \begin{pgfonlayer}{main}
          \tikzstyle{every path}=[draw, line width=1.0pt, color=black]
          \draw (q1.center) -- (q2.center) -- (q3.center) -- (q4.center) --
          (q5.center) -- (q6.center) -- (q7.center) -- (q8.center) --
          (q9.center) -- (q0.center) -- cycle;
        \end{pgfonlayer}

        \begin{pgfonlayer}{foreground}
          \tikzstyle{every path}=[draw, line width=1.2pt, color=WildStrawberry!50!Red]

          \draw (4.9915, 0.1949) -- (4.9873, 0.1924) -- (4.9831,
          0.1899) -- (4.9747, 0.1848) -- (4.9664, 0.1798) -- (4.9580,
          0.1748) -- (4.9498, 0.1699) -- (4.9415, 0.1649) -- (4.9333,
          0.1600) -- (4.9251, 0.1550) -- (4.9169, 0.1501) -- (4.9088,
          0.1453) -- (4.9007, 0.1404) -- (4.8926, 0.1355) -- (4.8845,
          0.1307) -- (4.8765, 0.1259) -- (4.8684, 0.1211) -- (4.8605,
          0.1163) -- (4.8525, 0.1115) -- (4.8445, 0.1067) -- (4.8364,
          0.1033) -- (4.8280, 0.1011) -- (4.8197, 0.0990) -- (4.8111,
          0.0982) -- (4.8024, 0.0987) -- (4.7937, 0.0994) -- (4.7852,
          0.0987) -- (4.7766, 0.0994) -- (4.7680, 0.1002) -- (4.7596,
          0.0997) -- (4.7513, 0.0993) -- (4.7428, 0.1003) -- (4.7346,
          0.0999) -- (4.7263, 0.0996) -- (4.7181, 0.0993) -- (4.7099,
          0.0990) -- (4.7018, 0.0988) -- (4.6937, 0.0987) -- (4.6856,
          0.0986) -- (4.6774, 0.1000) -- (4.6694, 0.0999) -- (4.6614,
          0.0999) -- (4.6534, 0.1000) -- (4.6455, 0.1001) -- (4.6377,
          0.0986) -- (4.6299, 0.0988) -- (4.6220, 0.0990) -- (4.6142,
          0.0992) -- (4.6065, 0.0995) -- (4.5987, 0.0998) -- (4.5910,
          0.1001) -- (4.5834, 0.0988) -- (4.5757, 0.0992) -- (4.5682,
          0.0980) -- (4.5606, 0.0985) -- (4.5530, 0.0990) -- (4.5455,
          0.0995) -- (4.5380, 0.1000) -- (4.5306, 0.1004) -- (4.5232,
          0.1008) -- (4.5161, 0.0995) -- (4.5088, 0.0999) -- (4.5018,
          0.0984) -- (4.4947, 0.0988) -- (4.4877, 0.0991) -- (4.4807,
          0.0994) -- (4.4738, 0.0997) -- (4.4671, 0.0999) -- (4.4603,
          0.1001) -- (4.4539, 0.0983) -- (4.4473, 0.0984) -- (4.4407,
          0.0986) -- (4.4343, 0.0987) -- (4.4279, 0.0989) -- (4.4216,
          0.0989) -- (4.4154, 0.0990) -- (4.4092, 0.0991) -- (4.4031,
          0.0991) -- (4.3970, 0.0992) -- (4.3911, 0.0992) -- (4.3852,
          0.0992) -- (4.3793, 0.0992) -- (4.3736, 0.0992) -- (4.3679,
          0.0991) -- (4.3622, 0.0991) -- (4.3567, 0.0990) -- (4.3511,
          0.0990) -- (4.3457, 0.0990) -- (4.3403, 0.0989) -- (4.3350,
          0.0988) -- (4.3297, 0.0988) -- (4.3245, 0.0987) -- (4.3193,
          0.0986) -- (4.3142, 0.0986) -- (4.3092, 0.0985) -- (4.3042,
          0.0984) -- (4.2992, 0.0984) -- (4.2943, 0.0983) -- (4.2895,
          0.0982) -- (4.2847, 0.0981) -- (4.2800, 0.0981) -- (4.2753,
          0.0980) -- (4.2706, 0.0979) -- (4.2660, 0.0979) -- (4.2615,
          0.0978) -- (4.2570, 0.0977) -- (4.2525, 0.0977) -- (4.2472,
          0.1002) -- (4.2420, 0.1027) -- (4.2360, 0.1078) -- (4.2300,
          0.1129) -- (4.2240, 0.1181) -- (4.2181, 0.1233) -- (4.2123,
          0.1285) -- (4.2064, 0.1338) -- (4.2007, 0.1391) -- (4.1949,
          0.1444) -- (4.1893, 0.1497) -- (4.1836, 0.1551) -- (4.1780,
          0.1605) -- (4.1746, 0.1638) -- (4.1713, 0.1670) -- (4.1702,
          0.1681) -- (4.1691, 0.1692) -- (4.1702, 0.1681) -- (4.1691,
          0.1692) -- (4.1701, 0.1681) -- (4.1690, 0.1692) -- (4.1701,
          0.1681) -- (4.1690, 0.1691) -- (4.1701, 0.1680) -- (4.1690,
          0.1691) -- (4.1701, 0.1680) -- (4.1690, 0.1691) -- (4.1700,
          0.1680) -- (4.1690, 0.1691) -- (4.1700, 0.1680) -- (4.1689,
          0.1691) -- (4.1700, 0.1680) -- (4.1689, 0.1690) -- (4.1700,
          0.1679) -- (4.1689, 0.1690) -- (4.1700, 0.1679) -- (4.1689,
          0.1690) -- (4.1699, 0.1679) -- (4.1689, 0.1690) -- (4.1699,
          0.1679) -- (4.1688, 0.1689) -- (4.1699, 0.1679) -- (4.1688,
          0.1689) -- (4.1699, 0.1679) -- (4.1688, 0.1689) -- (4.1699,
          0.1678) -- (4.1688, 0.1689) -- (4.1698, 0.1678) -- (4.1688,
          0.1689) -- (4.1698, 0.1678) -- (4.1687, 0.1688) -- (4.1698,
          0.1678) -- (4.1687, 0.1688) -- (4.1698, 0.1678) -- (4.1687,
          0.1688) -- (4.1697, 0.1677) -- (4.1687, 0.1688) -- (4.1697,
          0.1677) -- (4.1687, 0.1688) -- (4.1697, 0.1677) -- (4.1687,
          0.1687) -- (4.1697, 0.1677) -- (4.1686, 0.1687) -- (4.1697,
          0.1677) -- (4.1686, 0.1687) -- (4.1696, 0.1677) -- (4.1686,
          0.1687) -- (4.1696, 0.1676) -- (4.1686, 0.1687) -- (4.1696,
          0.1676) -- (4.1686, 0.1686) -- (4.1696, 0.1676) -- (4.1685,
          0.1686) -- (4.1696, 0.1676) -- (4.1685, 0.1686) -- (4.1695,
          0.1676) -- (4.1685, 0.1686) -- (4.1695, 0.1675) -- (4.1685,
          0.1686) -- (4.1695, 0.1675) -- (4.1685, 0.1685) -- (4.1695,
          0.1675) -- (4.1685, 0.1685) -- (4.1695, 0.1675) -- (4.1684,
          0.1685) -- (4.1694, 0.1675) -- (4.1684, 0.1685) -- (4.1694,
          0.1675) -- (4.1684, 0.1685) -- (4.1694, 0.1674) -- (4.1684,
          0.1684) -- (4.1694, 0.1674) -- (4.1684, 0.1684) -- (4.1694,
          0.1674) -- (4.1684, 0.1684) -- (4.1694, 0.1674) -- (4.1683,
          0.1684) -- (4.1693, 0.1674) -- (4.1683, 0.1684) -- (4.1693,
          0.1674) -- (4.1683, 0.1683) -- (4.1693, 0.1673) -- (4.1683,
          0.1683) -- (4.1693, 0.1673) -- (4.1683, 0.1683) -- (4.1693,
          0.1673) -- (4.1683, 0.1683) -- (4.1692, 0.1673) -- (4.1682,
          0.1683) -- (4.1692, 0.1673) -- (4.1682, 0.1682) -- (4.1692,
          0.1673) -- (4.1682, 0.1682) -- (4.1692, 0.1672) -- (4.1682,
          0.1682) -- (4.1692, 0.1672) -- (4.1682, 0.1682) -- (4.1691,
          0.1672) -- (4.1682, 0.1682) -- (4.1691, 0.1672) -- (4.1681,
          0.1682) -- (4.1691, 0.1672) -- (4.1681, 0.1681) -- (4.1691,
          0.1672) -- (4.1681, 0.1681) -- (4.1691, 0.1671) -- (4.1681,
          0.1681) -- (4.1691, 0.1671) -- (4.1681, 0.1681) -- (4.1690,
          0.1671) -- (4.1681, 0.1681) -- (4.1690, 0.1671) -- (4.1680,
          0.1680) -- (4.1690, 0.1671) -- (4.1680, 0.1680) -- (4.1690,
          0.1671) -- (4.1680, 0.1680) -- (4.1690, 0.1670) -- (4.1680,
          0.1680) -- (4.1690, 0.1670) -- (4.1680, 0.1680) -- (4.1689,
          0.1670) -- (4.1680, 0.1680) -- (4.1689, 0.1670) -- (4.1680,
          0.1679) -- (4.1689, 0.1670) -- (4.1679, 0.1679) -- (4.1689,
          0.1670) -- (4.1679, 0.1679) -- (4.1689, 0.1669) -- (4.1679,
          0.1679) -- (4.1688, 0.1669) -- (4.1679, 0.1679) -- (4.1688,
          0.1669) -- (4.1679, 0.1679) -- (4.1688, 0.1669) -- (4.1679,
          0.1678) -- (4.1688, 0.1669) -- (4.1678, 0.1678) -- (4.1688,
          0.1669) -- (4.1678, 0.1678) -- (4.1688, 0.1669) -- (4.1678,
          0.1678) -- (4.1687, 0.1668) -- (4.1678, 0.1678) -- (4.1687,
          0.1668) -- (4.1678, 0.1678) -- (4.1687, 0.1668) -- (4.1678,
          0.1677) -- (4.1687, 0.1668) -- (4.1678, 0.1677) -- (4.1687,
          0.1668) -- (4.1677, 0.1677) -- (4.1687, 0.1668) -- (4.1677,
          0.1677) -- (4.1686, 0.1668) -- (4.1677, 0.1677) -- (4.1686,
          0.1667) -- (4.1677, 0.1677) -- (4.1686, 0.1667) -- (4.1677,
          0.1676) -- (4.1686, 0.1667) -- (4.1677, 0.1676) -- (4.1686,
          0.1667) -- (4.1681, 0.1671) -- (4.1689, 0.1664);
          \node[draw, shape=circle, fill=WildStrawberry!50!Red, style 0] (t1) at (5, 0.2) {};
          \node[text = WildStrawberry!50!Red] at
          ([shift={(0:0.25)}]{t1}) {$\genericPoint_b$};
          \draw[-{to[round, scale=0.5]}, color=Maroon] (4.52, 0.1) -- (4.5, 0.1);
        \end{pgfonlayer}

        \begin{pgfonlayer}{background}
          \tikzstyle{every path}=[draw, line width=0.5pt, color=Violet]
          \draw[fill = Violet!20] (6.00000, 3.00000) -- (6.00000, 2.00000) -- (5.00000, 2.00000) -- (5.00000, 3.00000) -- (6.00000, 3.00000) -- cycle;
          \node[text = Violet] (l2) at (4.97 ,2.5) {$\visibilityPolygon{\genericPoint_b} \intersection \advDomain$};
          \draw (5, 0.2) -- (5, 2);
        \end{pgfonlayer}

      \end{tikzpicture}
    }%
  }
  %
  \subfigure[At the end]{
    \label{fig:example_1_after}
    \resizebox{0.47\linewidth}{!}{
      \centering
      \pgfdeclarelayer{background}
      \pgfdeclarelayer{foreground}
      \pgfsetlayers{background,main,foreground}
      \begin{tikzpicture}[scale=1.10]
        \tikzset{
          style 0/.style={inner sep=0pt, minimum size=3pt},
          style A/.style={shape=circle, fill=PineGreen, style 0},
          style B/.style={shape=circle, fill=WildStrawberry, style 0},
          style C/.style={shape=circle, fill=WildStrawberry!50!Red, inner sep=0pt, minimum size=2pt},
          font={\fontsize{8pt}{9.6}\selectfont}
        }

        \begin{pgfonlayer}{foreground}
          \node[draw, style A] (q1) at (0, 0) {};
          \node[draw, style A] (q2) at (1, 0) {};
          \node[draw, style A] (q3) at (2, 1) {};
          \node[draw, style A] (q4) at (3, 0) {};
          \node[draw, style A] (q5) at (3, 1) {};
          \node[draw, style A] (q6) at (5, 1) {};
          \node[draw, style A] (q7) at (4, 0) {};
          \node[draw, style A] (q8) at (6, 0) {};
          \node[draw, style A] (q9) at (6, 3) {};
          \node[draw, style A] (q0) at (0, 3) {};
        \end{pgfonlayer}

        \begin{pgfonlayer}{background}
          \tikzstyle{every path}=[draw, line width=1.0pt, color=black]
          \draw[fill = gray!20] (q1.center) -- (q2.center) -- (q3.center) --
          (q4.center) -- (q5.center) -- (q6.center) -- (q7.center) --
          (q8.center) -- (q9.center) -- (q0.center) -- cycle;
        \end{pgfonlayer}
        
        \begin{pgfonlayer}{background}
          \tikzstyle{every path}=[draw, line width=0.5pt, color=ForestGreen]
          \draw[fill = ForestGreen!20] (0.00, 0.00) -- (0.00, 0.10) -- (0.96, 0.10) -- (2.00, 1.14) -- (2.90, 0.24) -- (2.90, 1.10) -- (5.24, 1.10) -- (4.24, 0.10) -- (6.00, 0.10) -- (6.00, 0.00) -- (4.00, 0.00) -- (5.00, 1.00) -- (3.00, 1.00) -- (3.00, 0.00) -- (2.00, 1.00) -- (1.00, 0.00) -- (0.00, 0.00) -- cycle;
          \node[text = ForestGreen] at (0.5, 0.25) {$\myDomain$};
        \end{pgfonlayer}

        \begin{pgfonlayer}{background}
          \tikzstyle{every path}=[draw, line width=0.5pt, color=RedOrange]
          \draw[fill = RedOrange!20] (0.00, 2.00) -- (6.00, 2.00) -- (6.00, 3.00) -- (0.00, 3.00) -- cycle;
          \node[text = Red!50!RedOrange] at (3, 2.5) {$\advDomain$};
        \end{pgfonlayer}

        \begin{pgfonlayer}{main}
          \tikzstyle{every path}=[draw, line width=1.0pt, color=black]
          \draw (q1.center) -- (q2.center) -- (q3.center) -- (q4.center) --
          (q5.center) -- (q6.center) -- (q7.center) -- (q8.center) --
          (q9.center) -- (q0.center) -- cycle;
        \end{pgfonlayer}

        \begin{pgfonlayer}{foreground}
          \tikzstyle{every path}=[draw, line width=1.2pt, color=WildStrawberry!50!Red]

          \draw (4.9915, 0.1949) -- (4.9873, 0.1924) -- (4.9831,
          0.1899) -- (4.9747, 0.1848) -- (4.9664, 0.1798) -- (4.9580,
          0.1748) -- (4.9498, 0.1699) -- (4.9415, 0.1649) -- (4.9333,
          0.1600) -- (4.9251, 0.1550) -- (4.9169, 0.1501) -- (4.9088,
          0.1453) -- (4.9007, 0.1404) -- (4.8926, 0.1355) -- (4.8845,
          0.1307) -- (4.8765, 0.1259) -- (4.8684, 0.1211) -- (4.8605,
          0.1163) -- (4.8525, 0.1115) -- (4.8445, 0.1067) -- (4.8364,
          0.1033) -- (4.8280, 0.1011) -- (4.8197, 0.0990) -- (4.8111,
          0.0982) -- (4.8024, 0.0987) -- (4.7937, 0.0994) -- (4.7852,
          0.0987) -- (4.7766, 0.0994) -- (4.7680, 0.1002) -- (4.7596,
          0.0997) -- (4.7513, 0.0993) -- (4.7428, 0.1003) -- (4.7346,
          0.0999) -- (4.7263, 0.0996) -- (4.7181, 0.0993) -- (4.7099,
          0.0990) -- (4.7018, 0.0988) -- (4.6937, 0.0987) -- (4.6856,
          0.0986) -- (4.6774, 0.1000) -- (4.6694, 0.0999) -- (4.6614,
          0.0999) -- (4.6534, 0.1000) -- (4.6455, 0.1001) -- (4.6377,
          0.0986) -- (4.6299, 0.0988) -- (4.6220, 0.0990) -- (4.6142,
          0.0992) -- (4.6065, 0.0995) -- (4.5987, 0.0998) -- (4.5910,
          0.1001) -- (4.5834, 0.0988) -- (4.5757, 0.0992) -- (4.5682,
          0.0980) -- (4.5606, 0.0985) -- (4.5530, 0.0990) -- (4.5455,
          0.0995) -- (4.5380, 0.1000) -- (4.5306, 0.1004) -- (4.5232,
          0.1008) -- (4.5161, 0.0995) -- (4.5088, 0.0999) -- (4.5018,
          0.0984) -- (4.4947, 0.0988) -- (4.4877, 0.0991) -- (4.4807,
          0.0994) -- (4.4738, 0.0997) -- (4.4671, 0.0999) -- (4.4603,
          0.1001) -- (4.4539, 0.0983) -- (4.4473, 0.0984) -- (4.4407,
          0.0986) -- (4.4343, 0.0987) -- (4.4279, 0.0989) -- (4.4216,
          0.0989) -- (4.4154, 0.0990) -- (4.4092, 0.0991) -- (4.4031,
          0.0991) -- (4.3970, 0.0992) -- (4.3911, 0.0992) -- (4.3852,
          0.0992) -- (4.3793, 0.0992) -- (4.3736, 0.0992) -- (4.3679,
          0.0991) -- (4.3622, 0.0991) -- (4.3567, 0.0990) -- (4.3511,
          0.0990) -- (4.3457, 0.0990) -- (4.3403, 0.0989) -- (4.3350,
          0.0988) -- (4.3297, 0.0988) -- (4.3245, 0.0987) -- (4.3193,
          0.0986) -- (4.3142, 0.0986) -- (4.3092, 0.0985) -- (4.3042,
          0.0984) -- (4.2992, 0.0984) -- (4.2943, 0.0983) -- (4.2895,
          0.0982) -- (4.2847, 0.0981) -- (4.2800, 0.0981) -- (4.2753,
          0.0980) -- (4.2706, 0.0979) -- (4.2660, 0.0979) -- (4.2615,
          0.0978) -- (4.2570, 0.0977) -- (4.2525, 0.0977) -- (4.2472,
          0.1002) -- (4.2420, 0.1027) -- (4.2360, 0.1078) -- (4.2300,
          0.1129) -- (4.2240, 0.1181) -- (4.2181, 0.1233) -- (4.2123,
          0.1285) -- (4.2064, 0.1338) -- (4.2007, 0.1391) -- (4.1949,
          0.1444) -- (4.1893, 0.1497) -- (4.1836, 0.1551) -- (4.1780,
          0.1605) -- (4.1746, 0.1638) -- (4.1713, 0.1670) -- (4.1702,
          0.1681) -- (4.1691, 0.1692) -- (4.1702, 0.1681) -- (4.1691,
          0.1692) -- (4.1701, 0.1681) -- (4.1690, 0.1692) -- (4.1701,
          0.1681) -- (4.1690, 0.1691) -- (4.1701, 0.1680) -- (4.1690,
          0.1691) -- (4.1701, 0.1680) -- (4.1690, 0.1691) -- (4.1700,
          0.1680) -- (4.1690, 0.1691) -- (4.1700, 0.1680) -- (4.1689,
          0.1691) -- (4.1700, 0.1680) -- (4.1689, 0.1690) -- (4.1700,
          0.1679) -- (4.1689, 0.1690) -- (4.1700, 0.1679) -- (4.1689,
          0.1690) -- (4.1699, 0.1679) -- (4.1689, 0.1690) -- (4.1699,
          0.1679) -- (4.1688, 0.1689) -- (4.1699, 0.1679) -- (4.1688,
          0.1689) -- (4.1699, 0.1679) -- (4.1688, 0.1689) -- (4.1699,
          0.1678) -- (4.1688, 0.1689) -- (4.1698, 0.1678) -- (4.1688,
          0.1689) -- (4.1698, 0.1678) -- (4.1687, 0.1688) -- (4.1698,
          0.1678) -- (4.1687, 0.1688) -- (4.1698, 0.1678) -- (4.1687,
          0.1688) -- (4.1697, 0.1677) -- (4.1687, 0.1688) -- (4.1697,
          0.1677) -- (4.1687, 0.1688) -- (4.1697, 0.1677) -- (4.1687,
          0.1687) -- (4.1697, 0.1677) -- (4.1686, 0.1687) -- (4.1697,
          0.1677) -- (4.1686, 0.1687) -- (4.1696, 0.1677) -- (4.1686,
          0.1687) -- (4.1696, 0.1676) -- (4.1686, 0.1687) -- (4.1696,
          0.1676) -- (4.1686, 0.1686) -- (4.1696, 0.1676) -- (4.1685,
          0.1686) -- (4.1696, 0.1676) -- (4.1685, 0.1686) -- (4.1695,
          0.1676) -- (4.1685, 0.1686) -- (4.1695, 0.1675) -- (4.1685,
          0.1686) -- (4.1695, 0.1675) -- (4.1685, 0.1685) -- (4.1695,
          0.1675) -- (4.1685, 0.1685) -- (4.1695, 0.1675) -- (4.1684,
          0.1685) -- (4.1694, 0.1675) -- (4.1684, 0.1685) -- (4.1694,
          0.1675) -- (4.1684, 0.1685) -- (4.1694, 0.1674) -- (4.1684,
          0.1684) -- (4.1694, 0.1674) -- (4.1684, 0.1684) -- (4.1694,
          0.1674) -- (4.1684, 0.1684) -- (4.1694, 0.1674) -- (4.1683,
          0.1684) -- (4.1693, 0.1674) -- (4.1683, 0.1684) -- (4.1693,
          0.1674) -- (4.1683, 0.1683) -- (4.1693, 0.1673) -- (4.1683,
          0.1683) -- (4.1693, 0.1673) -- (4.1683, 0.1683) -- (4.1693,
          0.1673) -- (4.1683, 0.1683) -- (4.1692, 0.1673) -- (4.1682,
          0.1683) -- (4.1692, 0.1673) -- (4.1682, 0.1682) -- (4.1692,
          0.1673) -- (4.1682, 0.1682) -- (4.1692, 0.1672) -- (4.1682,
          0.1682) -- (4.1692, 0.1672) -- (4.1682, 0.1682) -- (4.1691,
          0.1672) -- (4.1682, 0.1682) -- (4.1691, 0.1672) -- (4.1681,
          0.1682) -- (4.1691, 0.1672) -- (4.1681, 0.1681) -- (4.1691,
          0.1672) -- (4.1681, 0.1681) -- (4.1691, 0.1671) -- (4.1681,
          0.1681) -- (4.1691, 0.1671) -- (4.1681, 0.1681) -- (4.1690,
          0.1671) -- (4.1681, 0.1681) -- (4.1690, 0.1671) -- (4.1680,
          0.1680) -- (4.1690, 0.1671) -- (4.1680, 0.1680) -- (4.1690,
          0.1671) -- (4.1680, 0.1680) -- (4.1690, 0.1670) -- (4.1680,
          0.1680) -- (4.1690, 0.1670) -- (4.1680, 0.1680) -- (4.1689,
          0.1670) -- (4.1680, 0.1680) -- (4.1689, 0.1670) -- (4.1680,
          0.1679) -- (4.1689, 0.1670) -- (4.1679, 0.1679) -- (4.1689,
          0.1670) -- (4.1679, 0.1679) -- (4.1689, 0.1669) -- (4.1679,
          0.1679) -- (4.1688, 0.1669) -- (4.1679, 0.1679) -- (4.1688,
          0.1669) -- (4.1679, 0.1679) -- (4.1688, 0.1669) -- (4.1679,
          0.1678) -- (4.1688, 0.1669) -- (4.1678, 0.1678) -- (4.1688,
          0.1669) -- (4.1678, 0.1678) -- (4.1688, 0.1669) -- (4.1678,
          0.1678) -- (4.1687, 0.1668) -- (4.1678, 0.1678) -- (4.1687,
          0.1668) -- (4.1678, 0.1678) -- (4.1687, 0.1668) -- (4.1678,
          0.1677) -- (4.1687, 0.1668) -- (4.1678, 0.1677) -- (4.1687,
          0.1668) -- (4.1677, 0.1677) -- (4.1687, 0.1668) -- (4.1677,
          0.1677) -- (4.1686, 0.1668) -- (4.1677, 0.1677) -- (4.1686,
          0.1667) -- (4.1677, 0.1677) -- (4.1686, 0.1667) -- (4.1677,
          0.1676) -- (4.1686, 0.1667) -- (4.1677, 0.1676) -- (4.1686,
          0.1667) -- (4.1681, 0.1671) -- (4.1689, 0.1664);
          \node[draw, shape=circle, fill=WildStrawberry!50!Red, style 0] (t1) at (4.1667, 0.1667) {};
          \node[text = WildStrawberry!50!Red] at
          ([shift={(135:0.25)}]{t1}) {$\genericPoint_e$};
          \draw[-{to[round, scale=0.5]}, color=Maroon] (4.52, 0.1) -- (4.5, 0.1);
        \end{pgfonlayer}

        \begin{pgfonlayer}{background}
          \tikzstyle{every path}=[draw, line width=0.5pt, color=Violet]
          \node[text = Violet] (l2) at (4.7, 1.75)
          {$\visibilityPolygon{\genericPoint_e} \intersection 
          \advDomain = \varnothing$};
          \draw (4.1667, 0.1667) -- (6, 2);
        \end{pgfonlayer}

      \end{tikzpicture}
    }%
  }
  \caption{Visibility metric at the beginning and the end of
  \NorcentTO~for a climber hiding from an invisible quadrotor. Note
  that the initial condition is outside $\myDomain$, and the
  algorithm balances between getting into $\myDomain$ and reducing
  chances of detection. Once it gets into the set, it does not leave
  it. }
  \label{fig:example_1_1}
\end{figure}

\subsection{Robot seeking another robot in hide and seek}
\label{sub:example_2}

For the second scenario, we consider robots playing hide and seek,
albeit with a twist--- seekers are restricted to the green domain
$\myDomain$ and the hiders to the red domain $\advDomain$, as
illustrated in \Cref{fig:example_2}. The wheeled seeking robots are
equipped with Lidar sensors, and in order to maximize chances of
detection of the walking robot, they implement the \NorcentTO. The
trajectories of their implementation starting from four different
initializations are shown in \Cref{fig:example_2}.
\Cref{fig:example_2_1} shows how the region of the red domain
$\advDomain$ seen by the wheeled robots changes for one of the
initializations.

\begin{figure}[!h]
  \centering
  \subfigure[3D hide-and-seek problem]{
    \includegraphics[width=0.425\linewidth]{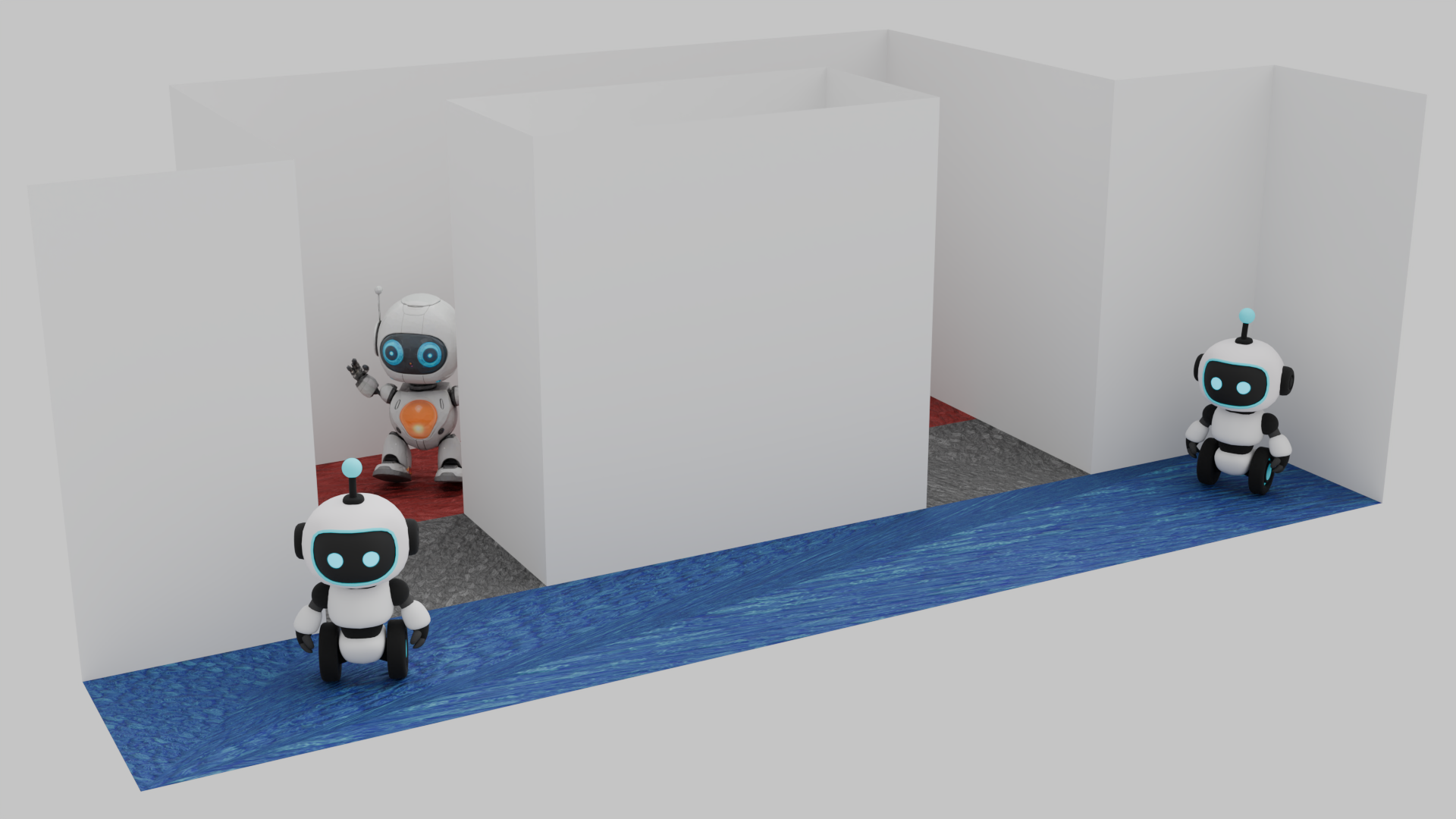}
    \label{fig:problem_2}
  }
  \subfigure[2D horizontal mathematical formulation]{
    \label{fig:solution_2}
    \resizebox{0.51\linewidth}{!}{
      \centering
      \pgfdeclarelayer{background}
      \pgfdeclarelayer{foreground}
      \pgfsetlayers{background,main,foreground}

    }
  }
  \caption{Seeking in hide and seek. The seekers are restricted to the
    green domain, while the hider is restricted to the red domain.}
    \label{fig:example_2}
\end{figure}

\begin{figure}[!h]
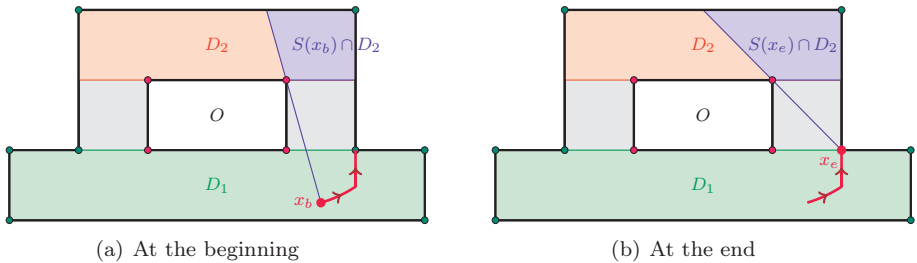

  \centering
  \subfigure[At the beginning]{
    \label{fig:example_2_before}
    \resizebox{0.47\linewidth}{!}{
      \centering
      \pgfdeclarelayer{background}
      \pgfdeclarelayer{foreground}
      \pgfsetlayers{background,main,foreground}
      %
    }%
  }
  %
  \subfigure[At the end]{
    \label{fig:example_2_after}
    \resizebox{0.47\linewidth}{!}{
      \centering
      \pgfdeclarelayer{background}
      \pgfdeclarelayer{foreground}
      \pgfsetlayers{background,main,foreground}
      %
    }%
  }
  \caption{Visibility metric at the beginning and the end of
  \NorcentTO~for a robot seeking other robots.}
  \label{fig:example_2_1}
\end{figure}


\section{Conclusions}\label{sec:conclusions}

We have studied the problem of identifying optimal hiding spots for
agents in obstacle-rich environments in the presence of adversaries
whose location is unknown, and developed methodologies to find them.
Our technical approach has first characterized the critical points of
the visibility polygon using anchors and inflection segments. Building
on this, we have relied on notions from non-smooth analysis, including
$\mu$-local Lipschitzness and $\mu$-directional derivatives, to
analyze the regularity properties of various visibility metrics,
including under limited range and limited field of view.  We have
proposed the \NorcentTO{} and established its almost sure asymptotic
convergence to local minimizers of the visibility metric despite its
non-smooth and non-convex nature.  Future work will leverage the
results for practical visibility-based metrics in robotic problems and
extend the treatment to multi-agent scenarios.

\newpage

\bibliographystyle{siamplain}
\bibliography{1091.bib}

\appendix

\section{Auxiliary result for proof of \texorpdfstring{\Cref{thm:non_stopping}}{Theorem~\ref{thm:non_stopping}}}

The following result states that, if the generalized gradient can be
expressed as a convex hull of finitely many gradients at a non-Clarke
stationary point, then there exists a cone of aperture less than $\pi$
which completely contains the generalized gradient.

\begin{lemma}[Generator of the generalized gradient]
  \label{lem:cone_lemma}
  Given a function $f \colon \genericPointSet \to \real$, let $ \gen{}
  \partial f(\genericPoint) := \big\{ \lim_{i \rightarrow \infty}
  \nabla f(\genericPoint_i) \mid \genericPoint_i \rightarrow
  \genericPoint, \genericPoint_i \not \in E \cup E_f \big\}, $ be the
  set of generators of $\partial f(\genericPoint)$, where $E$ is any
  zero measure set in the Lebesgue sense, $E_f$ is the set of points
  where $f$ fails to be differentiable, and $\nabla f$ is the gradient
  of $f$. If $|\gen{}\partial f (\genericPoint)| < \infty$ and $0
  \not\in \partial f (\genericPoint)$, then there exists a direction
  $u \in \real^2$ and an angle $\alpha \in (0, \sfrac{\pi}{2})$ such
  that $\partial f (\genericPoint) \subset
  \fovConeInfinite{2\alpha}{u}$.
  Moreover, $\partial f(\genericPoint) = \convexHull{\gen{} \partial
  f(\genericPoint)}$ holds trivially.
\end{lemma}


\end{document}